\documentclass[a4paper]{article}
\usepackage{amssymb}
\usepackage{amsthm}
\usepackage{latexsym}
\usepackage{amsmath}
\usepackage{color}

\usepackage{comment}
\excludecomment{discuss}

\newif\ifcol
\colfalse
\ifcol
\newcommand{\colorr}{\color[rgb]{0.8,0,0}}
\newcommand{\colorg}{\color[rgb]{0,0.5,0}}
\newcommand{\colorb}{\color[rgb]{0,0,0.8}}
\else
\newcommand{\colorr}{\color{black}}
\newcommand{\colorg}{\color{black}}
\newcommand{\colorb}{\color{black}}
\fi

\newtheorem{lemma}{Lemma}
\newtheorem{proposition}{Proposition}
\newtheorem{theorem}{Theorem}
\newtheorem{remark}{Remark}
\newtheorem{example}{Example}
\newtheorem{corollary}{Corollary}

\begin{document}
\title{
Quasi-Likelihood Analysis for Stochastic Regression Models with Nonsynchronous Observations
}
\author{T. Ogihara$^*$ and N. Yoshida$^{**}$\\
$*$ 
\begin{small}Center for the Study of Finance and Insurance, Osaka University\end{small}\\
\begin{small}Japan Science and Technology Agency, CREST\end{small}\\
\begin{small}1-3 Machikaneyama-cho, Toyonaka, Osaka 560-8531, Japan. \end{small}\\
\begin{small}e-mail: ogihara@sigmath.es.osaka-u.ac.jp\end{small}\\ 
$**$ 
\begin{small}Graduate School of Mathematical Sciences, University of Tokyo\end{small}\\
\begin{small}3-8-1 Komaba, Meguro-ku, Tokyo 153, Japan. \end{small}\\
\begin{small}e-mail: nakahiro@ms.u-tokyo.ac.jp\end{small}\\
}
\date{}
\maketitle
\ \\
{\bf Abstract.}
We consider nonsynchronous sampling of {\colorr parameterized stochastic regression models, which contain stochastic differential equations}. 
Constructing a quasi-likelihood function, we prove that 
the quasi-maximum likelihood estimator and the Bayes type estimator are consistent and asymptotically mixed normal 
when the sampling frequency of the nonsynchronous data becomes large. 

\section{Introduction}
Given a probability space $(\Omega, \mathcal{F}, P)$ with a right-continuous filtration $(\mathcal{F}_t)_{t\in [0,T]}$, 
we consider a {\colorb stochastic regression model specified by} the following equation :
\begin{eqnarray}
Y_t=Y_0+\int^t_0\mu_sds+\int^t_0b(X_s,\sigma)dW_s, \quad t \in [0,T], \nonumber
\end{eqnarray}
where {\colorb $\{Y_t\}_{0\leq t\leq T}=\{(Y^1_t,Y^2_t)\}_{0\leq t\leq T}$ is a $2$-dimensional $\mathcal{F}_t$-adapted process,} $\{W_t,\mathcal{F}_t\}_{0\leq t\leq T}$ is a $2$-dimensional standard Wiener process, 
$b=(b^{ij})_{1\leq i,j\leq 2}:\mathbb{R}^{n_2} \times \Lambda \to \mathbb{R}^2\otimes \mathbb{R}^2$ is a Borel function, 
$\{\mu_t\}$ and $\{X_t\}$ are $\mathcal{F}_t$-progressively measurable processes with values in $\mathbb{R}^{2}$ and $\mathbb{R}^{n_2}$, respectively, 
$\sigma \in \Lambda$, and $\Lambda$ is a bounded open subset of $\mathbb{R}^{n_1}$. 
{\colorb In the case that $\mu_t=\mu(t,Y_t)$ and $X_t=(t,Y_t)$, then $\{Y_t\}$ is a time-inhomogeneous diffusion process. }

Our purpose is to estimate the true value $\sigma_{\ast}$ of parameter $\sigma \in \Lambda$ 
by nonsynchronous observations $\{Y^1_{S^i}\}_i$, $\{Y^2_{T^j}\}_j$ and $\{X^k_{\mathcal{T}^j_k}\}_{j,k}$. 
In our setting, $\{\mu_t\}_{0\leq t\leq T}$ is completely unobservable and unknown.

The problem of nonsynchronous observations appears in the analysis of high-frequency financial data.
Recently, as availability of intraday security prices gets increase, 
the analysis of high-frequency data becomes more significant.
In particular, the realized volatility has been studied actively as an estimator of security returns' volatility.

In the study of portfolio risk management of financial assets, 
the quadratic covariation of two security log-prices is also a significant risk measure.
Therefore estimation of quadratic covariation with high-frequency data {\colorr has} also been studied by many authors. 
{\colorb One problem} of estimation is nonsynchronous trading.
The observation times of two different security prices {\colorb do} not necessarily coincide.

{\colorb
If $Y^1$ and $Y^2$ are synchronously observed at some stopping times $\{S^i\}$, 
then the realized covariance between $Y^1$ and $Y^2$ converges to $\langle Y^1,Y^2\rangle_T$ in probability as $\max_i|S^i-S^{i-1}|\to^p 0$.
}
When observation times {\colorb of $Y^1$ and $Y^2$} are nonsynchronous, to calculate the realized covariance, 
we need to synchronize the data by some method.
However, the realized covariance has serious bias 
if we use a simple synchronizing method such as {\it previous-tick} interpolation or linear interpolation.
Epps \cite{epp} first indicated this phenomenon by U.S. stock data analysis, 
and this phenomenon is called the {\it Epps Effect}.

{\colorr To solve this problem, Malliavin and Mancino \cite{mal-man} proposed a {\colorb Fourier analytic method},
and Hayashi and Yoshida \cite{hay-yos01}} {\colorb proposed a estimator based on overlapping of observation intervals.}
In sequent papers, Hayashi and Yoshida \cite{hay-yos02},\cite{hay-yos03} studied 
the asymptotic distribution of estimation error of {\colorb their} estimator
and proved asymptotic mixed normality.
{\colorr
There also exist some works about estimation of the quadratic covariation with nonsynchronous data contaminated 
by market microstructure noise. We refer the reader to Barndorff-Nielsen et al. \cite{bar-etc} for a kernel based method,
Christensen, Kinnebrock and Podolskij \cite{chr-etc} for a pre-averaged Hayashi-Yoshida estimator, 
A${\rm \ddot{\i}}$t-Sahalia, Fan and Xiu \cite{ait-fan} for a {\colorb method with the maximum likelihood estimator
of a model with deterministic diffusion coefficients}, 
Bibinger \cite{bib},\cite{bib2} for a multiscale estimator.
}

With respect to the problem of nonsynchronous observations, {\colorb essentially nonparametric approach} {\colorr has} been studied.
In this work, we use a quasi-likelihood function{\colorb , that approximates} the likelihood function {\colorb in diffusion cases}
and construct a quasi-maximum likelihood estimator and a Bayes type estimator 
for a parametric stochastic regression model with nonsynchronous observations{\colorb . T}he asymptotic behavior of estimators 
{\colorb will be investigated} when time end $T$ is fixed and $\max_{i,j}|S^i-S^{i-1}|\vee |T^j-T^{j-1}| \to 0$ in probability.
{\colorr
Hence our method can be applied not only to estimating the quadratic covariation but also to identifying {\colorb nonlinear} structure of process $Y$.
}

{\colorr
There exist many studies about asymptotic theory of parametric estimation for stochastic differential equations with high-frequency data.
Among many studies in a long history, we refer the reader to Plakasa Rao \cite{pra},\cite{pra2}, Yoshida \cite{yos92},\cite{yos06},\cite{yos05}, 
Kessler \cite{kes} under ergodicity, Shimizu and Yoshida \cite{shi-yos}, Ogihara and Yoshida \cite{ogi-yos} for jump diffusion processes, 
Masuda \cite{mas} for Ornstein-Uhlenbeck processes driven by heavy-tailed symmetric L${\rm \acute{e}}$vy processes,
S$\o$rensen and Uchida \cite{sor-uch}, Uchida \cite{uch03},\cite{uch04} for perturbed diffusions, Dohnal \cite{doh}, Genon-Catalot and Jacod \cite{gen-jac93},\cite{gen-jac94},
Gobet \cite{gob}, Uchida and Yoshida \cite{uch-yos} for the fixed interval case.
}

One of the most useful approaches to study asymptotic behaviors of quasi-maximum likelihood estimators and Bayes type estimators
is the theory of random field of likelihood ratios initiated by Ibragimov and Has'minskii \cite{ibr-has01}-\cite{ibr-has03}.
Their theory enabled to reduce the problem of asymptotic behaviors 
of estimators to more tractable properties of the random field of likelihood ratios.
{\colorb In \cite{ibr-has03}, they applied their theory to  
independent observations and white Gaussian noise models.}
Kutoyants \cite{kut01}, \cite{kut02} developed Ibragimov-Has'minskii's theory for diffusion processes and point processes.
Yoshida \cite{yos05} investigated polynomial type large deviation inequalities to apply Ibragimov-Has'minskii's theory
{\colorb and discussed} consistency and asymptotic normality of quasi-maximum likelihood estimators and Bayes type estimators
for ergodic diffusion processes.
{\colorb This scheme was also} {\colorr applied to jump diffusion processes in Ogihara and Yoshida \cite{ogi-yos}, Ornstein-Ohlenbeck processes driven 
by heavy-tailed symmetric L${\rm \acute{e}}$vy processes in Masuda \cite{mas}, 
and diffusion processes in the fixed interval in Uchida and Yoshida \cite{uch-yos}.}

In this work, we construct a quasi-log-likelihood function for the stochastic regression model with nonsynchronous observations.
Then we will show consistency, asymptotic mixed normality and the convergence of {\colorb moments of} the quasi-maximum likelihood estimator
and the Bayes type estimator with aid of polynomial type large deviation inequalities.
The advantage of our approach is to obtain asymptotic mixed normality, exact representation of asymptotic variance 
and convergence of {\colorb moments of the estimators}.
The convergence of moments of {\colorb the estimators is important, e.g.,} when we investigate the asymptotic expansion and the theory of information criteria.
Moreover, {\colorb our method does not require any synchronization methods.} 

{\colorr
In the case of synchronous and equi-space samplings : $S^i=T^i=iT/n$, Gobet \cite{gob} showed local asymptotic normality of the likelihood function of observations
and obtained the asymptotic lower bound for the variance of estimators. The estimator proposed by Genon-Catalot and Jacod \cite{gen-jac93} attains this bound.
In the case of nonsynchronous observations, we expect that local asymptotic mixed normality of the likelihood function holds and our estimators attain the lower bound 
of the variance of estimators since our quasi-likelihood function seems to be asymptotically equivalent with the 'true' likelihood function and our quasi-likelihood ratio
has a LAMN type limit distribution. However, these {\colorb problems} are not proved in this paper and {\colorb left} as future work.
}

This paper is organized as follows.
In Section $2$, we construct a quasi-log-likelihood function $H_n$ and discuss its non-degeneracy.
Section $3$ {\colorb gives} the asymptotic behavior of $H_n$. Section $3.1$ deals with two equivalent conditions 
of the asymptotic behavior of observation times $\{S^i\}$ and $\{T^j\}$ to control the asymptotic behavior of $H_n$.
In Section $3.2$, we specify the limit of $H_n$ and estimate the rate of convergence.
Section $4$ studies the degree of separation of the limit of $H_n$, 
which is {\colorb necessary} to prove asymptotic properties of the quasi-maximum likelihood estimator and the Bayes type estimator. 
We also introduce {\colorb sufficient conditions} for the condition of separation.
In Section $5$, our main results about asymptotic properties for estimators are stated.
Section $6$ introduces {\colorb easily} tractable sufficient conditions for {\colorb assumptions about the observation times in} main theorems.
Proofs are collected in Section $7$.

\section{Construction of a quasi-likelihood}
In this section, we define a quasi-log-likelihood function $H_n$ 
to construct a quasi-maximum likelihood estimator and a Bayes type estimator.

First, we define some notations. 
For a real number $a$, $[a]$ denote the maximum integer which is not greater than $a$.
For a matrix $A$, $A^{\star}$ denotes transpose of $A$ and $\parallel A \parallel$ represents the norm of $A$ as a linear map. 
We regard a $p$-dimensional vector $v$ as a $p\times 1$ matrix.
$\mathcal{E}_p$ denotes unit matrix of size $p$.
We set $\sup \emptyset = -\infty$, $\inf \emptyset = +\infty$ and $2\mathbb{N}=\{2k;k\in\mathbb{N}\}$.
For $M\in\mathbb{N}$ and $K\subset \mathbb{R}^M$, $\bar{K}$ denotes the closure of $K$.
For a set $K\subset \Omega$, $K^c$ denotes the complementary set of $K$.
For an interval $K\subset [0,T]$ and a stochastic process $\{{\colorg {\sf X}}_t\}_{0\leq t\leq T}$, 
we denote $L(K)=\inf K,R(K)=\sup K$, ${\colorg {\sf X}}(K)={\colorg {\sf X}}_{R(K)}-{\colorg {\sf X}}_{L(K)}$, $K_t=K\cap [0,t)$ and $|K|=R(K)-L(K)$. 
Let $b^i(x,\sigma)=(b^{i1}(x,\sigma),b^{i2}(x,\sigma))^{\star} \ (i=1,2)$.
For a vector $\kappa=(\kappa_1,\cdots, \kappa_M)$, 
we denote $\partial_{\kappa}^k=(\frac{\partial^k}{\partial\kappa_{i_1}\cdots \partial\kappa_{i_k}})_{i_1,\cdots i_k=1}^M$.
We denote $|x|^2=\sum_{i_1,\cdots i_M}|x_{i_1,\cdots i_M}|^2$ for $x=\{x_{i_1,\cdots i_M}\}_{i_1,\cdots i_M}$.

Let $\Lambda$ satisfy Sobolev's inequality, that is, for any $p>n_1$, there exists $C>0$ such that
\begin{equation*}
\sup_{x\in \Lambda}|u(x)|\leq C\sum_{k=0,1}\parallel \partial_x^ku(x)\parallel_p
\end{equation*}
for $u\in C^1(\lambda)$. It is the case if $\Lambda$ has Lipschitz boundary.
See Adams \cite{adams}, Adams and Fournier \cite{adams-fou} for more details.

We recall the definition of stable convergence. {\colorg Given an extension $(\tilde{\Omega},\tilde{\mathcal{F}},\tilde{P})$ of $(\Omega, \mathcal{F},P)$,
let $\{{\colorg {\sf X}}_n\}_{n\in\mathbb{N}}$ and ${\colorg {\sf X}}$ be random variables on $(\tilde{\Omega},\tilde{\mathcal{F}},\tilde{P})$
with values in a metric space $E$.}
Then we say that ${\colorg {\sf X}}_n$ stably converges in law to ${\colorg {\sf X}}$, and write ${\colorg {\sf X}}_n\to^{s\mathchar`-\mathcal{L}}{\colorg {\sf X}}$, if
$E[{\colorg {\sf Y}}f({\colorg {\sf X}}_n)]\to E[{\colorg {\sf Y}}f({\colorg {\sf X}})]$ as $n\to \infty$ for any bounded continuous function $f:E\to \mathbb{R}$
and any bounded variable ${\colorg {\sf Y}}$ on $(\Omega,\mathcal{F})$. See Jacod \cite{jacod} for more details.

For $1\leq k \leq n_2$, let observations $\{S^i\}_i$, $\{T^j\}_j$ and $\{\mathcal{T}^j_k\}_{j,k}$ are strictly increasing with respect to $i$ or $j$ almost surely
and satisfy $S^0=T^0=\mathcal{T}^0_k=0$, $S^i=\inf\{t\geq 0; N^1_t\geq i\}\wedge T$,
$T^j=\inf\{t\geq 0; N^2_t\geq j\}\wedge T$, and $\mathcal{T}^j_k=\inf\{t\geq 0; N^{k+2}_t\geq j\}\wedge T$ for $i,j\geq 1$, 
where $\{N^1_t\}_t$, $\{N^2_t\}_t$ and $\{N^{k+2}_t\}_t$ {\colorg are simple point processes}, 
that is, $\{N^k_t\}$ is a ${\rm c\grave{a}dl\grave{a}g}$ $\mathbb{Z}_+$-valued stochastic process whose jumps are equal to 1 and $N^k_0=0$ $(1\leq k\leq n_2+2)$.
{\colorg These observations and point processes depend on a positive integer $n\in\mathbb{N}$. }
Let $\Pi=\Pi_n=((S^i)_i,(T^j)_j,(\mathcal{T}^j_k)_{j,k})$, $l_n=N^1_{T-}+1, m_n=N^2_{T-}+1,m^k_n=N^{k+2}_{T-}+1$ for $1\leq k\leq n_2$, then $l_n,m_n,\{m_n^k\}_{k=1}^{n_2}$ are observation counts.
We also assume $\{\Pi_n\}_{n\in\mathbb{N}}$ are independent of $\mathcal{F}_T$.
{\colorg Denote} $I^i=[S^{i-1},S^i) \ (1\leq i\leq l_n)$, $J^j=[T^{j-1},T^j) \ (1\leq j\leq m_n)$, 
\begin{equation*}
r_n=\max_{i,j}(|I^i|\vee |J^j|)\vee \max_{1\leq k\leq n_2}\max_{1\leq j\leq m^k_n}|\mathcal{T}^{j}_k-\mathcal{T}^{j-1}_k|,
\end{equation*}
and $\mathcal{T}'_k(K)=\max\{\mathcal{T}^j_k;j\in\mathbb{Z}_+,\mathcal{T}^j_k\leq L(K)\}$ for $1\leq k\leq n_2$ and an interval $K\subset [0,T]$.
Let $\{b_n\}_{n\in\mathbb{N}}$ be a sequence of positive numbers such that $b_n\geq 1 \ (n\in\mathbb{N})$ and $b_n\to \infty$ as $n\to \infty$.
$\{b_n\}$ represents order of observation counts. Conditions for $\{b_n\}$ are given in $[A2\mathchar`-q,\delta]$, $[A3'\mathchar`-q,\eta]$, $[A4\mathchar`-q,\delta]$ later.

For a function $g:\mathbb{R}^{n_2}\times \Lambda\to \mathbb{R}$, let $g_t=g(X_t,\sigma),g_{t,\ast}=g(X_t,\sigma_{\ast})$,
$g_{K,t}=g(\{X^k_{\mathcal{T}'_k(K)\wedge t}\}_k,\sigma)$ and $g_K=g_{K,T}$ for interval $K\subset [0,T]$.
We use the symbol $C$ for a generic positive constant which {\colorg is independent of $n$ and $p$,} and is varying from line to line
without specially stated.

We assume the following conditions.
\begin{description}
\item{[$A1$]} 
\begin{enumerate}
\item The mapping $b: \mathbb{R}^{n_2} \times \Lambda \to \mathbb{R}^2\otimes \mathbb{R}^2$ 
has the continuous derivative $\partial_x^j\partial_{\sigma}^ib$ and $\partial_{\sigma}^ib$ can be continuously extended to $\mathbb{R}^{n_1}\times \bar{\Lambda}$
for $0\leq j\leq 3$ and $0\leq i\leq 4$. 
Moreover, 
\begin{equation*}
\sup_{\sigma \in \Lambda} |\partial_x^j\partial^i_{\sigma}b(x,\sigma)|\leq C(1+|x|)^C
\end{equation*}
for  $0\leq j \leq 3$, $0 \leq i \leq 4$ and $x\in\mathbb{R}^{n_2}$.
\item $bb^{\star}(x,\sigma)$ is elliptic uniformly in $(x,\sigma)\in \mathbb{R}^{n_2}\times \Lambda$, 
that is, there exists $\epsilon>0$ such that $\det bb^{\star}(x,\sigma)\geq \epsilon$ for $(x,\sigma)\in\mathbb{R}^{n_2}\times \Lambda$.
\item $|b(x,\sigma)-b(y,\sigma)| \leq C|x-y|$
for $x,y \in \mathbb{R}^{n_2}$ and $\sigma\in\Lambda$.
\item $Y_0\in \cap_{q>0}L^q(\Omega)$.
\item There exists $\gamma\in (0,1)$ such that 
\begin{equation*}
\sup_{0\leq t\leq T}E[|\mu_t|^q]<\infty \quad {\rm and} \quad  \sup_{0\leq s<t\leq T}\frac{E[|\mu_t-\mu_s|^q]}{|t-s|^{q\gamma}} <\infty
\end{equation*}
for any $q>0$.
\item There exists $n_3\in\mathbb{Z}_+$ such that $\{X_t\}$ can be decomposed as 
\begin{equation*}
X_t=X_0+\int^t_0\tilde{b}^1_sds+\int^t_0\tilde{b}^2_sdW_s+\int^t_0\tilde{b}^3_sd\hat{W}_s,
\end{equation*}
where 
\begin{equation*}
\tilde{b}^i_t=\tilde{b}^i_0+\int^t_0\hat{b}^{i1}_sds+\int^t_0\hat{b}^{i2}_sdW_s+\int^t_0\hat{b}^{i3}_sd\hat{W}_s, \ (i=2,3)
\end{equation*}
$\{\tilde{b}^i_t\}_{0\leq t\leq T} \ (1\leq i\leq 3)$ and $\{\hat{b}^{ij}_t\}_{0\leq t\leq T} \ (2\leq i\leq 3,1\leq j\leq 3)$ are $\mathcal{F}_t$-progressively measurable processes,
$\{\hat{W}_t,\mathcal{F}_t\}_{0\leq t\leq T}$ is an $n_3$-dimensional standard Wiener process independent of $\{W_t\}$ and
$E[\sup_{0\leq t\leq T}(|\hat{b}^{ij}_t|\vee |\tilde{b}^i_t|\vee |X_0|)^p]<\infty$ for any $i,j$ and $p>0$.
We ignore the terms $\tilde{b}^3_t$, $\int^t_0\tilde{b}^3_sd\hat{W}_s$ and $\int^t_0\hat{b}^{i3}_sd\hat{W}_s$ when $n_3=0$.
\end{enumerate}
\end{description}

Our setting contains the case where $\{X_t\}$ or $Y_0$ depends on $\sigma$ and main results hold in this case. 
However, if $\{X_t\}$ or $Y_0$ depends on $\sigma$, our estimator $\hat{\sigma}_n$ may not be the quasi-maximum likelihood estimator
since we need to consider the density of observations $\{X^k_{\mathcal{T}^j_k}\}$ or $Y_0$.
Nevertheless, we use the terms "quasi-maximum likelihood estimator" and "Bayes type estimator" in this case.
If $X_t=(t,Y_t)$ and $Y_0$ does not depend on $\sigma$, we can see $\hat{\sigma}_n$ is the maximum likelihood type estimator.

Under $[A1]$ $2$, there exists $\epsilon>0$ such that 
\begin{equation}\label{detB-nondeg}
\det (b_tb_t^{\star}) = | b^1_t | ^2| b^2_t |^2 (1-\rho^2_t)\geq \epsilon
\end{equation}
for $\rho(x,\sigma)=b^1\cdot b^2|b^1|^{-1}|b^2|^{-1}(x,\sigma)$.
Therefore 
\begin{equation}\label{rho-sup}
\bar{\rho}=\sup_{t\in [0,T],\sigma\in \Lambda}|\rho_t| < 1 \quad a.s.
\end{equation}
by $[A1] \ 1$.

Let us denote
\begin{eqnarray}
S(\sigma)=\left( 
\begin{array}{cc}
{\rm diag}(\{| b^1_I |^2\}_I) & \left\{ b^1_I\cdot b^2_J\frac{|I\cap J|}{\sqrt{|I|}\sqrt{|J|}}\right\}_{IJ} \\
\left\{ b^1_I\cdot b^2_J\frac{|I\cap J|}{\sqrt{|I|}\sqrt{|J|}}\right\}_{JI} & {\rm diag}(\{| b^2_J |^2\}_J)\\
\end{array} 
\right) \nonumber
\end{eqnarray}
and define a quasi-log-likelihood function $H_n=H_n(\sigma)$ of $((Y^1(I)/\sqrt{|I|})^{\star}_I,(Y^2(J)/\sqrt{|J|})^{\star}_J)$ 
{\colorg by} 
\begin{eqnarray}
H_n&=&-\frac{1}{2}\bigg(\bigg(\frac{Y^1(I)}{\sqrt{|I|}}\bigg)^{\star}_I,\bigg(\frac{Y^2(J)}{\sqrt{|J|}}\bigg)^{\star}_J\bigg)S^{-1}\bigg(\bigg(\frac{Y^1(I)}{\sqrt{|I|}}\bigg)^{\star}_I,\bigg(\frac{Y^2(J)}{\sqrt{|J|}}\bigg)^{\star}_J\bigg)^{\star} -\frac{1}{2}\log \det S \nonumber 
\end{eqnarray}
when $\det S >0$.
If $X_t\equiv t$ and $\mu\equiv 0$, $S$ is the covariance matrix for the Euler-Maruyama type approximation $((\tilde{Y}^1(I))_I, (\tilde{Y}^2(J))_J)$ of $((Y^1(I))_I,(Y^2(J))_J)$
defined by 
$\tilde{Y}^1(I)=b^1(L(I))\cdot W(I)$, $Y^2(J)=b^2(L(J))\cdot W(J)$.
Though $H_n$ is the quasi-log-likelihood function for $\mu\equiv 0$, 
we can see that the effect of drift term $\mu$ in a quasi-likelihood function can be ignored asymptotically.
So $H_n$ is applicable for general cases. 

In the case of synchronous observations, we have uniform non-degeneracy of $S$ 
by the condition $[A1] \ 2$.
However, in the case of nonsynchronous observations, the problem becomes more complicated
since the observation times of diffusion coefficients are not the same for $Y^1$ and $Y^2$.
However, the following proposition ensures that $H_n$ is well-defined under $[A1] \ 2.$
\begin{proposition}\label{S-nondeg}
Assume $[A1] \ 2$. Then $\det S(\sigma) >0$ almost surely for any $\sigma\in\Lambda$.
\end{proposition}
\begin{proof}
Fix $\omega\in \Omega$. 
It is sufficient to show that $S$ is positive definite.
Let $((u_I)_I,(v_J)_J)$ be a real vector satisfying
\begin{eqnarray}
((u_I)_I,(v_J)_J)S((u_I)_I,(v_J)_J)^{\star}=0. \nonumber
\end{eqnarray}
We assume that $((u_I)_I,(v_J)_J)$ has a non-zero element and lead to a contradiction.

Let $\{\tilde{W}^1_t\}$ and $\{\tilde{W}^2_t\}$ be two independent Wiener processes on some probability space, and $\{M_t\}_{0\leq t\leq T}$ is a stochastic process on the same probability space, satisfying
\begin{equation*}
M_t=\sum_I\frac{u_I}{\sqrt{|I|}}b^1_I\cdot \tilde{W}(I_t)+\sum_J\frac{v_J}{\sqrt{|J|}}b^2_J\cdot \tilde{W}(J_t).
\end{equation*}
Then $\{M_t\}$ is a martingale satisfying
\begin{eqnarray}
\langle M\rangle_t&=&\sum_I\frac{u_I^2}{|I|}| b^1_I|^2|I_t|+\sum_J\frac{v_J^2}{|J|}| b^2_J|^2|J_t|+2\sum_{I,J}u_Iv_Jb^1_I\cdot b^2_J\frac{|(I\cap J)_t|}{\sqrt{|I||J|}}. \nonumber 
\end{eqnarray}
Since
\begin{eqnarray}
\langle M\rangle_T=((u_I)_I,(v_J)_J)S((u_I)_I,(v_J)_J)^{\star}=0, \nonumber
\end{eqnarray}
it follows that $\langle M\rangle_t=0$ for $0\leq t \leq T$.

We may assume some $I$ satisfies $L(I)=\min\{L(I);u_I\neq 0\} \wedge \min\{L(J);v_J\neq 0\}$ 
without loss of generality. We fix this $I$ below.

First, we consider the case that $L(I)<\min\{L(J);v_J\neq 0\}$. 
Then
\begin{equation*}
\langle M\rangle_{L(I)+\delta}=| b^1_I | ^2u_I^2\delta/|I|=0
\end{equation*}
for sufficiently small $\delta>0$.
Therefore we have $| b^1_I |=0$, which contradicts $[A1] \ 2.$

In the case that $L(I)=L(J)$ for some $J$ with $v_J\neq 0$, we obtain
\begin{eqnarray}\label{M-quad}
\langle M\rangle_{L(I)+\delta}=| b^1_I| ^2\frac{u_I^2}{|I|}\delta+| b^2_J|^2\frac{v_J^2}{|J|}\delta+2\frac{u_Iv_J}{\sqrt{|I||J|}}b^{1}_I\cdot b^{2}_J\delta=0
\end{eqnarray}
for sufficiently small $\delta>0$.
Since $L(I)=L(J)$, we obtain $b^2_J=b^2_I$. Therefore 
\begin{equation*}
\bigg(\frac{u_I}{\sqrt{|I|}},\frac{v_J}{\sqrt{|J|}}\bigg)b_{I}b^{\star}_{I}\bigg(\frac{u_I}{\sqrt{|I|}},\frac{v_J}{\sqrt{|J|}}\bigg)^{\star}=0
\end{equation*}
by (\ref{M-quad}). This contradicts the fact that $b_Ib^{\star}_I$ is positive definite by $[A1] \ 2.$
\end{proof}

Let
\begin{eqnarray}
&&\bar{\rho}_t=\sup_{0\leq s\leq t}|\rho_s|, \quad \rho_{I,J,t}=\frac{b^1_{I,t}\cdot b^2_{J,t}}{| b^1_{I,t}|| b^2_{J,t} |}, \quad 
\tilde{\rho}_n(t)=\sup_{\sigma, I,J;I\cap J \neq \emptyset}|\rho_{I,J,t}|\vee \bar{\rho}_t, \nonumber 
\end{eqnarray}
To discuss asymptotic behavior of the quasi-likelihood, 
we need a more precise estimate for non-degeneracy of $S$. 
To this end, we will estimate $1-\tilde{\rho}_n(t)$ from below.
Assuming $[A1]$ and $r_n\to^p 0 \ (n\to \infty)$, 
we have $\sup_t|\bar{\rho}_t - \tilde{\rho}_n(t)|\to^p 0\ (n\to \infty)$ 
by uniform continuity of $b^1$ and $b^2$ with respect to  $t$ and $\sigma$ for fixed $\omega$.
Therefore $\lim_{n\to \infty}P[\sup_t\tilde{\rho}_n(t)\geq 1]=0$ by (\ref{rho-sup}).

We need a more strong condition for $\tilde{\rho}_n$. For stochastic processes $\{s_n(t)\}_{0\leq t\leq T,n\in\mathbb{N}}$, we consider the following condition:
\begin{description}
\item{[$S$]} 
There exists $M\in\mathbb{N}$, stochastic processes $\{\bar{s}_n(t,x)\}$ and a $\sigma(\Pi_n)$-measurable $\mathbb{R}^M$-valued
random variable ${\colorg {\sf X}}$ such that $s_n(t)=\bar{s}_n(t;{\colorg {\sf X}})$, $\bar{s}_n(t,x)$ is continuous with respect to $(t,x)$ a.s.,
$\bar{s}_n(0,x)\leq 1-|\rho_0|$, $t\mapsto \bar{s}_n(t,x)$ is non-increasing 
and $\{\bar{s}_n(t,x)\}_{0\leq t\leq T}$ is the $[0,1]$-valued $\mathcal{F}_t$-adapted process for $n\in\mathbb{N}$ and $x\in\mathbb{R}^M$.
\end{description}
Let 
\begin{equation*}
\tau_n=\tau(s_n)=\inf\{t\in [0,T];\tilde{\rho}_n(t)\geq 1-s_n(t)\}\wedge T.
\end{equation*} 
We consider the following condition for $q>0$ and $\xi>0$.
\begin{description}
\item{[$S\mathchar`-q,\xi$]} 
$\{s_n\}$ satisfies $[S]$, $P[\tau(s_n)<T]=O(b_n^{-\xi})$ and $\sup_nE[(s_n(T))^{-q}]<\infty$. 
\end{description}

Define $\hat{S}=\hat{S}(\sigma;s_n)$ and $\hat{H}_n=\hat{H}_n(\sigma;s_n)$ similarly to $S$ and $H_n$ respectively, substituting $b^1_{I,\tau_n}$ for $b^1_I$ and $b^2_{J,\tau_n}$ for $b^2_J$ in the definition of $S$.
Under $[S\mathchar`-q,\xi]$, it is easy to see that 
$\sup_{\sigma}|H_n-\hat{H}_n|\to^p 0$ as $n \to \infty$.
To investigate asymptotic properties of estimators, 
it is convenient to use $\hat{H}_n$.

If $\sup_t\tilde{\rho}_n(t)< c$ almost surely for some $0<c<1$, we can set $s_n\equiv 1-c$.
However, in general, we need the following conditions to obtain $\{s_n\}$ for $q>0$ and $\delta\in (0,1)$.
\begin{description}
\item{[$A2$]} $r_n\to^p 0$ as $n\to \infty$.
\end{description}
\begin{description}
\item{[$A2\mathchar`-q,\delta$]} $E[r_n^q]=O(b_n^{-\delta q})$. 
\end{description}
The following lemma gives examples of $\{s_n\}$ in general cases.
\begin{lemma}\label{Sq-est}
Let $P>1, q'>0$, $0<\delta<1$ and $s_n(t)=(1-\bar{\rho}_t)/P$. Assume $[A1],[A2\mathchar`-q',\delta]$. 
Then $[S\mathchar`-q,\xi]$ holds for any $q>0$ and $0<\xi<\delta q'$.
\end{lemma}
\begin{proof}
It is clear that $\{s_n\}$ are continuous $[0,1]$-valued $\mathcal{F}_t$-adapted processes 
and the mapping $t\mapsto s_n(t)$ is non-increasing for $n\in\mathbb{N}$.
Since $1/(1-|\rho_t|)\leq 2|b^1_t|^2|b^2_t|^2/\epsilon$ for any $q>0$ by (\ref{detB-nondeg}),
we obtain $\sup_nE[(s_n(T))^{-q}]<\infty$ by $[A1]$.

Moreover, let $\eta=(\delta-\xi/q')/9$ and $q_0\geq \xi/\eta$, then by $[A1]$ and the mean value theorem, we obtain
\begin{equation*}
\sup_{0\leq t\leq T}\bigg|\tilde{\rho}_n(t)-\bar{\rho}_t\bigg|\leq C\sup_t(1+|X_t|)^C\times \sup_{|t-s|\leq r_n}|X_t-X_s|.
\end{equation*}
For $t\in (0,T)$, we have 
\begin{eqnarray}
1-\tilde{\rho}_n(t)\leq s_n(t) &\Rightarrow & \frac{1-\tilde{\rho}_n(t)}{1-\rho_t}\leq \frac{1}{P} \Rightarrow 1-\frac{1}{P}\leq \frac{\tilde{\rho}_n(t)-\bar{\rho}_t}{1-\bar{\rho}_t} \nonumber \\
&\Rightarrow & 1-\bar{\rho}_t\leq b_n^{-\eta}\quad {\rm or} \quad b_n^{-\eta}(1-1/P)\leq (\tilde{\rho}_n(t)-\bar{\rho}_t). \nonumber
\end{eqnarray}
From this relation, we see that
\begin{eqnarray}
P[\tau_n<T]&=&P[{\rm There \ exists} \ t\in [0,T) \ {\rm s.t.} \ 1-\tilde{\rho}_n(t)\leq s_n(t)] \nonumber \\
&\leq &b_n^{-q_0\eta}E[1/(1-\bar{\rho}_T)^{q_0}]+P[b_n^{-\eta}(1-1/P) \leq b_n^{2\eta} r_n^{1/3}] \nonumber \\
&&+P\bigg[C\sup_t(1+|X_t|)^C\vee \sup_{s\neq t}\frac{|X_t-X_s|}{|t-s|^{1/3}}\geq b_n^{\eta}\bigg]. \nonumber
\end{eqnarray}
Then by $[A1]$, $[A2\mathchar`-q',\delta]$ and Kolmogorov criterion(\cite{Rev-Yor} Chapter I, Theorem (2.1)), we obtain
\begin{eqnarray}
P[\tau_n<T]&\leq & E[(b_n^{3\eta}(1-1/P)^{-1}r_n^{1/3})^{3q'}]+O(b_n^{-\xi})= O(b_n^{-\xi}). \nonumber 
\end{eqnarray}
\end{proof}
From now on, we fix $\{s_n\}$ which satisfy $[S\mathchar`-q,\xi]$ for some $q>0$ and $\xi>0$ unless otherwise indicated.

Next, we expand $\hat{H}_n$. We denote
\begin{eqnarray}
D&=&{\rm diag}(\{| b^1_{I,\tau_n}|\}_I,\{| b^2_{J,\tau_n}|\}_J), \quad L=\bigg\{\rho_{I,J,\tau_n}\frac{|I\cap J|}{\sqrt{|I||J|}}\bigg\}_{I,J}, \nonumber \\
\tilde{L}&=&\left(
\begin{array}{ll}
0 & L \\
L^{\star} & 0 \\
\end{array}
\right), \quad Z=\bigg(\bigg(\frac{Y^1(I)}{| b^1_{I,\tau_n}| \sqrt{|I|}}\bigg)_I^{\star}, \bigg(\frac{Y^2(J)}{| b^2_{J,\tau_n}| \sqrt{|J|}}\bigg)_J^{\star}\bigg)^{\star}. \nonumber
\end{eqnarray}
Since $\hat{S}=D(\mathcal{E}_{l_n+m_n}+\tilde{L})D$,   
\begin{equation*}
\hat{H}_n=-\frac{1}{2}Z^{\star}MZ-\log\det D+\frac{1}{2}\log\det M
\end{equation*}
for $M=(\mathcal{E}_{l_n+m_n}+\tilde{L})^{-1}$.
Moreover, for $G=\{ |I\cap J|/\sqrt{|I||J|} \}_{IJ}$, we obtain
\begin{equation*}
\parallel \tilde{L} \parallel^2= \parallel \{\rho_{I,J,\tau_n}G_{IJ}\}_{IJ}\parallel^2 \vee \parallel \{\rho_{I,J,\tau_n}G_{IJ}\}_{JI} \parallel^2 \leq  (1-s_n(T))^2(\parallel G\parallel ^2 \vee \parallel G^{\star}\parallel^2).
\end{equation*}
\begin{lemma}\label{GGT-lambda}
For any $n\in \mathbb{N}$ and deterministic partitions $\{I\},\{J\}$ of $[0,T]$, all the eigenvalues of the symmetric matrices $GG^{\star},G^{\star}G$ are in $[0,1]$. 
In particular, $\parallel G\parallel \vee \parallel G^{\star}\parallel \leq 1$.
\end{lemma}
\begin{proof}
We denote by $\{\lambda_i\}_{i=1}^{l_n}$ the eigenvalues of $GG^{\star}$. 
Obviously, $0 \leq \lambda_i\ (1 \leq i \leq l)$.
Let $\{\tilde{W}_t\}_t$ be a one-dimensional standard Wiener process and  
\begin{eqnarray}
\Sigma_1=\left( 
\begin{array}{cc}
\mathcal{E}_{l_n} & G \\
G^{\star} & \mathcal{E}_{m_n} \\
\end{array} 
\right),\quad \Sigma_2=\left( 
\begin{array}{cc}
\mathcal{E}_{l_n} & -G \\
0 & \mathcal{E}_{m_n} \\
\end{array} 
\right), \nonumber
\end{eqnarray}
then $\Sigma_1$ is the covariance matrix of $((\tilde{W}(I)/\sqrt{|I|})_I,((\tilde{W}(J)/\sqrt{|J|})_J)$ and 
\begin{eqnarray}
\Sigma_2\Sigma_1 \Sigma_2^{\star}=\left( 
\begin{array}{cc}
\mathcal{E}_{l_n}-GG^{\star} & 0 \\
0 & \mathcal{E}_{m_n} \\
\end{array} 
\right). \nonumber
\end{eqnarray}
Since $\Sigma_1$ is non-negative definite, $\mathcal{E}_{l_n}-GG^{\star}$ is also non-negative definite,
and hence $1-\lambda_i \geq 0 \ (1 \leq i \leq l)$.
Therefore we conclude $0\leq \lambda_i\leq 1 \ (1 \leq i \leq l)$.

In particular, we have $\parallel G\parallel^2=\parallel G^{\star}G\parallel=\sup_i\lambda_i\leq 1$. 
The same conclusion can be drawn for $G^{\star}G$ and $\parallel G^{\star}\parallel$.
\end{proof}

Since $\parallel \tilde{L} \parallel \leq 1-s_n(T)$ by Lemma \ref{GGT-lambda}, 
$\sum_{p=0}^{\infty}(-1)^p\tilde{L}^p$ exists almost surely and this gives $M=\sum_{p=0}^{\infty}(-1)^p\tilde{L}^p$, 
under $[S\mathchar`-q,\xi]$. 
Moreover, we obtain
\begin{equation*}
\max_{1\leq k \leq l+m}|\eta_k|=\parallel \tilde{L}\parallel \leq 1-s_n(T),
\end{equation*}
where $\{\eta_k\}_{k=1}^{l+m}$ be the eigenvalues of $\tilde{L}$.
Hence
\begin{equation*}
\log \det (\mathcal{E}_{l_n+m_n}+\tilde{L})=\sum_{k=1}^{l+m}\log (1+\eta_k) =\sum_{k=1}^{l+m}\sum_{p=1}^{\infty}\frac{(-1)^{p+1}\eta_k^p}{p}=\sum_{p=1}^{\infty}\frac{(-1)^{p+1}}{p}{\rm tr}(\tilde{L}^p)
\end{equation*}
almost surely. Therefore
\begin{eqnarray}\label{Hn2}
\hat{H}_n&=&-\frac{1}{2}Z^{\star}\left(\sum_{p=0}^{\infty}(-1)^p\tilde{L}^p\right)Z-\log\det D+\frac{1}{2}\sum_{p=1}^{\infty}\frac{(-1)^p}{p}{\rm tr}(\tilde{L}^p) \nonumber \\
&=&-\frac{1}{2}Z^{\star}\sum_{p=0}^{\infty}\left(
\begin{array}{ll}
(LL^{\star})^p & -(LL^{\star})^pL \\
-(L^{\star}L)^pL^{\star} & (L^{\star}L)^p \\
\end{array}
\right)Z-\log\det D+\frac{1}{2}\sum_{p=1}^{\infty}\frac{(-1)^p}{p}{\rm tr}(\tilde{L}^p) 
\end{eqnarray}
almost surely.

\section{The limit of $H_n$ and observation times}
In this section, we investigate the asymptotic behavior of $H_n$ and $\hat{H}_n$
to apply Ibragimov-Has'minskii's theory. 
To obtain these estimates, we need some convergence conditions
($[A3],[A3']$ and $[A3'\mathchar`-q,\eta]$ given in Section $3.1$) for the observation times.
Proposition \ref{Hn-lim} in Section $3.2$ will give asymptotic properties of $H_n$ and $\hat{H}_n$.

\subsection{Convergence conditions of functions of the observation times}
{\colorb
Since $M$ is a functional of $\{\rho_{I^i,J^j,\tau_n}\}_{i,j}$, we can write
\begin{equation*}
M=\left(
\begin{array}{ll}
M^{11}(\{\rho_{I^i,J^j,\tau_n}\}) & M^{12}(\{\rho_{I^i,J^j,\tau_n}\}) \\
(M^{12}(\{\rho_{I^i,J^j,\tau_n}\}))^{\star} & M^{22}(\{\rho_{I^i,J^j,\tau_n}\}) \\
\end{array}
\right).
\end{equation*}
Let $A_n$ be defined as 
\begin{equation*}
A_n(\mathcal{C}^1,\mathcal{C}^2,\mathcal{C}^3,\mathcal{C}^4)={\rm tr}(M^{11}(\mathcal{C}^4)\mathcal{C}^1)+2{\rm tr}((M^{12}(\mathcal{C}^4))^{\star}\mathcal{C}^3)+{\rm tr}(M^{22}(\mathcal{C}^4)\mathcal{C}^2),
\end{equation*}
where $\mathcal{C}^1,\mathcal{C}^2,\mathcal{C}^3,\mathcal{C}^4$ are complex matrices of size $l_n\times l_n,m_n\times m_n,l_n\times m_n$ and $l_n\times m_n$, respectively, 
and the absolute value of each element of $\mathcal{C}^4$ is less than $1$. Then}
we see $Z^{\star}MZ$ can be rewritten as 
\begin{equation*}
Z^{\star}MZ=A_n(\{Z_iZ_{i'}\}_{i,i'=1}^{l_n},\{Z_{j+l_n}Z_{j'+l_n}\}_{j,j'=1}^{m_n},\{Z_iZ_{j+l_n}\}_{i,j},\{\rho_{I^i,J^j,\tau_n}\}_{i,j}).
\end{equation*}
Let ${\bf 1}$ denote an $l_n\times m_n$ matrix with all elements equal $1$, 
$\{\nu_n^{p,i}\}_{n\in\mathbb{N},p\in\mathbb{Z}_+,i=1,2}$ be random measures on $[0,T)$ which satisfy
\begin{equation*}
\nu_n^{p,1}([0,t))=b_n^{-1}\sum_I((GG^{\star})^p)_{II}1_{\{L(I)\in [0,t)\} },\ \nu_n^{p,2}([0,t))=b_n^{-1}\sum_J((G^{\star}G)^p)_{JJ}1_{\{L(J)\in [0,t)\} },
\end{equation*}
and 
\begin{equation*}
\mathcal{E}^1(t)=\{\delta_{i,i'}1_{\{I^i\cap [0,t)\neq \emptyset \} }\}_{i,i'=1}^{l_n}, \quad \mathcal{E}^2(t)=\{\delta_{j,j'}1_{\{J^j\cap [0,t)\neq \emptyset \} }\}_{j,j'=1}^{m_n},
\end{equation*}
where $\delta$ denotes the Kronecker delta function.
Moreover, for $p\in\mathbb{Z}_+$ and $i=1,2$, let
\begin{eqnarray}
\Psi^{p,i}(f,g)&=&\Psi^{p,i,n}(f,g)=\int^T_0f(s)\nu^{p,i}_n(ds)-\int^T_0f(s)g(s)ds, \nonumber 
\end{eqnarray}
for $\mathbb{R}$-valued functions $f,g$ on $[0,T]$ such that $f$ is  ${\rm c\grave{a}dl\grave{a}g}$ and $g$ is Lebesgue integrable.
Note that $b_n^{-1}A_n(\mathcal{E}^1(t),{\bf 0},{\bf 0},z {\bf 1})=\sum_{p=0}^{\infty}z^{2p}\nu^{p,1}_n([0,t))$,
$b_n^{-1}A_n({\bf 0},\mathcal{E}^2(t),{\bf 0},z {\bf 1})=\sum_{p=0}^{\infty}z^{2p}\nu^{p,2}_n([0,t))$,
$\Psi^{0,1}(1_{[0,t)},g)=b_n^{-1}\sum_I1_{\{L(I)\in [0,t)\} }-\int^t_0g(s)ds$
and $\Psi^{0,2}(1_{[0,t)},g)=b_n^{-1}\sum_J1_{\{L(J)\in [0,t)\} }-\int^t_0g(s)ds$  
for $z\in\mathbb{C}$, $|z|<1$ and $t\in (0,T]$.

To obtain convergence of $H_n$, we consider the following condition. 
\begin{description}
\item{[$A3$]} There exist $\sigma(\{\Pi_n\}_n)$-measurable left-continuous processes $a_0(t)$ and $c_0(t)$ such that $\int^T_0a_0(t)dt\vee \int^T_0c_0(t)dt<\infty$ almost surely and 
\begin{equation}\label{a0}
\Psi^{0,1}(1_{[0,t)},a_0)\vee \Psi^{0,2}(1_{[0,t)},c_0)\to^p 0 \quad {\rm as} \quad n \to \infty 
\end{equation}
for any $t\in (0,T]$.
Moreover, at least one of the following conditions holds true.
\begin{enumerate}
\item There exist $\eta \in (0,1)$ and a $\sigma(\{\Pi_n\}_n)$-measurable left-continuous process $a(z, t)$ such that $a$ is continuous with respect to $z$, $\int^T_0a(z,t)dt<\infty$ and
$b_n^{-1}A_n(\mathcal{E}^1(t),{\bf 0},{\bf 0},z {\bf 1}) \to^p \int^t_0a(z, s)ds$ as $n \to \infty$ for $z\in \mathbb{C}, |z|<\eta$ and $t\in(0,T]$.
\item There exist $\eta \in (0,1)$ and a $\sigma(\{\Pi_n\}_n)$-measurable left-continuous process $c(z, t)$ such that $c$ is continuous with respect to $z$, $\int^T_0c(z,t)dt<\infty$ and 
$b_n^{-1}A_n({\bf 0},\mathcal{E}^2(t),{\bf 0},z {\bf 1}) \to^p \int^t_0c(z, s)ds$ as $n \to \infty$ for $z\in \mathbb{C}, |z|<\eta$ and $t\in(0,T]$.
\end{enumerate}
\end{description}
In particular, $\{l_n/b_n\}_n$ and $\{m_n/b_n\}_n$ are tight under $(\ref{a0})$.

$A_n(\mathcal{E}^1(T),{\bf 0},{\bf 0},z{\bf 1})$ and $A_n({\bf 0},\mathcal{E}^2(T),{\bf 0},z{\bf 1})$ appear 
in an asymptotically equivalent representation of $H_n$ when $b(x,\sigma)$ does not depend on $x$ and $\mu_t\equiv 0$.
Therefore convergence condition for observation times like $[A3]$ $1$ and $2$ are natural conditions to specify the limit of $H_n$.

\begin{description}
\item{[$A3'$]} There exist $\sigma(\{\Pi_n\}_n)$-measurable left-continuous processes $a_0(t)$ and $c_0(t)$ 
such that $\int^T_0a_0(t)dt\vee \int^T_0c_0(t)dt<\infty$ almost surely and 
(\ref{a0}) holds for any $t\in(0,T]$. Moreover, at least one of the following conditions holds true.
\begin{enumerate}
\item For any $p\in \mathbb{N}$, there exists a $\sigma(\{\Pi_n\}_n)$-measurable left-continuous process $a_p(t)$ such that 
$\int^T_0a_p(t)dt<\infty$ a.s. and for any $t\in(0,T]$, $\Psi^{p,1}(1_{[0,t)},a_p)\to^p 0$ as $n\to \infty$. 
\item For any $p\in \mathbb{N}$, there exists a $\sigma(\{\Pi_n\}_n)$-measurable left-continuous process $c_p(t)$ such that 
$\int^T_0c_p(t)dt<\infty$ a.s. and  for any $t\in(0,T]$, $\Psi^{p,2}(1_{[0,t)},c_p)\to^p 0$ as $n\to \infty$.
\end{enumerate}
\end{description}
As we will show later in Proposition \ref{A3eq}, $[A3]$ and $[A3']$ are equivalent under $[A2]$.

Let $q>2$ and $\eta\in (0,1)$.
For $\alpha\in (0,1/2)$ and $f:[0,T]\to \mathbb{R}$, $\alpha-$H${\rm \ddot{o}}$lder continuous, 
we denote $\omega_{\alpha}(f)=\sup_{t\neq s}|f_t-f_s|/|t-s|^{\alpha}$.

\begin{description}
\item{[$A3'\mathchar`-q,\eta$]} There exist $n_0\in\mathbb{N}$, $\alpha\in(0,1/2-1/q)$ 
and $\sigma(\{\Pi_n\}_n)$-measurable left-continuous processes $\{a_0(t)\}$, $\{c_0(t)\}$, $\{a_p(t)\}_{p\in\mathbb{N}}$
such that $\int^T_0(c_0\vee a_p)(t)dt \in L^q(\Omega)$ for $p\in\mathbb{Z}_+$, $E[(l_n+m_n)^q]<\infty$ for $n\in\mathbb{N}$ and 
\begin{eqnarray}
\sup_{n\geq n_0}E[(b_n^{\eta}|\Psi^{0,1}(f,a_0)|)^q]\vee E[(b_n^{\eta}|\Psi^{0,2}(f,c_0)|)^q]&\leq & C\left(\sup_t|f_t|^q+\omega_{\alpha}(f)^q\right), \nonumber \\
\max_{i=1,2}\sup_{p\in\mathbb{N}}\sup_{n\geq n_0}E[(b_n^{\eta}|\Psi^{p,i}(f,a_p)|)^q]/(1+p)^C&\leq & C\left(\sup_t|f_t|^q+\omega_{\alpha}(f)^q\right) \nonumber
\end{eqnarray}
for any $\alpha$-H${\rm \ddot{o}}$lder continuous function $f$ on $[0,T]$.
\end{description}
For $q>2$ and $\eta\in (0,1)$, it can be shown that $[A3'\mathchar`-q,\eta]$ implies $[A3']$.

The following lemma is easy to check. 
\begin{lemma}\label{sum-conv}
Let $\{\alpha_p\}_{p\in\mathbb{N}}\subset \mathbb{C}$ with $\sum_{p=1}^{\infty}|\alpha_p|<\infty$ and $\{\xi^n_p\}_{n,p\in\mathbb{N}}$ and $\{F_n\}_{n\in\mathbb{N}}$ be random variables satisfying
$\xi^n_p\to^p 0 \ (n\to \infty)$ for $p\in\mathbb{N}$, $\{F_n\}_{n\in\mathbb{N}}$ are tight, and $|\xi^n_p|\leq F_n, \ (n,p\in\mathbb{N})$. Then
$\sum_{p=1}^{\infty}\alpha_p\xi^n_p\to^p 0$ as $n\to \infty$.
\end{lemma}
The equivalence of $[A3] \ 1$ and $[A3] \ 2$ is established by our next lemma.
\begin{lemma}\label{A3equiv}
Assume $[A2]$ and that there exist stochastic processes $a_0(t)$ and $c_0(t)$ such that $\int^T_0a_0(t)dt\vee \int^T_0c_0(t)dt<\infty$ a.s. and $(\ref{a0})$ holds for $t\in(0,T]$.
Then $\nu^{p,1}_n([0,t)) -\nu^{p,2}_n([0,t)) \to^p 0$ for $t\in [0,T], p\geq 1$ as $n \to \infty$ and 
\begin{equation*}
b_n^{-1}A_n(\mathcal{E}^1(t),{\bf 0},{\bf 0},z {\bf 1})-b_n^{-1}A_n({\bf 0},\mathcal{E}^2(t),{\bf 0},z {\bf 1})\to^p \int^t_0(a_0-c_0)(s)ds
\end{equation*}
as $n \to \infty$ for $z\in \mathbb{C},|z|<1$ and $t\in [0,T]$. 

In particular, $[A3] \ 1\Longleftrightarrow [A3] \ 2$, $[A3'] \ 1\Longleftrightarrow [A3'] \ 2$ 
and $a_p\equiv c_p \ dt\times P$-a.e. $(t,\omega)$ for $p\geq 1$ under the assumptions above.
\end{lemma}
\begin{proof}
Since $\parallel G \parallel \vee \parallel G^{\star} \parallel \leq 1$ by Lemma \ref{GGT-lambda}, we have $|((GG^{\star})^p)_{II'}|\leq 1$, $|G_{IJ}|\leq 1$ for any $I,I',J$ and $p\in \mathbb{Z}_+$.
Then since $G_{IJ}\neq 0$ implies $I\cap J\neq \emptyset$, we obtain 
\begin{eqnarray}
|\nu^{p,1}_n([0,t))-\nu^{p,2}_n([0,t))| &=&b_n^{-1}\bigg|\sum_{I;L(I)\in [0,t)}\sum_{I'}\sum_J((GG^{\star})^{p-1})_{II'}G_{I'J}G^{\star}_{JI} \nonumber \\
&&-\sum_{J;L(J)\in [0,t)}\sum_I\sum_{I'}G^{\star}_{JI}((GG^{\star})^{p-1})_{II'}G_{I'J}\bigg| \nonumber \\
&\leq & 2b_n^{-1}\sum_{t-r_n\leq L(I)\leq t+r_n}1\to^p 0 \nonumber
\end{eqnarray}
as $n\to \infty$ for $p\geq 1$ by $[A2]$. 

Since $|\nu^{p,1}_n([0,t))-\nu^{p,2}_n([0,t))|\leq b_n^{-1}(l_n+m_n)$,
the desired conclusions are given by tightness of $\{b_n^{-1}(l_n+m_n)\}_n$ and Lemma \ref{sum-conv}.
\end{proof}
\begin{proposition}\label{A3eq}
$[A3]$ and $[A3']$ are equivalent under $[A2]$. Moreover, under $[A2]$ and $[A3]$, 
$a(\rho,t)=\sum_{p=0}^{\infty}a_p(t)\rho^{2p}$, $c(\rho,t)=\sum_{p=0}^{\infty}c_p(t)\rho^{2p}$ and 
\begin{equation*}
b_n^{-1}A_n(x^2\mathcal{E}^1(t),y^2\mathcal{E}^2(t),xy\rho_{\ast}\mathcal{E}^1(t)G,\rho{\bf 1})\to^p \int^t_0A(x,y,\rho,\rho_{\ast},s)ds
\end{equation*}
as $n\to \infty$ for $x,y\in\mathbb{R},\rho,\rho_{\ast}\in (-1,1), t\in (0,T]$,
where 
\begin{equation*}
A(x,y,\rho,\rho_{\ast},t)=x^2a(\rho,t)+y^2c(\rho,t)-2xy(a(\rho,t)-a_0(t))\rho_{\ast}/\rho1_{\{\rho\neq 0\} }.
\end{equation*}
\end{proposition}
The convergent sequence which appears in Proposition \ref{A3eq} is asymptotically equivalent representation of $b_n^{-1}Z^{\star}MZ$
if $\mu_t\equiv 0$ and $b(x,\sigma)$ does not depend on $x$.
Therefore, the convergence result in Proposition \ref{A3eq} is the convergence result of $b_n^{-1}Z^{\star}MZ$ with ignoring the structure of diffusion coefficients $(b_t^1,b_t^2)$.

\subsection{The limit of $H_n$}
We discuss asymptotic behavior of $H_n$ under $[A3],[A3'\mathchar`-q,\eta]$.

First, we assume one more condition. 
Let $\mathcal{I}$ be a set of intervals defined by
\begin{equation*}
\mathcal{I}=\{I^i\}_{i=1}^{l_n}\cup \{J^j\}_{j=1}^{m_n}\cup \{[\mathcal{T}^{j-1}_k,\mathcal{T}^j_k); 1\leq k\leq n_2, 1\leq j \leq m^k_n\}.
\end{equation*}
Let $\theta_{0,k}=I^k$ for $1\leq k\leq l_n$, $\theta_{0,k}=J^{k-l_n}$ for $l_n<k \leq l_n+m_n$, and
\begin{equation*}
\theta_{p,k}=\cup \{K_{2p};K_1,\cdots,K_{2p}\in \mathcal{I}, K_1\cap \theta_{0,k}\neq \emptyset,K_j\cap K_{j-1}\neq \emptyset \ (1\leq j\leq 2p)\}
\end{equation*}
for $p\in\mathbb{N}$ and $1\leq k\leq l_n+m_n$.
Moreover, let $\Phi_{p,i}=\sum_k|\theta_{p,k}|^i$, $\bar{\Phi}_{p_1,p_2}=\sum_{k_1,k_2}|\theta_{p_1,k_1}\cap \theta_{p_2,k_2}|$ for $i\in \{1,2\}$ and $p,p_1,p_2\in\mathbb{Z}_+$. 
For $q\geq 2$ and $\delta\geq 1$, we consider the following conditions. 
\begin{description}
\item{[$A4$]} There exist $\delta'\geq 1$ such that
\begin{equation*}
\bigg\{(b_n^{-1}\vee r_n^2)\sum_{p=0}^{\infty}\frac{(\Phi_{2p+2,1})^2}{(p+1)^{2\delta'}}\bigg\}\vee \bigg\{b_n^{-1}\sum_{p_1,p_2=0}^{\infty}\frac{\bar{\Phi}_{2p_1+3,2p_2+3}}{(p_1+1)^{\delta'}(p_2+1)^{\delta'}}\bigg\}\to^p 0
\end{equation*}
as $n\to\infty$.
\item{[$A4\mathchar`-q,\delta$]} 
\begin{enumerate}
\item \begin{equation*}
\lim_{n\to\infty}E\bigg[(b_n^{-\frac{q}{2}}\vee r_n^q)\sum_{p=0}^{\infty}\frac{(\Phi_{2p+2,1})^q}{(p+1)^{q\delta}}\bigg]= 0.
\end{equation*}
\item 
\begin{equation*}
\lim_{n\to\infty}E\bigg[\bigg(b_n^{-1}\sum_{p_1,p_2=0}^{\infty}\frac{\bar{\Phi}_{2p_1+3,2p_2+3}}{(p_1+1)^{\delta}(p_2+1)^{\delta}}\bigg)^{\frac{q}{2}}\bigg]=0.
\end{equation*}
\end{enumerate}
\end{description}
We can see that $[A4\mathchar`-q,\delta]$ implies $[A4]$ for any $q\geq 2$ and $\delta\geq 1$ by Jensen's inequality.
Moreover, we can use the following condition instead of $[A4]$.

\begin{description}
\item{[$A4'$]} There exist positive constants $\delta_1,\delta_2,\delta_3$ such that $(3\delta_1+2\delta_3)\vee (\delta_1+\delta_2)<1$ 
and the following two conditions hold:
\begin{enumerate}
\item $\lim_{n\to\infty}P[r_n\geq b_n^{-1+\delta_1}]=0$.
\item 
\begin{eqnarray}
\lim_{n\to\infty}b_n^2\sup_{j_1,j_2\in\mathbb{N}, |j_1-j_2|\geq b_n^{\delta_2}}P\bigg[l_n\geq j_1\vee j_2  \ {\rm and} \ \frac{|S^{j_2}-S^{j_1}|}{|j_2-j_1|}\leq b_n^{-1-\delta_3}\bigg]&=&0, \nonumber \\
\lim_{n\to\infty}b_n^2\sup_{j_1,j_2\in\mathbb{N}, |j_1-j_2|\geq b_n^{\delta_2}}P\bigg[m_n\geq j_1\vee j_2  \ {\rm and} \ \frac{|T^{j_2}-T^{j_1}|}{|j_2-j_1|}\leq b_n^{-1-\delta_3}\bigg]&=&0. \nonumber
\end{eqnarray}
\end{enumerate}
\end{description}
\begin{lemma}
Assume $[A4']$ and that $\{(l_n+m_n)/b_n\}_{n\in\mathbb{N}}$ is tight. Then $[A4]$ holds.
\end{lemma}
\begin{proof}
Let $\epsilon>0$. Then there exists a positive constant $K$ such that $P[(l_n+m_n)/b_n>K]<\epsilon$ for any $n\in\mathbb{N}$ by tightness of $\{(l_n+m_n)/b_n\}_n$.

Let $\phi(k)$ denote minimal $k'>l_n$ which satisfy $I^k\cap J^{k'-l_n}\neq \emptyset$ for $1\leq k\leq l_n$,
\begin{equation*}
\mathcal{U}^n_{j_1,j_2}=\bigg\{l_n\geq j_1\vee j_2 \ {\rm and} \ \frac{|S^{j_2}-S^{j_1}|}{|j_2-j_1|}\leq b_n^{-1-\delta_3}\bigg\}^c\bigcap \bigg\{m_n\geq j_1\vee j_2 \ {\rm and} \  \frac{|T^{j_2}-T^{j_1}|}{|j_2-j_1|}\leq b_n^{-1-\delta_3}\bigg\}^c
\end{equation*}
for $j_1,j_2\in\mathbb{N}$ and 
\begin{equation*}
\bar{\mathcal{U}}_n=\{r_n<b_n^{-1+\delta_1}\}\cap \cap_{j_1,j_2\in\mathbb{N},|j_2-j_1|\geq b_n^{\delta_2}}\mathcal{U}^n_{j_1,j_2}.
\end{equation*}
Then on $\bar{\mathcal{U}}_n$, for $k_1\leq l_n$, $l_n<k_2\leq l_n+m_n$ and $p_1,p_2\in\mathbb{Z}_+$ which satisfy $|\phi(k_1)-k_2|\geq b_n^{\delta_2}$ and $\theta_{p_1,k_1}\cap \theta_{p_2,k_2}\neq \emptyset$, we have
\begin{equation*}
|\phi(k_1)-k_2|b_n^{-1-\delta_3}<|T^{\phi(k_1)}-T^{k_2}|\leq (2p_1+2p_2+2)r_n< (2p_1+2p_2+2)b_n^{-1+\delta_1}.
\end{equation*}
Therefore $|\phi(k_1)-k_2|b_n^{-\delta_1-\delta_3}<2(p_1+1)(p_2+1)$.

Then by using the relation $|\theta_{p_1,k_1}\cap \theta_{p_2,k_2}|\leq \{(4p_1+1)\wedge (4p_2+1)\}r_n$, we obtain
\begin{eqnarray}
&&\sum_{p_1,p_2=0}^{\infty}\sum_{k_1\leq l_n,k_2>l_n}\frac{|\theta_{p_1,k_1}\cap \theta_{p_2,k_2}|}{(p_1+1)^5(p_2+1)^5} \nonumber \\
&\leq &C\sum_{k_1\leq l_n}\bigg\{\sum_{k_2>l_n,|\phi(k_1)-k_2|\geq b_n^{\delta_2}}\sum_{p_1,p_2}\frac{\{(4p_1+1)\wedge (4p_2+1)\}b_n^{-1+\delta_1}}{(p_1+1)^3(p_2+1)^3|\phi(k_1)-k_2|^2b_n^{-2\delta_1-2\delta_3}}+b_n^{\delta_2}b_n^{-1+\delta_1}\bigg\} \nonumber \\
&\leq & Cb_n^{-1+(3\delta_1+2\delta_3)\vee (\delta_1+\delta_2)}\sum_{k_1\leq l_n}\bigg(\sum_{k_2>l_n,|\phi(k_1)-k_2|\geq b_n^{\delta_2}}\frac{1}{|\phi(k_1)-k_2|^2}+1\bigg) \nonumber \\
&\leq & Cl_nb_n^{-1+(3\delta_1+2\delta_3)\vee (\delta_1+\delta_2)} \nonumber
\end{eqnarray}
on $\bar{\mathcal{U}}_n$. Similar arguements for other combinations of $k_1$ and $k_2$ yield 
\begin{equation*}
\sum_{p_1,p_2=0}^{\infty}\frac{\bar{\Phi}_{2p_1+3,2p_2+3}}{(p_1+1)^5(p_2+1)^5}\leq C(l_n+m_n)b_n^{-1+(3\delta_1+2\delta_3)\vee (\delta_1+\delta_2)} \quad {\rm on } \quad \bar{\mathcal{U}}_n.
\end{equation*}
On the other hand, since
\begin{equation*}
P[\bar{\mathcal{U}}_n^c]\leq P[r_n\geq b_n^{-1+\delta_1}]+P[(l_n+m_n)/b_n>K]+\sum_{1\leq j_1,j_2\leq [Kb_n], |j_2-j_1|\geq b_n^{\delta_2}}P[(\mathcal{U}^n_{j_1,j_2})^c],
\end{equation*}
we obtain $\limsup_{n\to \infty}P[\bar{\mathcal{U}}_n^c]\leq \epsilon$ by $[A4']$. Since $\epsilon>0$ is arbitrary, we have 
\begin{equation*}
b_n^{-1}\sum_{p_1,p_2=0}^{\infty}\frac{\bar{\Phi}_{2p_1+3,2p_2+3}}{(p_1+1)^5(p_2+1)^5}\to^p 0.
\end{equation*}
It is easier to prove the convergence about $\Phi_{2p+2,1}$.
\end{proof}

Let $B^i(x,\sigma)=| b^i(x,\sigma_{\ast})|/| b^i(x,\sigma)| \ (i=1,2)$,
\begin{equation*}
C(\rho,t)=\sum_{p=1}^{\infty}\frac{a_p(t)}{p}\rho^{2p}=\sum_{p=1}^{\infty}\frac{c_p(t)}{p}\rho^{2p},
\end{equation*}
and
\begin{equation}\label{h8}
h_t^{\infty}(\sigma)= -\frac{1}{2}A(B^1_t,B^2_t,\rho_t,\rho_{t,\ast},t)-a_0\log | b^1_t| -c_0\log | b^2_t| -\frac{1}{2}C(\rho_t,t)
\end{equation}
for $t\in [0,T],\rho\in (-1,1)$.
\begin{proposition}\label{Hn-lim}
\begin{enumerate}
\item Let $0\leq v\leq 3$. Assume $[A1]-[A3]$. Then
\begin{equation*}
\sup_{\sigma\in \Lambda}\bigg|b_n^{-1}\partial_{\sigma}^vH_n(\sigma)- \int^T_0\partial_{\sigma}^vh^{\infty}_t(\sigma)dt\bigg| \to^p 0
\end{equation*}
as $n\to \infty$. 
\item Let $0\leq v\leq 3$, $q\in 2\mathbb{N}$, $q> 2\vee n_1$, $\delta>1$, $\xi>0$, $\eta \in (0,1)$ and $\{s_n\}_{n\in\mathbb{N}}$ be stochastic processes.
Assume $[A1],[A2],[A3'\mathchar`-q,\eta],[A4\mathchar`-(2q),\delta],[S\mathchar`-((2v+2[\delta]+12)q),\xi]$ for $\{s_n\}$, and that
$\sup_nE[b_n^{-2q}(l_n+m_n)^{2q}]<\infty$.
Then 
\begin{equation*}
\sup_nE\left[\left(\sup_{\sigma\in\Lambda}b_n^{\eta'}\bigg|b_n^{-1}\partial_{\sigma}^v\hat{H}_n(\sigma;s_n)-\int^T_0\partial_{\sigma}^vh^{\infty}_t(\sigma)dt\bigg|\right)^q\right] < \infty
\end{equation*}
for $\eta'\leq \eta\wedge (1/2)\wedge (\xi/2q)$.
\end{enumerate}
\end{proposition}

\section{Separation of the limit of $H_n$}
We deal with Condition $[H]$ about separation of the limit of $H_n$ which is necessary to apply Ibragimov-Has'minskii's theory.
In the case of synchronous and equi-space samplings : $S^i=T^i=[b_n]^{-1}iT \ (0\leq i\leq [b_n])$, 
Uchida and Yoshida \cite{uch-yos} discussed tractable sufficient conditions for Condition $[H0]$ of separation.
In this section, we will confirm that $[H0]$ implies $[H]$ under certain conditions.

Under $[A1]-[A3]$, we define 
$\mathcal{Y}_n(\sigma;\sigma_{\ast})=b_n^{-1}(H_n(\sigma)-H_n(\sigma_{\ast}))$, and $\mathcal{Y}(\sigma;\sigma_{\ast})$ denotes the probability limit of $\mathcal{Y}_n(\sigma;\sigma_{\ast})$.
By Proposition \ref{Hn-lim}, we obtain
$\mathcal{Y}(\sigma;\sigma_{\ast})=\int^T_0(h^{\infty}_t(\sigma)-h^{\infty}_{t}(\sigma_{\ast}))dt$.

Moreover, the equation (\ref{h8}) can be rewritten as
\begin{eqnarray}\label{h-inf-A}
h^{\infty}_t(\sigma)&=&-\frac{1}{2}(B^1_t)^2(a_0+\mathcal{A}(\rho_t))-\frac{1}{2}(B^2_t)^2(c_0+\mathcal{A}(\rho_t))+B^1_tB^2_t\mathcal{A}(\rho_t)\frac{\rho_{t,\ast}}{\rho_t} \nonumber \\
&&-a_0\log | b^1_t| -c_0\log| b^2_t| +\int^{\rho_t}_0\frac{\mathcal{A}(\rho)}{\rho}d\rho,
\end{eqnarray}
where $\mathcal{A}(\rho)=\mathcal{A}(\rho,t)=a(\rho,t)-a_0(t)=c(\rho,t)-c_0(t)$.
Since $B^1_{t,\ast}=B^2_{t,\ast}=1$,
\begin{eqnarray}
h^{\infty}_t(\sigma_{\ast})&=&-\frac{1}{2}a_0-\frac{1}{2}c_0-a_0\log | b^1_{t,\ast}| -c_0\log| b^2_{t,\ast}| +\int^{\rho_{t,\ast}}_0\frac{\mathcal{A}(\rho)}{\rho}d\rho. \nonumber
\end{eqnarray}
Therefore for $y_t(\sigma)=h^{\infty}_t(\sigma)-h^{\infty}_t(\sigma_{\ast})$, it follows that
\begin{eqnarray}\label{y-rep}
y_t(\sigma)&=&-\frac{1}{2}(B^1_t)^2(a_0+\mathcal{A})-\frac{1}{2}(B^2_t)^2(c_0+\mathcal{A})+B^1_tB^2_t\mathcal{A}\frac{\rho_{t,\ast}}{\rho_t}+\frac{a_0}{2}+\frac{c_0}{2} \nonumber \\
&&+a_0\log B^1_t+c_0\log B^2_t+\int^{\rho_t}_{\rho_{t,\ast}}\frac{\mathcal{A}}{\rho}d\rho \nonumber \\
&=&-\frac{1}{2}(a_0+\mathcal{A})(B^1_t-B^2_t)^2+a_0+a_0\log B^1_tB^2_t + \int^{\rho_t}_{\rho_{t,\ast}}\frac{\mathcal{A}}{\rho}d\rho \nonumber \\
&&+ \frac{c_0-a_0}{2}(1-(B^2_t)^2+\log (B^2_t)^2)+B^1_tB^2_t(\mathcal{A}\rho_{t,\ast}/\rho_t-\mathcal{A}-a_0) \nonumber \\
&=&-\frac{1}{2}(c_0+\mathcal{A})(B^1_t-B^2_t)^2+c_0+c_0\log B^1_tB^2_t + \int^{\rho_t}_{\rho_{t,\ast}}\frac{\mathcal{A}}{\rho}d\rho \nonumber \\
&&+ \frac{a_0-c_0}{2}(1-(B^1_t)^2+\log (B^1_t)^2)+B^1_tB^2_t(\mathcal{A}\rho_{t,\ast}/\rho_t-\mathcal{A}-c_0). 
\end{eqnarray}

Let $F(x)=1-x+\log x \ (x>0)$. 
\begin{lemma}\label{Fx-est}
Let $\epsilon_1\in (0, e^{-1}]$, $\epsilon_2\geq 1$. Then 
\begin{equation*}
-\log(1/\epsilon_1)(x-1)^2\leq F(x)\leq -(x-1)^2/4\epsilon_2^2,
\end{equation*}
for $x\in [\epsilon_1,1+\epsilon_2]$.
\end{lemma}
\begin{proof}
For $0< x\leq 1+\epsilon_2$, since $-1/\epsilon_2< (x-1)/\epsilon_2\leq 1$, it follows that
\begin{equation*}
F(x)\leq -((x-1)^2\wedge 1)/4\leq -(x-1)^2/4\epsilon_2^2.
\end{equation*}
On the other hand, for $x\geq \epsilon_1$, let $f(x)=F(x)+\log (1/\epsilon_1)(x-1)^2$ then 
since $f'(x)=(x-1)(2\log (1/\epsilon_1)-1/x)$, we have $f(x)\geq f(1)\wedge f(\epsilon_1)\geq 0$.
\end{proof}
Let
\begin{eqnarray}
f_1(t,x,\rho,\rho_{\ast})&=&a_0+a_0\log x+\int^{\rho}_{\rho_{\ast}}\mathcal{A}(\rho')/\rho' d\rho' + x(\mathcal{A}\rho_{\ast}/\rho-\mathcal{A}-a_0), \nonumber \\
f_2(t,x,\rho,\rho_{\ast})&=&c_0+c_0\log x+\int^{\rho}_{\rho_{\ast}}\mathcal{A}(\rho')/\rho' d\rho' + x(\mathcal{A}\rho_{\ast}/\rho-\mathcal{A}-c_0), \nonumber 
\end{eqnarray}
and $C_1=(1-\bar{\rho}_T^2)^2/(12R^8)$. 

\begin{lemma}\label{sep-lem}
Assume $[A1]-[A3]$. 
Then 
\begin{equation*}
(f_1\vee f_2)(t,B^1_tB^2_t,\rho_t,\rho_{t,\ast})\leq -C_1\left\{a_1(t)(\rho_t-\rho_{t,\ast})^2+(a_0\wedge c_0)(t)(B^1_tB^2_t-1)^2\right\}
\end{equation*}
for $dt\times P$-a.e. $(t,\omega)\in [0,T]\times \Omega$.
\end{lemma}

We write $\mathcal{Y}_0$ for $\mathcal{Y}$ defined by using the same processes $X,Y$ 
and synchronous equi-spaced samplings $S^i=T^i=\mathcal{T}^i_k=[b_n]^{-1}iT \ (0\leq i\leq [b_n],1\leq k\leq n_2)$. Let
\begin{equation*}
\chi(\sigma_{\ast})=\inf_{\sigma\neq \sigma_{\ast}}\frac{-\mathcal{Y}(\sigma;\sigma_{\ast})}{|\sigma-\sigma_{\ast}|^2}, \quad \chi_0(\sigma_{\ast})=\inf_{\sigma\neq \sigma_{\ast}}\frac{-\mathcal{Y}_0(\sigma;\sigma_{\ast})}{|\sigma-\sigma_{\ast}|^2}.
\end{equation*}
Moreover, we consider the following conditions.
\begin{description}
\item{$[H]$} For every $L>0$, there exists $c_L>0$ such that for all $r>0$, 
$P[\chi\leq r^{-1}]\leq c_L/r^L$.
\item{$[H0]$} For every $L>0$, there exists $c_L>0$ such that for all $r>0$, 
$P[\chi_0\leq r^{-1}]\leq c_L/r^L$.
\item{$[H']$} $\chi>0$ almost surely.
\end{description}
Obviously, $[H]$ implies $[H']$.

\begin{lemma}\label{y0-est}
Assume $[A1]$. Then there exists a positive random variable $\mathcal{R}'$ 
which does not depend on $\sigma,\sigma_{\ast}$, such that $E[(\mathcal{R}')^q]<\infty$ for any $q>0$ and
\begin{equation*}
\mathcal{Y}_0(\sigma;\sigma_{\ast})\geq -\mathcal{R}'\int^T_0\left\{(B^1_t-B^2_t)^2+(B^1_tB^2_t-1)^2 +(\rho_t-\rho_{t,\ast})^2\right\}dt
\end{equation*}
for any $\sigma, \sigma_{\ast}\in \Lambda$.
\end{lemma}

The following proposition ensures that to prove $[H]$, 
it is enough to prove $[H0]$ which is the condition of separation for synchronous and equi-spaced observations.

\begin{proposition}\label{sep-cond}
Assume $[A1]-[A3]$. Then there exist positive random variables $\mathcal{R}$ which do not depend on $\sigma, \sigma_{\ast}$ 
such that $E[\mathcal{R}^{-q}]<\infty$ for any $q\geq 1$ and
\begin{equation*}
\mathcal{Y}(\sigma;\sigma_{\ast})\leq -\mathcal{R}\int^T_0\left\{(a_0\wedge c_0)\{(B^1_t-B^2_t)^2+(B^1_tB^2_t-1)^2\} +a_1(\rho_t-\rho_{t,\ast})^2\right\}dt
\end{equation*}
for $\sigma,\sigma_{\ast}\in \Lambda$.
In particular, if $E[({\rm ess \, inf}_{t\in [0,T]} a_1(t))^{-q}]<\infty$ for any $q>0$, $[H0]$ implies $[H]$.
\end{proposition}

\begin{proof}
In the equation (\ref{y-rep}), by using the second representation if $a_0\leq c_0$ 
and using the third representation if $a_0> c_0$, we obtain
\begin{equation*}
y_t(\sigma)\leq -\frac{1}{2}(a_0\wedge c_0+\mathcal{A})(B^1_t-B^2_t)^2+(f_1\vee f_2)(t,B^1_tB^2_t,\rho_t,\rho_{t,\ast}).
\end{equation*}
By Lemma \ref{sep-lem}, we have
\begin{equation*}
y_t(\sigma)\leq -\frac{1}{2}(a_0\wedge c_0)(B^1_t-B^2_t)^2-C_1a_1(\rho_t-\rho_{t,\ast})^2-C_1(a_0\wedge c_0)(B^1_tB^2_t-1)^2
\end{equation*}
for $dt\times P$-a.e. $(t,\omega)$, where $C_1=(1-\bar{\rho}_T^2)^2/(12R^8)$.
Therefore by integrating with respect to $t$,
\begin{equation*}
\mathcal{Y}(\sigma;\sigma_{\ast})\leq -\mathcal{R}\int^T_0\left((a_0\wedge c_0)\{(B^1_t-B^2_t)^2+(B^1_tB^2_t-1)^2\} +a_1(\rho_t-\rho_{t,\ast})^2\right)dt,
\end{equation*}
where $\mathcal{R}=C_1$. 
In particular, let $E[({\rm ess \, inf}_{t\in [0,T]} a_1(t))^{-q}]<\infty$ for any $q>0$ and $[H0]$ hold.
We can see $a_0\wedge c_0 \geq a_1$ for $dt\times P$-a.e. $(t,\omega)$ since $\nu^{1,i}([0,t))\leq \nu^{0,i}([0,t))$ for any $t\in (0,T]$ and $i=1,2$.
Therefore by Lemma \ref{y0-est} we have
\begin{eqnarray}
\mathcal{Y}(\sigma;\sigma_{\ast})&\leq & -\mathcal{R}{\rm ess \, inf}_t a_1(t)\int^T_0\left((B^1_t-B^2_t)^2+(B^1_tB^2_t-1)^2+(\rho_t-\rho_{t,\ast})^2\right)dt \nonumber \\
&\leq &\mathcal{R}(\mathcal{R}')^{-1}{\rm ess \, inf}_t a_1(t)\mathcal{Y}_0(\sigma;\sigma_{\ast}) \nonumber
\end{eqnarray}
almost surely, where $\mathcal{R}'$ is defined in Lemma \ref{y0-est}.
Hence $\chi\geq \mathcal{R}(\mathcal{R}')^{-1}{\rm ess \, inf}_t a_1(t)\chi_0$ a.s. and for any $L>0$, there exists a constant $c_L>0$ such that
\begin{eqnarray}
P[\chi\leq r^{-1}]&\leq & P[\chi_0\leq r^{-1/2}]+P[\mathcal{R}(\mathcal{R}')^{-1}{\rm ess \,inf}_t a_1(t)\leq r^{-1/2}] \nonumber \\
&\leq & \frac{c_{2L,0}}{r^L}+\frac{1}{r^L}E\left[\left(\mathcal{R}'\mathcal{R}^{-1}({\rm ess \, inf}_t a_1(t))^{-1}\right)^{2L}\right]\leq \frac{c_L}{r^L}, \nonumber
\end{eqnarray}
where $c_{2L,0}$ denotes the coefficient of $r^{-2L}$ in $[H0]$. This gives $[H]$.
\end{proof}

\begin{remark}
In the case of nonsynchronous observation, under $[A1]$ and $[A3]$, we can prove an inequality
\begin{equation*}
\mathcal{Y}(\sigma;\sigma_{\ast})\geq -\mathcal{R}'\int^T_0\left((a_0\vee c_0)\{(B^1_t-B^2_t)^2+(B^1_tB^2_t-1)^2\} +a_1(\rho_t-\rho_{t,\ast})^2\right)dt,
\end{equation*}
which corresponds to Lemma \ref{y0-est}.
\end{remark}

\begin{remark}\label{H0-suff}
By Proposition \ref{sep-cond}, it follows that
\begin{equation*}
\mathcal{Y}_0(\sigma;\sigma_{\ast})\leq -\mathcal{R}\int^T_0\{(B^1_t-B^2_t)^2+(B^1_tB^2_t-1)^2+(\rho_t-\rho_{t,\ast})^2\}dt.
\end{equation*}
On the other hand, we can see that there exists a positive random variable $\tilde{\mathcal{R}}$ such that $E[\tilde{\mathcal{R}}^q]<\infty$ for any $q>0$ and
\begin{equation*}
| (bb^{\star})_t-(bb^{\star})_{t,\ast}|^2 \leq \tilde{\mathcal{R}}\{(B^1_t-B^2_t)^2+(B^1_tB^2_t-1)^2+(\rho_t-\rho_{t,\ast})^2\}
\end{equation*}
for any $t\in[0,T],\sigma,\sigma_{\ast}\in\Lambda$ by using the inequality $(x-1)^2+(y-1)^2\leq (x-y)^2+2(xy-1)^2 \ (x,y\geq 0)$.
Therefore $[H0]$ holds if there exists  a constant $\epsilon>0$ such that 
\begin{equation*}
|(bb^{\star})(x,\sigma_1)-(bb^{\star})(x,\sigma_2)| \geq \epsilon |\sigma_1-\sigma_2|
\end{equation*}
for any $x\in\mathbb{R}^{n_2},\sigma_1,\sigma_2\in \Lambda$.
Weaker sufficient conditions for $[H0]$ can be found in Uchida and Yoshida \cite{uch-yos}.
\end{remark}

\section{Asymptotic properties of the quasi-maximum-likelihood estimator and the Bayes type estimator}
In this section, we investigate consistency and asymptotic mixed normality of the quasi-maximum-likelihood estimator 
and the Bayes type estimator as main results. 

Let the quasi-maximum likelihood estimator $\hat{\sigma}_n$ of the parameter $\sigma$ be $\sigma\in\bar{\Lambda}$ which maximize $H_n$. 
If maximizing points are not unique, we select so that $\hat{\sigma}_n$ is measurable.
\begin{theorem}\label{thm-cons}
Assume $[A1]-[A3],[H']$. 
Then $\hat{\sigma}_n\to^p\sigma_{\ast}$ as $n\to \infty$.
\end{theorem}
\begin{proof}
By Proposition \ref{Hn-lim}, we have $\sup_{\sigma}|\mathcal{Y}_n(\sigma)-\mathcal{Y}(\sigma)|\to^p 0$ as $n\to \infty$. 
On the other hand, by $[H']$, for any $\epsilon,\delta>0$, 
there exists $\eta>0$ such that $P[\chi\leq \eta]\leq \epsilon$. 
Since $\mathcal{Y}_n(\hat{\sigma}_n)\geq 0$, it follows that
\begin{eqnarray}
P[|\hat{\sigma}_n-\sigma_{\ast}|\geq \delta]\leq P[\chi\leq \eta]+P[\mathcal{Y}(\hat{\sigma}_n)\leq -\eta\delta^2] \leq  \epsilon + P[\sup_{\sigma}|\mathcal{Y}_n(\sigma)-\mathcal{Y}(\sigma)|\geq \eta\delta^2]. \nonumber
\end{eqnarray}
Therefore there exists $n_0\in\mathbb{N}$ such that $P[|\hat{\sigma}_n-\sigma_{\ast}|\geq \delta]\leq 2\epsilon$ if $n\geq n_0$. 
\end{proof}

Let $\{s_n\}_{n\in\mathbb{N}}$ be stochastic processes which satisfy $[S]$, 
\begin{equation*}
\Gamma=-\int^T_0\partial^2_{\sigma}h^{\infty}_t(\sigma_{\ast})dt, 
\end{equation*}
$U_n(\sigma_{\ast})=\{u\in\mathbb{R}^{n_1};\sigma_{\ast}+b_n^{-1/2}u\in\Lambda\}$, 
$V_n(r,\sigma_{\ast})=\{|u|\geq r\}\cap U_n(\sigma_{\ast})$, 
and $\mathcal{Z}_n(u,\sigma_{\ast})=\exp(\hat{H}_n(\sigma_{\ast}+b_n^{-1/2}u;s_n)-\hat{H}_n(\sigma_{\ast};s_n))$
for $u\in U_n(\sigma)$.

\begin{proposition}\label{pld}(polynomial type large deviation inequalities)
Let $L>0$ and $\delta \in (0,1/2)$. Assume for any $q>0$ there exists $q'\in 2\mathbb{N},q'>q$ and $\delta'\geq 1$ 
such that $[A1]$, $[A2]$, $[A3'\mathchar`-q',\delta]$, $[A4\mathchar`-q',\delta']$, $[H]$ 
and $[S\mathchar`-q',2q'\delta]$ hold for $\{s_n\}$.
Then there exists $C_L>0$ such that
\begin{equation*}
P\bigg[\sup_{u\in V_n(r,\sigma_{\ast})}\mathcal{Z}_n(u,\sigma_{\ast})\geq e^{-r/2}\bigg]\leq \frac{C_L}{r^L}
\end{equation*}
for any $n\in\mathbb{N}$ and $r>0$.
\end{proposition}

Let $\mathcal{N}$ be an $n_1$-dimensional standard normal random variable which is defined on an extension of $(\Omega,\mathcal{F},P)$
and independent of $\mathcal{F}$. 
We use the same notation $E$ for expectations on the extension of $(\Omega,\mathcal{F},P)$.

\begin{theorem}\label{asym-dist}
\begin{enumerate}
\item Assume $[A1]-[A4],[H']$. Then $b_n^{1/2}(\hat{\sigma}_n-\sigma_{\ast})\to^{s\mathchar`-\mathcal{L}}\Gamma^{-1/2}\mathcal{N}$ as $n\to\infty$.
\item Let $\delta \in (0,1/2)$. Assume for any $q>0$, there exists $q'\in 2\mathbb{N},q'>q$ and $\delta'\geq 1$ such that $[A1]$, $[A2\mathchar`-q',\delta]$, $[A3'\mathchar`-q',\delta]$, $[A4\mathchar`-q',\delta']$, $[H]$ hold. 
Then $E[Y'f(b_n^{1/2}(\hat{\sigma}_n-\sigma_{\ast}))]\to E[Y'f(\Gamma^{-1/2}\mathcal{N})]$
as $n\to \infty$ for any continuous function $f$ of at most polynomial growth and any bounded random variable $Y'$ on $(\Omega,\mathcal{F})$.
\end{enumerate}
\end{theorem}

We also consider the Bayes type estimator $\tilde{\sigma}_n$ for a prior density $\pi : \Lambda \to \mathbb{R}_+$ defined as 
\begin{equation}\label{bayes-def}
\tilde{\sigma}_n=\bigg(\int_{\Lambda}\exp(H_n(\sigma))\pi(\sigma)d\sigma\bigg)^{-1}\int_{\Lambda}\sigma\exp(H_n(\sigma))\pi(\sigma)d\sigma.
\end{equation}

\begin{theorem}\label{bayes-asym-dist}
Let $\delta \in (0,1/2)$. Assume for any $q>0$ there exists $q'\in 2\mathbb{N},q'>q$ and $\delta' \geq 1$ such that $[A1]$, $[A2\mathchar`-q',\delta]$, $[A3'\mathchar`-q',\delta]$, $[A4\mathchar`-q',\delta']$, $[H]$ hold, 
and that the prior density $\pi$ is continuous and $0< \inf_{\sigma}\pi(\sigma)\leq \sup_{\sigma}\pi(\sigma)<\infty$. Then
$b_n^{1/2}(\tilde{\sigma}_n-\sigma_{\ast})\to^{s\mathchar`-\mathcal{L}}\Gamma^{-1/2}\mathcal{N}$ as $n\to \infty$.
Moreover, $E[Y'f(b_n^{1/2}(\tilde{\sigma}_n-\sigma_{\ast}))]\to E[Y'f(\Gamma^{-1/2}\mathcal{N})]$ as $n\to \infty$ for any continuous function $f$ of at most polynomial growth
and any bounded random variable $Y'$ on $(\Omega,\mathcal{F})$.
\end{theorem}

\section{Sufficient conditions for the conditions about the observation times}
In this section, we will introduce tractable sufficient conditions for $[A2\mathchar`-q,\delta]$, $[A3'\mathchar`-q,\eta]$, $[A4\mathchar`-q,\delta]$,
and the estimate with respect to ${\rm ess \, inf}_t a_1$ in Proposition \ref{sep-cond}.

Let $q>0$. We consider the following conditions for point processes 
$\{N^i_t\}_{0\leq t\leq T, 1\leq i\leq n_2+2}$ which generate observations.
\begin{description}
\item{[$B1\mathchar`-q$]} There exists $n_0\in\mathbb{N}$ such that
\begin{equation*}
\sup_{n\geq n_0}\max_{1\leq i\leq n_2+2}\sup_{0\leq t \leq T-b_n^{-1}}E\left[(N^i_{t+b_n^{-1}}-N^i_t)^q\right]<\infty.
\end{equation*}
\item{[$B2\mathchar`-q$]} There exists $n_0\in\mathbb{N}$ such that
\begin{equation*}
\limsup_{u\to\infty}\sup_{n\geq n_0}\max_{1\leq i\leq n_2+2}\sup_{0\leq t\leq T-ub_n^{-1}}u^qP[N^i_{t+ub_n^{-1}}-N^i_t=0]<\infty.
\end{equation*}
\end{description}

For example, let $X\equiv Y$, $\{b_n\}$ is a positive integer valued sequence, $\{\bar{N}^1_t\},\{\bar{N}^2_t\}$ be two independent homogeneous Poisson processes with positive intensities $\lambda_1,\lambda_2$ respectively,
and stochastic processes $\{N^1_t\},\{N^2_t\}$ satisfy $N^i_t=\bar{N}^i_{b_nt}, \ (i=1,2)$.
Then $[B1\mathchar`-q]$ obviously holds for any $q>0$.
Moreover, $[B2\mathchar`-q]$ holds for any $q>0$ since 
\begin{eqnarray}
\lim_{u \to \infty}u^q\sup_{i=1,2}\sup_{n\in \mathbb{N}}\sup_{0\leq t \leq T-ub_n^{-1}}P[N^i_{t+ub_n^{-1}}-N^i_t=0]=\lim_{u \to \infty}u^qe^{-(\lambda_1\wedge \lambda_2)u}=0. \nonumber
\end{eqnarray}

We will investigate sufficient conditions of $[A3'\mathchar`-q,\eta]$. 
First, we denote $t_k=Tk/[b_n] \ (0\leq k\leq [b_n])$, 
$\mathcal{G}^n_{j,k}=\sigma(N^i_t-N^i_s;t_j\leq s<t\leq t_k,i=1,2) \ (0\leq j<k\leq [b_n])$, $\alpha^n_0=1/4$ and
\begin{equation}\label{alpha-def}
\alpha^n_k=0\vee \sup_{1\leq j_1,j_2\leq [b_n]-1,j_2-j_1\geq k}\sup_{A\in\mathcal{G}^n_{0,j_1}}\sup_{B\in\mathcal{G}^n_{j_2,[b_n]}}|P(A\cap B)-P(A)P(B)|
\end{equation}
for $k\in\mathbb{N}$.

Let $\zeta^{p,i}_n$ be measures which satisfy $\zeta^{p,i}_n([s,t))=E[\nu^{p,i}_n([s,t))]$.
Moreover, for $n_0\in \mathbb{N}$, $\epsilon'>0$, a Lebesgue integrable function $g:[0,T]\mapsto \mathbb{R}$, 
and a continuous function $f:[0,T]\mapsto \mathbb{R}$, we define
\begin{equation*}
\bar{\Psi}^{p,i}_{n_0,\epsilon'}(f;g)=\sup_{n\geq n_0}b_n^{\epsilon'}\bigg|\int^T_0f_t\zeta^{p,i}_n(dt)-\int^T_0f_tg_tdt\bigg|.
\end{equation*}

\begin{proposition}\label{p-lim}
Let $q\in 2\mathbb{N}$, $q > 2$, $\epsilon \in (0,1)$, $\delta >0$ and $\beta\in (0,1/2-1/q)$. Assume that $[B1\mathchar`-(q(1+\delta))]$, $[B2\mathchar`-(q\epsilon)]$ hold,
$E[(N^1_T+N^2_T)^q]<\infty$ for $n\in\mathbb{N}$ and there exists $n_0\in \mathbb{N}$ such that
\begin{equation}\label{alpha-coeff-est}
\mathcal{S}=\sup_{n\geq n_0}\sum_{k=0}^{\infty}(k+1)^{q-2+(q-1)/\delta}\alpha^n_k<\infty.
\end{equation}
Moreover, assume that there exist $\epsilon'>0$, $C>0$, and left-continuous deterministic functions $a_0(t),c_0(t),a_p(t) \ (p\in \mathbb{N})$
such that $\int^T_0a_p(t)dt < \infty \ (p\in \mathbb{Z}_+)$, $\int^T_0c_0(t)dt < \infty$ and  
\begin{equation}\label{a2-lim}
\bar{\Psi}^{0,1}_{n_0,\epsilon'}(f;a_0)\vee \bar{\Psi}^{0,2}_{n_0,\epsilon'}(f;c_0)\vee \max_{i=1,2} \sup_{p\in\mathbb{N}}\frac{\bar{\Psi}^{p,i}_{n_0,\epsilon'}(f;a_p)}{(p+1)^C}\leq C(\sup_t|f_t|+\omega_{\beta}(f))
\end{equation}
for any $\beta$-H${\rm \ddot{o}}$lder continuous function $f:[0,T]\to \mathbb{R}$.
Then $[A3'\mathchar`-q,\eta]$ holds for $\eta=\epsilon'\wedge \beta \wedge (\delta\epsilon/(2(1+\delta-\delta\epsilon)))$ with $\alpha$ in $[A3'\mathchar`-q,\eta]$ equals $\beta$.
\end{proposition}

For example, let $\{\bar{N}^i_t\}_{t\geq 0}$ be a point process where the distribution of $(\bar{N}^i_{t+t_k}-\bar{N}^i_{t+t_{k-1}})_{k=1}^M$
does not depend on $t\geq 0$ for $i=1,2$, $M\in\mathbb{N}$ and $0\leq t_0<t_1<\cdots <t_M$.
Moreover, let $\{N^i_t\}_t$ satisfy $N^i_t=\bar{N}^i_{[b_n]t}$ for $t\in [0,T]$, $i=1,2$ and $n\in\mathbb{N}$.
Then the relation (\ref{a2-lim}) holds 
if $[B1\mathchar`-(2)]$ and $[B2\mathchar`-\epsilon]$ hold for some $\epsilon \in (0,2]$. 
In this case, we obtain $a_p=T^{-1}\lim_{n\to \infty}E[\nu^{p,1}_n([0,T))]$, $c_0=T^{-1}\lim_{n\to \infty}E[\nu^{0,2}_n([0,T))]$, and $\epsilon'=(\epsilon/4)\wedge \beta$.
In particular, $\{a_p\}_{p\in\mathbb{Z}_+},c_0$ are constants.

For general $\{N^i_t\}$, the following propositions are sufficient conditions 
for $[A4\mathchar`-q,\delta],[A2\mathchar`-q,\delta]$ and the estimate of ${\rm ess \, inf}_t a_1(t)$ in Proposition \ref{sep-cond}. 
\begin{proposition}\label{A4q-suff}
Let $q\in 2\mathbb{N}$, $q>2$, $p'_1,p'_2>1$, $1/p'_1+1/p'_2=1$.
Assume $[B1\mathchar`-(p'_1q)]$ and $[B2\mathchar`-(p'_2(q+2))]$.
\begin{enumerate}
\item Then there exists $n_0\in\mathbb{N}$ such that $\sup_{n\geq n_0}E[(\Phi_p)^q]\leq C(p+1)^{q+1}$ for any $p\in\mathbb{Z}_+$.
In particular, $[A4\mathchar`-q',(1+3/q')] \ 1$ holds for any $q'\in [2,q)$. 
\item Then there exists $n_0\in\mathbb{N}$ such that
\begin{equation*}
\sup_{n\geq n_0}E[(\bar{\Phi}_{p_1,p_2})^{\frac{q}{2}}]\leq C(p_1+1)^{\frac{q}{2}+1}(p_2+1)^{\frac{q}{2}+1}
\end{equation*}
for $p_1,p_2\in\mathbb{Z}_+$.
In particular, $[A4\mathchar`-q,3] \ 2$ holds.
\end{enumerate}
\end{proposition}
\begin{proposition}\label{A2q-suff}
Let $q\in\mathbb{N}$ and we assume $[B2\mathchar`-(q+1)]$. Then there exists $n_0\in\mathbb{N}$ such that
$\sup_{n\geq n_0}E\left[b_n^{q-1}r_n^q\right]<\infty$.
\end{proposition}

\begin{proposition}\label{inf-a1-suff}
Assume there exists $n_1\in\mathbb{N}$ and $q>0$ such that $[A3']$ and $[B2\mathchar`-q]$ hold, $a_1(t)$ does not depend on $t$,
$\{N^i_{t+[b_n]^{-1}T}-N^i_t\}_{0\leq t\leq T-[b_n]^{-1}T, n\geq n_1, i=1,2}$ is tight and
$\alpha$-mixing coefficients $\{\alpha^n_k\}_k$ defined by $(\ref{alpha-def})$ satisfy
$\sup_{n\geq n_1}\sum_{k=1}^{\infty}k\alpha^n_k<\infty$. 
Then there exists a constant $\delta>0$ such that $a_1\geq \delta$ almost surely.
\end{proposition}

Finally, we state a corollary of main theorems.

We assume $\{\bar{N}^i_t\}_{t\geq 0}$ is a exponential $\alpha$-mixing point process where $E[(\bar{N}^1_1+\bar{N}^2_1)^q]<\infty$
for $q>0$ and the distribution of $(\bar{N}^i_{t+t_k}-\bar{N}^i_{t+t_{k-1}})_{k=1}^M$
does not depend on $t\geq 0$ for $i=1,2$, $M\in\mathbb{N}$ and $0\leq t_0<t_1<\cdots <t_M$
for $1\leq i\leq 2$.
Let $\{N^i_t\}$ satisfy $N^i_t=\bar{N}^i_{[b_n]t}$ for $t\in [0,T]$, $i=1,2$ and $n\in\mathbb{N}$, $\hat{\sigma}_n$ be the quasi-maximum-likelihood estimator defined by $H_n$, 
$\pi:\Lambda\to\mathbb{R}_+$ be a continuous function
and $\tilde{\sigma}_n$ be defined by (\ref{bayes-def}).

\begin{corollary}\label{cor-main}
Assume $[A1],[H0],[B1\mathchar`-q],[B2\mathchar`-q]$ for any $q>2$. 
\begin{enumerate}
\item Then $\hat{\sigma}_n\to^p \sigma_{\ast}$, 
$b_n^{1/2}(\hat{\sigma}_n-\sigma_{\ast})\to^{s\mathchar`-\mathcal{L}} \Gamma^{-1/2}\mathcal{N}$ 
and $E[Y'f(b_n^{1/2}(\hat{\sigma}_n-\sigma_{\ast}))]\to E[Y'f(\Gamma^{-1/2}\mathcal{N})]$
as $n\to \infty$ for any continuous function $f$ of at most polynomial growth and any bounded random variable $Y'$ on $(\Omega, \mathcal{F})$. 
\item Assume $0<\inf_{\sigma}\pi(\sigma)\leq \sup_{\sigma}\pi(\sigma)<\infty$. Then $\tilde{\sigma}_n\to^p \sigma_{\ast}$, 
$b_n^{1/2}(\tilde{\sigma}_n-\sigma_{\ast})\to^{s\mathchar`-\mathcal{L}} \Gamma^{-1/2}\mathcal{N}$ 
and $E[Y'f(b_n^{1/2}(\tilde{\sigma}_n-\sigma_{\ast}))]\to E[Y'f(\Gamma^{-1/2}\mathcal{N})]$
as $n\to \infty$ for any continuous function $f$ of at most polynomial growth 
and any bounded random variable $Y'$ on $(\Omega,\mathcal{F})$. 
\end{enumerate}
\end{corollary}

\begin{example}
We consider a simple model with deterministic diffusion coefficients:
\begin{equation*}
\left\{
\begin{array}{lll}
dY^1_t&=&\sigma_1dW^1_t \\
dY^2_t&=&\sigma_3dW^1_t+\sigma_2dW^2_t \\
(Y^1_0,Y^2_0)&=&{\colorb (0,0)} \\
\end{array}
\right.
\end{equation*}
where $\epsilon>0$, $R'>\epsilon$ and $\sigma=(\sigma_1,\sigma_2,\sigma_3)\in \Lambda=(\epsilon,R')\times (\epsilon,R')\times (-R',R')$.
Let $\{\bar{N}^1_t\},\{\bar{N}^2_t\}$ be two independent homogeneous Poisson processes with positive intensities $\lambda_1,\lambda_2$ respectively
and point processes $\{N^1_t\},\{N^2_t\}$ which generate observations satisfy $N^{i,n}_t=\bar{N}^i_{nt} \ (i=1,2)$. 

Then we can easily check $[A1],[B1\mathchar`-q],[B2\mathchar`-q]$ hold for any $q>2$.
Since $(x+y)^2\geq x^2/2-y^2$ for $x,y\in\mathbb{R}$, we have
\begin{eqnarray}
|bb^{\star}(\sigma)-bb^{\star}(\bar{\sigma})|^2 
&\geq & (\sigma_1^2-\bar{\sigma}_1^2)^2+\frac{\epsilon^2}{(R')^2}(\sigma_1\sigma_3-\bar{\sigma}_1\bar{\sigma}_3)^2+\frac{\epsilon^4}{16(R')^4}(\sigma_2^2+\sigma_3^2-\bar{\sigma}_2^2-\bar{\sigma}_3^2)^2 \nonumber \\
&\geq & 4\epsilon^2(\sigma_1-\bar{\sigma}_1)^2+\frac{\epsilon^2}{(R')^2}\bigg\{\frac{\bar{\sigma}_1^2(\sigma_3-\bar{\sigma}_3)^2}{2}-\sigma_3^2(\sigma_1-\bar{\sigma}_1)^2\bigg\} \nonumber \\
&&+\frac{\epsilon^4}{16(R')^4}\bigg\{\frac{(\sigma_2^2-\bar{\sigma}_2^2)^2}{2}-(\sigma_3^2-\bar{\sigma}_3^2)^2\bigg\}\geq \frac{\epsilon^6}{8(R')^4}|\sigma-\bar{\sigma}|^2 \nonumber 
\end{eqnarray}
for $\sigma,\bar{\sigma}\in\Lambda$.
Then by Remark \ref{H0-suff}, we obtain $[H0]$.

$H_n$ can be written as 
\begin{eqnarray}
H_n(\sigma)&=&-\frac{1}{2}\bigg(\bigg(\frac{Y^1(I^i)}{\sqrt{|I^i|}}\bigg)_i^{\star},\bigg(\frac{Y^2(J^j)}{\sqrt{|J^j|}}\bigg)_j^{\star}\bigg)\left(
\begin{array}{ll}
\sigma_1^2\mathcal{E}_{l_n} & \sigma_1\sigma_3G \\
\sigma_1\sigma_3G^{\star} & (\sigma_2^2+\sigma_3^2)\mathcal{E}_{m_n} \\
\end{array}
\right)^{-1} \nonumber \\
&& \times \bigg(\bigg(\frac{Y^1(I^i)}{\sqrt{|I^i|}}\bigg)_i^{\star},\bigg(\frac{Y^2(J^j)}{\sqrt{|J^j|}}\bigg)_j^{\star}\bigg)^{\star}
-\frac{1}{2}\log\det\left(
\begin{array}{ll}
\sigma_1^2\mathcal{E}_{l_n} & \sigma_1\sigma_3G \\
\sigma_1\sigma_3G^{\star} & (\sigma_2^2+\sigma_3^2)\mathcal{E}_{m_n} \\
\end{array}
\right). \nonumber
\end{eqnarray}
By calculating $\sigma$ which maximize $H_n$, we obtain the quasi-maximum likelihood estimator $\hat{\sigma}_n=(\hat{\sigma}_{1,n},\hat{\sigma}_{2,n},\hat{\sigma}_{3,n})$.
By Corollary \ref{cor-main}, we have $\hat{\sigma}_n\to^p \sigma_{\ast}$, $\sqrt{n}(\hat{\sigma}_n-\sigma_{\ast})\to^dN(0,\Gamma^{-1})$
as $n\to \infty$, where $\sigma_{\ast}=(\sigma_{1,\ast},\sigma_{2,\ast},\sigma_{3,\ast})$ is the true value.
In this case, $\rho=\rho_t(\sigma_{\ast})$ can be written as $\rho=\sigma_{3,\ast}/\sqrt{\sigma_{2,\ast}^2+\sigma_{3,\ast}^2}$, 
$\{a_p\}_{p=0}^{\infty},\{c_p\}_{p=0}^{\infty}$ in $[A3']$ become constants 
and $a,c$ in $[A3]$ can be written as $a(\rho')-a_0=c(\rho')-c_0=\mathcal{A}(\rho')=\sum_{p=1}^{\infty}a_p\times (\rho')^{2p}$ for $\rho'\in (-1,1)$.
If $\rho\neq 0$, $\Gamma$ and $\Gamma^{-1}$ can be calculated by using Proposition \ref{st-conv} later as 
\begin{equation*}
\Gamma=T\left(
\begin{array}{lll}
\frac{a_0+a}{\sigma_{1,\ast}^2} & 0 & -\frac{\mathcal{A}}{\sigma_{1,\ast}\sigma_{3,\ast}} \\
0 & \frac{2c(1-\rho^2)^2+\partial_{\rho}\mathcal{A}\rho(1-\rho^2)^2}{\sigma_{2,\ast}^2} & \frac{2c\rho^2(1-\rho^2)-\partial_{\rho}\mathcal{A}\rho(1-\rho^2)^2}{\sigma_{2,\ast}\sigma_{3,\ast}} \\
-\frac{\mathcal{A}}{\sigma_{1,\ast}\sigma_{3,\ast}} & \frac{2c\rho^2(1-\rho^2)-\partial_{\rho}\mathcal{A}\rho(1-\rho^2)^2}{\sigma_{2,\ast}\sigma_{3,\ast}} & \frac{-\mathcal{A}+2c\rho^4+\partial_{\rho}\mathcal{A}\rho(1-\rho^2)^2}{\sigma_{3,\ast}^2} \\
\end{array}
\right), \nonumber 
\end{equation*}
\begin{equation*}
\Gamma^{-1}=\frac{1}{T\{4ac\mathcal{A}+2\partial_{\rho}\mathcal{A}\rho(a_0c+c_0a)\}}{\rm diag}(\{\sigma_{1,\ast},\sigma_{2,\ast},\sigma_{3,\ast}\})\mathcal{P}{\rm diag}(\{\sigma_{1,\ast},\sigma_{2,\ast},\sigma_{3,\ast}\}),
\end{equation*}
where 
\begin{equation*}
\mathcal{P}=\left(
\begin{array}{lll}
-2c\mathcal{A}+(c_0+c)\partial_{\rho}\mathcal{A}\rho & \mathcal{A}\{-\frac{2c\rho^2}{1-\rho^2}+\partial_{\rho}\mathcal{A}\rho\} & \mathcal{A}(2c+\partial_{\rho}\mathcal{A}\rho) \\
\mathcal{A}\{-\frac{2c\rho^2}{1-\rho^2}+\partial_{\rho}\mathcal{A}\rho\} & \frac{-2a\mathcal{A}+(a_0+a)\{2c\rho^4+\partial_{\rho}\mathcal{A}\rho(1-\rho^2)^2\}}{(1-\rho^2)^2} & (a_0+a)\{-\frac{2c\rho^2}{1-\rho^2}+\partial_{\rho}\mathcal{A}\rho\} \\
\mathcal{A}(2c+\partial_{\rho}\mathcal{A}\rho) & (a_0+a)\{-\frac{2c\rho^2}{1-\rho^2}+\partial_{\rho}\mathcal{A}\rho\} & (a_0+a)(2c+\partial_{\rho}\mathcal{A}\rho) \\
\end{array}
\right) \nonumber 
\end{equation*}
and $\rho'=\rho$ is substituted for $a,c,\mathcal{A},\partial_{\rho}\mathcal{A}$.

We can see the estimator $\hat{\sigma}_{1,n}\hat{\sigma}_{3,n}T$ for the cross variation 
$\langle Y^1,Y^2\rangle_T=\sigma_{1,\ast}\sigma_{3,\ast}T$ also has consistency and
{\colorb
\begin{equation*}
\sqrt{n}(\hat{\sigma}_{1,n}\hat{\sigma}_{3,n}T-\langle Y^1,Y^2\rangle_T)\to^d N(0,v)
\end{equation*}
as $n\to \infty$ by using the delta method, where
\begin{equation*}
\quad v=T\sigma_{1,\ast}^2\sigma_{3,\ast}^2\frac{2a(\rho)c(\rho)+\partial_{\rho}\mathcal{A}(\rho)\rho(a(\rho)+c(\rho))}{-2a(\rho)c(\rho)\mathcal{A}(\rho)+\partial_{\rho}\mathcal{A}(\rho)\rho(a_0c(\rho)+c_0a(\rho))}.
\end{equation*}
}

{\colorb
By using the result in Hayashi and Yoshida \cite{hay-yos02}, we can calculate the asymptotic variance of estimation error of the Hayashi-Yoshida estimator ${\rm HY}_n$.
In the settings in this example, we obtain
\begin{equation*}
\sqrt{n}({\rm HY}_n-\langle Y^1,Y^2\rangle_T)\to^d N(0,v_0)
\end{equation*}
as $n\to\infty$, where
\begin{equation*}
\quad v_0=T\sigma_{1,\ast}^2\sigma_{3,\ast}^2\bigg\{(1+\rho^{-2})\bigg(\frac{2}{\lambda_1}+\frac{2}{\lambda_2}\bigg)-\frac{2}{\lambda_1+\lambda_2}\bigg\}.
\end{equation*}

We also simulate $\hat{\sigma}_n,\hat{\sigma}_{1,n}\hat{\sigma}_{3,n}T,{\rm HY}_n$ for some values of parameters.
Table $1$ represents the results.
We can see that each estimators work well and sample standard deviation of $\hat{\sigma}_{1,n}\hat{\sigma}_{3,n}T$ is about two-thirds of that of ${\rm HY}_n$.
The lowest two rows represent numerical calculation results of asymptotic standard deviation of estimators
and we can find these values are close to sample standard deviation of estimators.


	\begin{table}
	\caption{Sample means of estimators for $10,000$ independent simulated sample path. $T=1,(\lambda_1,\lambda_2)=(1,1)$. 
	The left table represents the result for $(\sigma_1,\sigma_2,\sigma_3)=(1,1,0.5)$ and the right table represents the result for $(\sigma_1,\sigma_2,\sigma_3)=(0.5,2,1)$.
	Sample standard deviation are given in parentheses.}
	\footnotesize
	\begin{tabular}{|c|c|c|c|c|}
	\hline
	$n$ & $ $ & $50 $ & $100 $ & $500 $ \\ \hline
	 & true value & $ $ & $ $ & $ $ \\ \hline
	$\hat{\sigma}_{1,n}$ & $1 $ & $0.994 $ & $0.998 $ & $0.999 $ \\
	 & $ $ & $(0.102) $ & $(0.070) $ & $(0.031) $ \\ \hline
	$\hat{\sigma}_{2,n}$ & $1 $ & $0.968 $ & $0.983 $ & $0.996 $ \\
	 & $ $ & $(0.129) $ & $(0.091) $ & $(0.040) $ \\ \hline
	$\hat{\sigma}_{3,n}$ & $0.5 $ & $0.499 $ & $0.502 $ & $0.5 $ \\
	 & $ $ & $(0.224) $ & $(0.154) $ & $(0.067) $ \\ \hline
	$\hat{\sigma}_{1,n}\hat{\sigma}_{3,n}T$ & $0.5 $ & $0.5 $ & $0.503 $ & $0.5 $ \\
	 & $ $ & $(0.238) $ & $(0.165) $ & $(0.071) $ \\ \hline
	${\rm HY}_n$ & $0.5 $ & $0.501 $ & $0.504 $ & $0.5 $ \\
	 & $ $ & $(0.336) $ & $(0.236) $ & $(0.106) $ \\ \hline
	$\sqrt{v/n}$ & $ $ & $0.228 $ & $0.161 $ & $0.072 $ \\ \hline
	$\sqrt{v_0/n}$ & $ $ & $0.339 $ & $0.239 $ & $0.107 $ \\ \hline
	\end{tabular} 
	\begin{tabular}{|c|c|c|c|c|}
	\hline
	$n$ & $ $ & $50 $ & $100 $ & $500 $ \\ \hline
	 & true value & $ $ & $ $ & $ $ \\ \hline
	$\hat{\sigma}_{1,n}$ & $0.5 $ & $0.497 $ & $0.499 $ & $0.499 $ \\
	 & $ $ & $(0.050) $ & $(0.035) $ & $(0.015) $ \\ \hline
	$\hat{\sigma}_{2,n}$ & $2 $ & $1.936 $ & $1.968 $ & $1.995 $ \\
	 & $ $ & $(0.259) $ & $(0.181) $ & $(0.079) $ \\ \hline
	$\hat{\sigma}_{3,n}$ & $1 $ & $0.986 $ & $0.996 $ & $0.997 $ \\
	 & $ $ & $(0.449) $ & $(0.307) $ & $(0.135) $ \\ \hline
	$\hat{\sigma}_{1,n}\hat{\sigma}_{3,n}T$ & $0.5 $ & $0.495 $ & $0.499 $ & $0.498 $ \\
	 & $ $ & $(0.239) $ & $(0.164) $ & $(0.072) $ \\ \hline
	${\rm HY}_n$ & $0.5 $ & $0.498 $ & $0.499 $ & $0.498 $ \\
	 & $ $ & $(0.335) $ & $(0.237) $ & $(0.108) $ \\ \hline
	$\sqrt{v/n}$ & $ $ & $0.228 $ & $0.161 $ & $0.072 $ \\ \hline
	$\sqrt{v_0/n}$ & $ $ & $0.339 $ & $0.239 $ & $0.107 $ \\ \hline
	\end{tabular} 
	\end{table}
}
\end{example}

\section{Proofs}
\subsection{Proof of Proposition \ref{A3eq}}
{\bf Proof of Proposition \ref{A3eq}.}

\noindent
$[A3']\Rightarrow [A3]$:

Since $\nu^{p,1}_n([0,t))\leq b_n^{-1}l_n$, 
$\{b_n^{-1}l_n\}_{n\in\mathbb{N}}$ is tight, $\int^t_0a_p(s)ds\leq \int^t_0a_0(s)ds$ and $\nu^{p,1}_n([0,t))$ converges to $\int^t_0a_p(s)ds$ in probability by $[A3']$ for $p\in \mathbb{Z}_+$,
we have $\sum_{p=0}^{\infty}z^{2p}\int^t_0a_p(s)ds<\infty$ and 
\begin{equation*}
b_n^{-1}A_n(\mathcal{E}^1(t),{\bf 0},{\bf 0},z{\bf 1})=\sum_{p=0}^{\infty}z^{2p}\nu^{p,1}_n([0,t)) \to^p \sum_{p=0}^{\infty}z^{2p}\int^t_0a_p(s)ds
\end{equation*}
for any $t\in(0,T]$ and $z\in\mathbb{C},|z|<1$ by Lemma \ref{sum-conv}.
\\

\noindent
$[A3]\Rightarrow [A3']$:

Fix $t\in (0,T]$. Let $\{f_n\}$ be functions on $\{z\in \mathbb{C};|z|<1\}$ satisfying
\begin{eqnarray}
f_n(z)=b_n^{-1}A_n(\mathcal{E}^1(t),{\bf 0},{\bf 0},z{\bf 1}) =\sum_{p=0}^{\infty}z^{2p}\nu_n^{p,1}([0,t)). \nonumber
\end{eqnarray}
Then since $\nu^{p,1}_n([0,t))\leq b_n^{-1} l_n$, the power series in the right-hand side converges absolutely on $\{|z|<1\}$.
Consequently, $f_n$ is the holomorphic function.
Then we have
\begin{equation}\label{GGTp-int}
\nu^{p,1}_n([0,t))=\frac{1}{2\pi i}\int_{|z|=\eta/3}\frac{f_n(z)}{z^{2p+1}}dz.
\end{equation}

Let $f(z)=\int^t_0a(z,s)ds$. 
Since
\begin{equation}\label{fn-ub}
|f_n(z)|\leq \frac{l_n}{b_n}\frac{1}{1-(2/3)^2}=\frac{9l_n}{5b_n}
\end{equation}
 on $\{|z|\leq 2/3\}$, 
$\{\sup_{|z|\leq 2/3} |f_n(z)|\}$ are tight.
Moreover, since $f_n(z)\to^p f(z) \ (n\to \infty)$ for $z\in\mathbb{C},|z|<\eta$, 
$\sup_{|z|\leq \eta/2} |f(z)|\leq (9/5)\int^T_0a_0(t)dt<\infty$, almost surely.
Therefore $\{\sup_{|z|\leq \eta/2} |f_n(z)-f(z)|\}$ are also tight. 

Let $\Gamma: |z|=2/3$. For any $z_1,z_2\in\{|z|<1/2\}$, the Cauchy integral formula gives
\begin{eqnarray}
2\pi|f_n(z_1)-f_n(z_2)|&= &\bigg|\int_{\Gamma}\left(\frac{f_n(\xi)}{\xi-z_1}-\frac{f_n(\xi)}{\xi-z_2}\right)d\xi\bigg|=\bigg|\int_{\Gamma}\frac{f_n(\xi)(z_1-z_2)}{(\xi-z_1)(\xi-z_2)}d\xi\bigg| \nonumber \\
&\leq & |z_1-z_2|\cdot 2\pi\cdot \frac{2}{3}\cdot 6^2\cdot\sup_{|z|\leq 2/3}|f_n(z)| \leq Cb_n^{-1}l_n|z_1-z_2|.\nonumber 
\end{eqnarray}
By the convergence $f_n(z)\to^p f(z) \ (n\to \infty)$ for $z\in\mathbb{C},|z|<\eta$, we obtain
\begin{eqnarray}
|f(z_1)-f(z_2)|&\leq &C |z_1-z_2|\int^T_0a_0(s)ds \quad {\rm a.s.}\nonumber
\end{eqnarray}
for $z_1,z_2\in\{z;|z|\leq \eta/2\}$.
Then for any $\epsilon>0$, tightness of $\{b_n^{-1}l_n\}$ gives
\begin{equation*}
\lim_{\eta'\to 0}\sup_nP\bigg[\sup_{z_1,z_2\in \{|z|\leq \eta/2 \},|z_1-z_2|<\eta'}|(f_n-f)(z_1)-(f_n-f)(z_2)|>\epsilon\bigg]=0.
\end{equation*}
Then by the tightness of $\{\sup_{|z|\leq \eta/2}|f_n(z)-f(z)|\}_n$ and tightness criterion in $C$ space in Billingsley \cite{bil} 
which can be extended to the one in $C(\{|z|\leq \eta/2\})$, $\{f_n-f\}_n$ is tight in $C(\{|z|\leq \eta/2\})$.
Therefore, since $f_n(z)\to^p f(z)$ as $n\to \infty$,
we see that $\{f_n-f\}$ converge in probability to $0$ in $C(\{|z|\leq \eta/2\})$.
Therefore by (\ref{GGTp-int}), we have 
\begin{equation*}
\nu^{p,1}_n([0,t))\to^p \frac{1}{2\pi i}\int_{|z|=\eta/3}\frac{f(z)}{z^{2p+1}}dz
\end{equation*}
as $n\to \infty$ for $p\geq 1$.

By the equation $f(z)=\int^t_0a(z,s)ds$ and Fubini's theorem, there exists $a_p(s)$ such that $\int^T_0a_p(s)ds<\infty$ and
$\nu^{p,1}_n([0,t))\to^p \int^t_0 a_p(s)ds$
as $n\to \infty$. We thus get $[A3']$.

Moreover, under $[A2]$ and $[A3]$, the above proof gives the relations between $a,c$ and $\{a_p\},\{c_p\}$ in the statement. 
The rest of the proof is easy since 
\begin{eqnarray}
&&b_n^{-1}A_n(x^2\mathcal{E}^1(t),y^2\mathcal{E}^2(t),xy\rho_{\ast}\mathcal{E}^1(t)G,\rho{\bf 1}) \nonumber \\
&=&\sum_{p=0}^{\infty}\rho^{2p}\{x^2\nu^{p,1}_n([0,t))+y^2\nu^{p,2}_n([0,t))-2xy\rho_{\ast}\rho\nu^{p+1,1}_n([0,t))\}. \nonumber
\end{eqnarray}
\qed

\subsection{Proof of Proposition \ref{Hn-lim}}
To prove Proposition \ref{Hn-lim}, we use some Lemmas. 
First, let 
\begin{eqnarray}
R&=&\max_{0\leq i\leq 4}\max_{0\leq j\leq 3}\max_{1\leq p\leq 2}\sup_{\sigma\in \Lambda,t_k\in [0,T] \ (1\leq k\leq n_2)}\bigg(| \partial_x^j\partial_{\sigma}^ib^p|\vee \bigg|\partial_x^j\partial_{\sigma}^i\frac{1}{| b^p|}\bigg|(\{X^k_{t_k}\}_{k=1}^{n_2},\sigma)\bigg) \nonumber \\
&&\vee \max_{1\leq i\leq 3}\sup_{t\in [0,T]}|\tilde{b}^i_t|\vee \max_{2\leq i\leq 3,1\leq j\leq 3}\sup_{t\in [0,T]}|\hat{b}^{ij}_t|, \nonumber
\end{eqnarray} 
then $R\geq 1$ and $E[R^q]<\infty$ for any $q>0$ under $[A1]$.


\begin{lemma}\label{q-sum-est}
Let $q\in\mathbb{N}$, $M\in \mathbb{N}\cup \{\infty\}$, $(\Omega',\mathcal{F}',P')$ be a probability space, $\{F_j\}_{j=1}^M$ be random variables, and $\mathcal{G}$ be a sub $\sigma$-field of $\mathcal{F}'$. 
\begin{enumerate}
\item Then $E'\big[\big|\sum_{j=1}^MF_j\big|^q\big|\mathcal{G}\big]\leq \big(\sum_{j=1}^ME'[|F_j|^q|\mathcal{G}]^{\frac{1}{q}}\big)^q$, where $E'$ denotes expectation with respect to $P'$.
\item We assume $q\in 2\mathbb{N}$ and $\{\sum_{j=1}^kF_j\}_{0\leq k\leq M}$ is martingale with respect to some filtration. 
Then there exists a constant $C_q>0$ which depends only $q$, such that 
$E'\big[\big|\sum_{j=1}^MF_j\big|^q\big]\leq C_q\big(\sum_{j=1}^ME'[|F_j|^q]^{\frac{2}{q}}\big)^{\frac{q}{2}}.$
\end{enumerate}
\end{lemma}

\begin{proof}
We expand the summation and use H${\rm \ddot{o}}$lder's inequality. 
\begin{eqnarray}
E'\big[\big|\sum_{j=1}^MF_j\big|^q\big|\mathcal{G}\big]&\leq &\sum_{j_1,\ldots,j_q=1}^ME'[|F_{j_1}\ldots F_{j_q}||\mathcal{G}]\leq \sum_{j_1,\ldots,j_q=1}^ME'[|F_{j_1}|^q|\mathcal{G}]^{\frac{1}{q}}\ldots E'[|F_{j_q}|^q|\mathcal{G}]^{\frac{1}{q}} \nonumber \\
&\leq & \Big(\sum_{j=1}^ME'[|F_j|^q|\mathcal{G}]^{\frac{1}{q}}\Big)^q. \nonumber
\end{eqnarray}
For $2$., we use the Burkholder-Devis-Gundy inequality and apply $1$. for $\mathcal{G}=\{\emptyset, \Omega'\}$.
\end{proof}

\begin{lemma}\label{Fp-est}
Let $\{G_p\}_{p\in\mathbb{Z}_+}$ be a sequence of positive numbers,
$a\in\mathbb{N}, b,r,s\in\mathbb{Z}_+$ and $\rho\in[0,1)$.
Then there exists a constant $C>0$ which depends only on $a,b,r,s$ such that
\begin{equation*}
\sum_{p=0}^{\infty}\rho^{a(p-b)\vee 0}(p+1)^sG_p\leq C(1-\rho)^{-(s+\frac{r+1}{2})}\bigg(\sum_{p=0}^{\infty}\frac{G_p^2}{(p+1)^r}\bigg)^{\frac{1}{2}}.
\end{equation*} 
\end{lemma}

\begin{proof}
By the Cauchy-Schwarz inequality, we have
\begin{equation*}
\sum_{p=0}^{\infty}\rho^{a(p-b)\vee 0}(p+1)^sG_p\leq \bigg(\sum_{p=0}^{\infty}\rho^{2a(p-b)\vee 0}(p+1)^{2s+r}\bigg)^{\frac{1}{2}}\bigg(\sum_{p=0}^{\infty}\frac{G^2_p}{(p+1)^r}\bigg)^{\frac{1}{2}}.
\end{equation*}
Then the conclusion follows since 
\begin{equation*}
\sum_{p=0}^{\infty}\rho^{2a(p-b)\vee 0}(p+1)^{2s+r}\leq C+C\sum_{p=0}^{\infty}\rho^{2ap}(p+1)^{2s+r}\leq C+C\frac{(2s+r)!}{(1-\rho^{2a})^{2s+r+1}}.
\end{equation*}
\end{proof}

\begin{lemma}\label{prob-conv}
Let $(\Omega',\mathcal{F}',P')$ be probability space and $\{\mathcal{G}_n\}_{n\in\mathbb{N}}\subset \mathcal{F}'$ be sub $\sigma$-fields.
\begin{enumerate}
\item Let $\{X'_n\}_{n\in\mathbb{N}}\subset L^1(\Omega')$. 
Assume $E'[|X'_n||\mathcal{G}_n]\to^p 0 \ (n\to \infty)$. 
Then $X'_n\to^p 0 \ ( n\to \infty)$. 
\item Let $d_1,d_2\in\mathbb{N}$, $p>d_1$, $\Lambda'\subset \mathbb{R}^{d_1}$ be a bounded open set and $X'_n : \Omega' \to C^1(\Lambda';\mathbb{R}^{d_2})$ be random variables $(n\in\mathbb{N})$.
Assume that $\Lambda'$ satisfies Sobolev's inequality, $\{\sup_{\sigma\in\Lambda'}|X'_n(\sigma)|^p\vee |\partial_{\sigma}X'_n(\sigma)|^p\}_{n\in\mathbb{N}}\subset L^1(\Omega')$ and 
$\sup_{\sigma\in\mathbb{Q}^{d_1}\cap\Lambda'}E'[|\partial_{\sigma}X'_n(\sigma)|^p\vee |X'_n(\sigma)|^p|\mathcal{G}_n]\to^p 0$ as $n\to \infty$.
Then $\sup_{\sigma\in\Lambda'}|X'_n(\sigma)|\to^p 0$ as $n\to \infty$.
\end{enumerate}
\end{lemma}
\begin{proof}
$1$. For any $\epsilon, \delta >0$, there exists $N\in\mathbb{N}$ such that $P'[E'[|X'_n||\mathcal{G}_n]\geq \epsilon\delta/2]<\epsilon/2$ for $n\geq N$.
Therefore, for $n\geq N$, we have
\begin{eqnarray}
P'[|X'_n|\geq \delta]&\leq & P'[E'[|X'_n||\mathcal{G}_n]\geq \epsilon \delta/2]+P'[|X'_n|\geq \delta, E'[|X'_n||\mathcal{G}_n]<\epsilon\delta/2] \nonumber \\
&\leq & \frac{\epsilon}{2}+\frac{1}{\delta}E'[|X'_n|,E'[|X'_n||\mathcal{G}_n]<\epsilon\delta/2] \nonumber \\
&=& \frac{\epsilon}{2}+\frac{1}{\delta}E'[E'[|X'_n||\mathcal{G}_n],E'[|X'_n||\mathcal{G}_n]<\epsilon\delta/2]\leq \epsilon. \nonumber 
\end{eqnarray}
\\
$2$. First, by Sobolev's inequality, we have
\begin{eqnarray}
E'\bigg[\bigg(\sup_{\sigma\in\Lambda'}|X'_n|\bigg)^p\bigg|\mathcal{G}_n\bigg]\leq CE'\left[\int_{\Lambda'}|\partial_{\sigma}X'_n|^pd\sigma \bigg|\mathcal{G}_n\right]+CE'\left[\int_{\Lambda'}|X'_n|^pd\sigma \bigg| \mathcal{G}_n\right]. \nonumber
\end{eqnarray}
Moreover, for $v=0,1$ and $A\in \mathcal{G}_n$, it follows that
\begin{eqnarray}
&&E'\bigg[\int_{\Lambda'}|\partial^v_{\sigma}X'_n|^pd\sigma,A\bigg]=\int_{\Lambda'}E'[E'[|\partial^v_{\sigma}X'_n|^p|\mathcal{G}_n],A]d\sigma \nonumber \\
&\leq &\int_{\Lambda'}E'\bigg[\sup_{\sigma\in\mathbb{Q}^{d_1}\cap \Lambda'}E'[|\partial^v_{\sigma}X'_n|^p|\mathcal{G}_n],A\bigg]d\sigma \leq E'\bigg[\sup_{\sigma\in\mathbb{Q}^{d_1}\cap \Lambda'}E'[|\partial^v_{\sigma}X'_n|^p|\mathcal{G}_n]\cdot |\Lambda'|,A\bigg], \nonumber
\end{eqnarray}
where $|\Lambda'|$ denotes volume of $\Lambda'$.
Since $A\in\mathcal{G}_n$ is arbitrary, we have
\begin{equation*}
E'\bigg[\int_{\Lambda'}|\partial^v_{\sigma}X'_n|^pd\sigma \bigg| \mathcal{G}_n\bigg]\leq \sup_{\sigma\in \mathbb{Q}^{d_1}\cap \Lambda'}E'[|\partial^v_{\sigma}X'_n|^p|\mathcal{G}_n]\cdot |\Lambda'| \quad {\rm a.s.}
\end{equation*}
Therefore we obtain
\begin{equation*}
E'\bigg[\bigg(\sup_{\sigma}|X'_n|\bigg)^p\bigg|\mathcal{G}_n\bigg]\leq C|\Lambda'|\sum_{v=0}^1\sup_{\sigma\in\mathbb{Q}^{d_1}\cap \Lambda'}E'[|\partial^v_{\sigma}X'_n|^p|\mathcal{G}_n]\to^p 0
\end{equation*}
as $n\to \infty$. Then the proof is completed by $1$.
\end{proof}

\begin{lemma}\label{Burk3}
Let $(\Omega',\mathcal{F}',P')$ be a probability space, $T'>0$, $q\in 2\mathbb{N}$, $\{\mathcal{F}'_t\}_{0\leq t\leq T'}$ be a filtration, 
$M\in \mathbb{N}$, $\{K^i\}_{i=1}^M$ be a deterministic partition of $[0,T']$ where $L(K^i)<L(K^j)$ for $i<j$.
Let $(\tilde{W}^l_t,\mathcal{F}'_t)_{0\leq t\leq T'}$ be standard Brownian motions $(l=1,2,3)$, and 
$F_{i,j,k}$ be $\mathcal{F}'_{L(K^i)\wedge L(K^j)\wedge L(K^k)}$-measurable random variables. 
Assume $\langle \tilde{W}^{p_1},\tilde{W}^{p_2}\rangle$ are deterministic for $1\leq p_1<p_2\leq 3 $.  
Then for $\Delta \tilde{W}^l_i=\tilde{W}^l(K^i)$, 
$F^1_{i,j}=F_{i,j,j}, F^2_{i,j}=F_{j,i,j}$ and $F^3_{i,j}=F_{j,j,i}$,
there exists a constant $C_q>0$ which depends only on $q$ such that
\begin{eqnarray}
E'\bigg[\bigg|\sum_{i,j,k}\Delta \tilde{W}^1_i\Delta \tilde{W}^2_j\Delta \tilde{W}^3_kF_{i,j,k}\bigg|^q\bigg] &\leq &C_q\bigg(\sum_{i,j,k}|K^i||K^j||K^k|E'[|F_{i,j,k}|^q]^{\frac{2}{q}}\bigg)^{\frac{q}{2}} \nonumber \\
&&+C_q\bigg(\sum_i|K^i|\bigg(\sum_{j\neq i}|K^j|\sum_{l=1}^3E'[|F^l_{i,j}|^q]^{\frac{1}{q}}\bigg)^2\bigg)^{\frac{q}{2}}.\nonumber 
\end{eqnarray} 
\end{lemma}
\begin{proof}
In this proof, we set general constants denoted by $C$ depend only on $q$.

Let us denote
\begin{equation*}
\mathcal{H}_{i,j,k}=\Delta \tilde{W}^1_i\Delta \tilde{W}^2_j\Delta \tilde{W}^3_kF_{i,j,k}, \quad \mathcal{H}^2_{i,j}=\mathcal{H}_{i,j,j}+\mathcal{H}_{j,i,j}+\mathcal{H}_{j,j,i},
\end{equation*}
\begin{equation*}
\mathcal{H}^3_{i,j,k}=\mathcal{H}_{i,j,k}+\mathcal{H}_{i,k,j}+\mathcal{H}_{j,i,k}+\mathcal{H}_{j,k,i}+\mathcal{H}_{k,i,j}+\mathcal{H}_{k,j,i},
\end{equation*}
then it follows that
\begin{equation*}
\sum_{i,j,k}\mathcal{H}_{i,j,k}=\sum_i\mathcal{H}_{i,i,i}+\sum_i\sum_{j<i}(\mathcal{H}^2_{i,j}+\mathcal{H}^2_{j,i})+\sum_i\sum_{j<i}\sum_{k<j}\mathcal{H}^3_{i,j,k}.
\end{equation*}
Since $\langle\tilde{W}^{p_1},\tilde{W}^{p_2}\rangle$ is deterministic for $1\leq p_1<p_2\leq 3$, Ito's formula yields
\begin{equation*}
E'[\mathcal{H}_{i,i,i}|\mathcal{F}_{L(K^i)}]=E'[\mathcal{H}^2_{i,j}|\mathcal{F}_{L(K^i)}]=E'[\mathcal{H}^3_{i,j,k}|\mathcal{F}_{L(K^i)}]=0
\end{equation*}
for $k<j<i$. Therefore by Lemma \ref{q-sum-est} we have
\begin{eqnarray}\label{HK-est}
E'\bigg[\bigg|\sum_{i,j,k}\mathcal{H}_{i,j,k}\bigg|^q\bigg] 
&\leq & C\bigg\{\sum_iE'[|\mathcal{H}_{i,i,i}|^q]^{\frac{2}{q}}+\sum_iE'\bigg[\bigg|\sum_{j<i}(\mathcal{H}^2_{i,j}+\mathcal{H}^2_{j,i}-E'[\mathcal{H}^2_{j,i}|\mathcal{F}'_{L(K^i)}])\bigg|^q\bigg]^{\frac{2}{q}} \nonumber \\
&&+\sum_iE'\bigg[\bigg|\sum_{j<i}\sum_{k<j}\mathcal{H}^3_{i,j,k}\bigg|^q\bigg]^{\frac{2}{q}}\bigg\}^{\frac{q}{2}}+CE'\bigg[\bigg|\sum_i\sum_{j<i}E'[\mathcal{H}^2_{j,i}|\mathcal{F}'_{L(K^i)}]\bigg|^q\bigg]. 
\end{eqnarray}
We will estimate each term of the right-hand side of (\ref{HK-est}). First, 
\begin{eqnarray}\label{HK-est1}
\sum_iE'[|\mathcal{H}_{i,i,i}|^q]^{\frac{2}{q}}&=&\sum_iE'[|F_{i,i,i}|^qE'[(\Delta \tilde{W}^1_i\Delta \tilde{W}^2_i\Delta \tilde{W}^3_i)^q|\mathcal{F}'_{L(K^i)}]]^{\frac{2}{q}} 
\leq  C\sum_i|K^i|^3E'[|F_{i,i,i}|^q]^{\frac{2}{q}}.
\end{eqnarray}
Let 
\begin{eqnarray}
\mathcal{H}^{2,1}_{i,j}=\Delta \tilde{W}^2_j\Delta \tilde{W}^3_jF_{i,j,j}, \quad \mathcal{H}^{2,2}_{i,j}=\Delta \tilde{W}^1_j\Delta \tilde{W}^3_jF_{j,i,j}, \quad \mathcal{H}^{2,3}_{i,j}=\Delta \tilde{W}^1_j\Delta \tilde{W}^2_jF_{j,j,i}. \nonumber
\end{eqnarray}
Since
\begin{eqnarray}
E'\bigg[\bigg(\sum_{j<i}\mathcal{H}^{2,l}_{i,j}\bigg)^q\bigg]^{\frac{2}{q}}\leq  C\sum_{j<i}E'\big[(\mathcal{H}^{2,l}_{i,j}-E'[\mathcal{H}^{2,l}_{i,j}|\mathcal{F}'_{L(K^j)}])^q\big]^{\frac{2}{q}} + C\bigg(\sum_{j<i}E'\big[E'[\mathcal{H}^{2,l}_{i,j}|\mathcal{F}'_{L(K^j)}]^q\big]^{\frac{1}{q}}\bigg)^2 \nonumber
\end{eqnarray}
for each $i,l$ by Lemma \ref{q-sum-est}, we obtain 
\begin{eqnarray}\label{HK-est2}
&&\sum_iE'\bigg[\bigg(\sum_{j<i}(\mathcal{H}^2_{i,j}+\mathcal{H}^2_{j,i}-E'[\mathcal{H}^2_{j,i}|\mathcal{F}'_{L(K^i)}])\bigg)^q\bigg]^{\frac{2}{q}} \nonumber \\
&\leq & C\sum_i|K^i|E'\bigg[\sum_{l=1}^3\bigg(\sum_{j<i}\mathcal{H}^{2,l}_{i,j}\bigg)^q\bigg]^{\frac{2}{q}} +C\sum_i|K^i|^2E'\bigg[\sum_{l=1}^3\bigg(\sum_{j<i}\Delta \tilde{W}^l_jF^l_{j,i}\bigg)^q\bigg]^{\frac{2}{q}} \nonumber \\
&\leq &C\sum_i|K^i|\sum_{j<i}|K^j|^2\sum_{l=1}^3E'[(F^l_{i,j})^q]^{\frac{2}{q}}+C\sum_i|K^i|\sum_{l=1}^3\bigg(\sum_{j<i}|K^j|E'[(F^l_{i,j})^q]^{\frac{1}{q}}\bigg)^2 \nonumber \\
&&+C\sum_i|K^i|^2\sum_{j<i}|K^j|\sum_{l=1}^3E'[(F^l_{j,i})^q]^{\frac{2}{q}}.
\end{eqnarray}

Moreover, let 
\begin{eqnarray}
\mathcal{H}^{3,1}_{i,j,k}&=&\Delta \tilde{W}^2_j\Delta \tilde{W}^3_kF_{i,j,k}+\Delta \tilde{W}^2_k\Delta \tilde{W}^3_jF_{i,k,j}, \nonumber \\
\mathcal{H}^{3,2}_{i,j,k}&=&\Delta \tilde{W}^1_j\Delta \tilde{W}^3_kF_{j,i,k}+\Delta \tilde{W}^1_k\Delta \tilde{W}^3_jF_{k,i,j}, \nonumber \\
\mathcal{H}^{3,3}_{i,j,k}&=&\Delta \tilde{W}^1_j\Delta \tilde{W}^2_kF_{j,k,i}+\Delta \tilde{W}^1_k\Delta \tilde{W}^2_jF_{k,j,i}, \nonumber 
\end{eqnarray}
then by Lemma \ref{q-sum-est} we have
\begin{eqnarray}\label{HK-est3}
\sum_iE'\bigg[\bigg(\sum_{j<i}\sum_{k<j}\mathcal{H}^3_{i,j,k}\bigg)^q\bigg]^{\frac{2}{q}} 
&\leq & C\sum_i|K^i|\sum_{l=1}^3E'\bigg[\bigg(\sum_{j<i}\sum_{k<j}\mathcal{H}^{3,l}_{i,j,k}\bigg)^q\bigg]^{\frac{2}{q}} 
 \leq C\sum_i|K^i|\sum_{l=1}^3\sum_{j<i}E'\bigg[\bigg(\sum_{k<j}\mathcal{H}^{3,l}_{i,j,k}\bigg)^q\bigg]^{\frac{2}{q}} \nonumber \\
&\leq & C\sum_i|K^i|\sum_{j<i}|K^j|\sum_{k<j}|K^k|\Big(E'[(F_{i,j,k})^q]^{\frac{2}{q}}+E'[(F_{i,k,j})^q]^{\frac{2}{q}} \nonumber \\
&&+E'[(F_{j,i,k})^q]^{\frac{2}{q}}+E'[(F_{j,k,i})^q]^{\frac{2}{q}}+E'[(F_{k,i,j})^q]^{\frac{2}{q}}+E'[(F_{k,j,i})^q]^{\frac{2}{q}}\Big). 
\end{eqnarray}
Furthermore, let $g_1(K^i)=\langle\tilde{W}^2,\tilde{W}^3\rangle(K^i)$, $g_2(K^i)=\langle\tilde{W}^1,\tilde{W}^3\rangle(K^i)$, $g_3(K^i)=\langle\tilde{W}^1,\tilde{W}^2\rangle(K^i)$, then we obtain
\begin{eqnarray}\label{HK-est4}
E'\bigg[\bigg|\sum_i\sum_{j<i}E'[\mathcal{H}^2_{j,i}|\mathcal{F}'_{L(K^i)}]\bigg|^q\bigg] =E'\bigg[\bigg|\sum_i\sum_{j<i}\sum_{l=1}^3g_l(K^i)\Delta \tilde{W}^l_jF^l_{j,i}\bigg|^q\bigg] 
\leq  3^q\sum_{l=1}^3E'\bigg[\bigg|\sum_j\bigg(\sum_{i>j}g(K^i)F^l_{j,i}\bigg)\Delta \tilde{W}^l_j\bigg|^q\bigg]. \nonumber 
\end{eqnarray}
Hence Lemma \ref{q-sum-est} yields
\begin{eqnarray}
E'\bigg[\bigg|\sum_i\sum_{j<i}E'[\mathcal{H}^2_{j,i}|\mathcal{F}'_{L(K^i)}]\bigg|^q\bigg]&\leq & C\sum_{l=1}^3\bigg(\sum_j|K^j|E'\bigg[\bigg(\sum_{i>j}g_l(K^i)F^l_{j,i}\bigg)^q\bigg]^{\frac{2}{q}}\bigg)^{\frac{q}{2}} \nonumber \\
&\leq & C\bigg(\sum_j|K^j|\bigg(\sum_{i>j}|K^i|\sum_{l=1}^3E'[(F^l_{j,i})^q]^{\frac{1}{q}}\bigg)^2\bigg)^{\frac{q}{2}}. 
\end{eqnarray}
By (\ref{HK-est})-(\ref{HK-est4}), we obtain the conclusion.
\end{proof}

For $p\geq 0$, we denote
\begin{eqnarray}
L_p&=&\{\rho_{L(\theta_{[p/2],i}\cup\theta_{[p/2],j+l_n})\wedge \tau_n}G_{I^i,J^j}\}_{i,j}, \quad \tilde{L}_p=\left(
\begin{array}{ll}
0 & L_p \\
L^{\star}_p & 0 \\
\end{array}
\right),\quad \tilde{M}=\sum_{p=0}^{\infty}(-1)^p\tilde{L}_p^p, \nonumber 
\end{eqnarray}
\begin{equation*}
\tilde{M}_p=\left(
\begin{array}{ll}
(GG^{\star})^p \ \ (GG^{\star})^pG \\
(G^{\star}G)^pG^{\star} \ \ (G^{\star}G)^p \\
\end{array}
\right), \quad
\hat{Z}_{k,t}=\left\{
\begin{array}{ll}
\int_{I^k_t}b^1_{s,\ast}dW_s/(|b^1_{I^k,\tau_n}|\sqrt{|I^k|}) & (k\leq l_n) \\
\int_{J_t^{k-l_n}}b^2_{s,\ast}dW_s/\left(|b^2_{J^{k-l_n},\tau_n}|\sqrt{|J^{k-l_n}|}\right) & (k>l_n) \\
\end{array}
\right.
\end{equation*}
and $D_t={\rm diag}(\{|b^1_{I,\tau_n}|\}_{I\cap [0,t)\neq \emptyset},\{|b^2_{J,\tau_n}|\}_{J\cap [0,t)\neq \emptyset})$.
Though $\{\hat{Z}_{k,t}\}_t$ is not necessarily a local martingale, we denote by $\langle\hat{Z}\rangle_t$ 
the quadratic variation of $\hat{Z}$ regarding $\Pi$ as deterministic functions, that is,  
$\langle\hat{Z}\rangle_t$ be an $(l_n+m_n) \times (l_n+m_n)$ symmetric matrix with
\begin{eqnarray}
(\langle\hat{Z}\rangle_t)_{k,k'}=\left\{
\begin{array}{ll}
\int_{I^k_t}|b^1_{s,\ast}|^2ds\delta_{k,k'}/(|b^1_{I^k,\tau_n}|^2|I^k|) & (k,k'\leq l_n) \\
\int_{J^{k-l_n}_t}|b^2_{s,\ast}|^2ds\delta_{k,k'}/(|b^2_{J^{k-l_n},\tau_n}|^2|J^{k-l_n}|) & (k,k'> l_n) \\
\int_{I^k_t\cap J^{k'-l_n}_t}b^1_{s,\ast}\cdot b^2_{s,\ast}ds/(|b^1_{I^k,\tau_n}||b^2_{J^{k'-l_n},\tau_n}|\sqrt{|I^k||J^{k'-l_n}|} & (k\leq l_n, k'>l_n) 
\end{array}
\right. \nonumber
\end{eqnarray}
Moreover we define  
\begin{eqnarray}
\tilde{H}^1_n(t)&=&\tilde{H}^1_{n,s_n}(t;\sigma)= -\frac{1}{2}\hat{Z}_{\cdot,t}^{\star}\tilde{M}\hat{Z}_{\cdot,t}-\log\det D_t +\frac{1}{2}\sum_{p=1}^{\infty}\frac{(-1)^p}{p}\sum_k(\tilde{L}_p^p)_{k,k}1_{\{\theta_{0,k}\cap [0,t)\neq \emptyset\} }, \nonumber \\
\tilde{H}^2_n(t)&=&\tilde{H}^2_{n,s_n}(t;\sigma)=-\frac{1}{2}{\rm tr}(\tilde{M}\langle\hat{Z}\rangle_t)-\log\det D_t +\frac{1}{2}\sum_{p=1}^{\infty}\frac{(-1)^p}{p}\sum_k(\tilde{L}_p^p)_{k,k}1_{\{\theta_{0,k}\cap [0,t)\neq \emptyset\} }, \nonumber
\end{eqnarray}
\begin{equation*}
\tilde{H}^3_n(t)=\tilde{H}^3_{n,s_n}(t;\sigma)=b_n\sum_{p=0}^{\infty}\sum_{i=1}^2\int^t_0\mathcal{D}^i_p(s\wedge \tau(s_n),s;\sigma)\nu^{p,i}_n(ds),
\end{equation*}
$\hat{Z}_k=\hat{Z}_{k,T}$, and $\tilde{H}^i_n=\tilde{H}^i_n(T) \ ( 1\leq i\leq 3)$, where
\begin{equation*}
\mathcal{D}^i_p(s,t;\sigma)=\left\{
\begin{array}{ll}
-\frac{|b^i_{t,\ast}|^2}{2|b^i_{s}|^2}-\log |b^i_{s}| & (p=0) \\
-\frac{|b^i_{t,\ast}|^2}{2|b^i_{s}|^2}\rho_s^{2p}+\frac{1}{2}\left(\frac{|b^1_{t,\ast}||b^2_{t,\ast}|}{|b^1_s||b^2_s|}+\frac{|b^i_{t,\ast}|}{|b^i_s|}-\frac{|b^{3-i}_{t,\ast}|}{|b^{3-i}_s|}\right)\rho^{2p-1}_s\rho_{t,\ast}+\frac{\rho^{2p}_s}{4p} & (p\geq 1)
\end{array}
\right.
\end{equation*}
for $i=1,2$. Then we have $\partial_{\sigma}\tilde{H}^3_{n,s_n}(t;\sigma_{\ast})\equiv 0$ on $\{\tau(s_n)=T\}$.

Let $q\in 2\mathbb{N}$, $\gamma\in (0,1)$ be defined in $[A1] \ 5.$, 
$\Theta^1_p=\sup_{0\leq t\leq T}E[|\mu_t|^p]$, $\Theta^2_p=\sup_{0\leq s<t\leq T}E[|\mu_t-\mu_s|^p]/|t-s|^{p\gamma}$,
and $\{s_n(t)\}_{0\leq t\leq T,n\in\mathbb{N}}$ be stochastic processes which satisfies $[S]$. 
Moreover, we define $\varphi_q(\{x_p\})=(\sum_{p=0}^{\infty}x_p)^q\vee (\sum_{p=0}^{\infty}x_p^{2q/(2q-1)})^{q-1/2}$ for $\{x_p\}_{p=0}^{\infty}\subset \mathbb{R}_+$ and $q\in 2\mathbb{N}$.

\begin{lemma}\label{H1-diff}
Assume $[A1]$. Let $r\in\mathbb{N}, r\geq 2$. Then there exists a constant $C>0$ which depends only on $q,r,n_2$ and $n_3$ such that
\begin{eqnarray}\label{H1-diff-ineq}
&&E[|\partial^v_{\sigma}\hat{H}_{n}(\sigma;s_n)-\partial^v_{\sigma}\tilde{H}^1_{n,s_n}(T;\sigma)|^q|\Pi] \nonumber \\
&\leq & C(T^{\frac{3}{2}q}\vee 1)(E[R^C]+\Theta^1_C)E[R^Cs_n(T)^{-(2v+2r+7)q}|\Pi]^{\frac{1}{2}} \nonumber \\
&&\times \bigg\{(T\vee 1)^{\frac{q}{2}} +((r_n^{q(\gamma\wedge \frac{1}{2})}(l_n+m_n)^{\frac{q}{2}})\vee 1)\{(\Theta^1_{8q})^{\frac{1}{8}}+(\Theta^2_{8q})^{\frac{1}{8}}\} \nonumber \\
&&+\varphi_q\bigg(\bigg\{\frac{\sqrt{(l_n+m_n)\Phi_{2p+2,2}}\vee \Phi_{2p+2,1}}{(p+1)^r}\bigg\}_p\bigg)+\bigg(\sum_{p_1,p_2=0}^{\infty}\frac{\bar{\Phi}_{2p_1+3,2p_2+3}}{(p_1+1)^r(p_2+1)^r}
\bigg)^{\frac{q}{2}}\bigg\} 
\end{eqnarray}
and
\begin{eqnarray}
&&E[|\partial^v_{\sigma}\tilde{H}^2_{n,s_n}(T;\sigma)-\partial^v_{\sigma}\tilde{H}^3_{n,s_n}(T;\sigma)|^q|\Pi] \nonumber \\
&\leq & CE[R^C(1-\bar{\rho}_T)^{-C}]\bigg\{\bigg(\sum_{p=0}^{\infty}\frac{\Phi_{2p+2,1}}{(p+1)^r}\bigg)^q+\bigg(\sum_{p_1,p_2=0}^{\infty}\frac{\bar{\Phi}_{2p_1+2,2p_2+2}}{(p_1+1)^r(p_2+1)^r}
\bigg)^{\frac{q}{2}}\bigg\} \nonumber
\end{eqnarray}
for $0\leq v\leq 4$ and $\sigma\in\Lambda$.
\end{lemma}

\begin{proof}
In this proof, we set general constants denoted by $C$ depend only on $q,r,n_2,n_3$.

We first prove (\ref{H1-diff-ineq}). Let 
\begin{eqnarray}
\tilde{\mu}_k=\left\{
\begin{array}{ll}
\int_{I^k}\mu^1_sds/(|b^1_{I^k,\tau_n}|\sqrt{|I^k|}) & (k\leq l_n) \\
\int_{J^{k-l_n}}\mu^2_sds/(|b^2_{J^{k-l_n},\tau_n}|\sqrt{|J^{k-l_n}|}) & (k>l_n) \\
\end{array}
\right. \nonumber
\end{eqnarray} 
for $1\leq k \leq l_n+m_n$. Then $\hat{H}_n(\sigma;s_n)-\tilde{H}_{n,s_n}^1(T;\sigma)$ can be decomposed as 
\begin{equation*}
\hat{H}_n(\sigma;s_n)-\tilde{H}_{n,s_n}^1(T;\sigma)=-\frac{1}{2}Z^{\star}(M-\tilde{M})Z-\tilde{\mu}^{\star}\tilde{M}(\hat{Z}+\frac{\tilde{\mu}}{2})+\frac{1}{2}\sum_{p=1}^{\infty}\frac{(-1)^p}{p}({\rm tr}(\tilde{L}^p)-{\rm tr}(\tilde{L}_p^p)).
\end{equation*}
We will estimate each term of the right-hand side of this equation.

First, we assume that $\Pi$ is deterministic. Let $\mathcal{W}_t=(W_t,\hat{W}_t)$, then we can write
\begin{equation*}
((\tilde{L})^{2p}-(\tilde{L})^{2p+1})_{kk'}-((\tilde{L}_{2p})^{2p}-(\tilde{L}_{2p+1})^{2p+1})_{kk'}=\int_{\theta_{2p+2,k}}\xi^{k,k'}_{p,t}\cdot d\mathcal{W}_t+\int_{\theta_{2p+2,k}}\eta^{k,k'}_{p,t}dt,
\end{equation*}
by Ito's formula, where
\begin{eqnarray}
|\partial^v_{\sigma}\xi^{k,k'}_{p,t}|&\leq & 2\cdot 2\cdot 4^{v+1}(2p+1)^{v+1}\cdot n_2R^{2v+3}(1-s_n(T))^{(2p-v-1)\vee 0}(\tilde{M}_p)_{kk'}1_{\theta_{2p+2,k}}(t), \nonumber \\
|\partial^v_{\sigma}\eta^{k,k'}_{p,t}|&\leq &4\cdot 2\cdot \frac{1}{2}4^{v+2}(2p+1)^{v+2}\cdot n_2^2R^{2v+6}(1-s_n(T))^{(2p-v-2)\vee 0}(\tilde{M}_p)_{kk'}1_{\theta_{2p+2,k}}(t). \nonumber 
\end{eqnarray}

Moreover, let $\tilde{\xi}^{k,k'}_{p,t}$ be the one constructed by $\xi^{k,k'}_{p,t}$ 
substituting $L(\theta_{2p+2,k}) \wedge \tau(s_n)$ for all times of $X$, $\tilde{b}^2$ and $\tilde{b}^3$,
then we can write $\tilde{\xi}_{p,t}^{k,k'}=\sum_K\tilde{\xi}^{k,k'}_{p,K}1_K(t)$ 
for some random variables $\{\tilde{\xi}^{k,k'}_{p,K}\}_K$,
where $\{K\}$ denotes the set of intervals obtained by unifying partitions $\{S^i\}$,$\{T^j\}$ and $\{\mathcal{T}^j_k\}$. 
Furthermore, let 
\begin{equation}\label{xi-diff-decom}
\xi^{kk'}_{p,t}-\tilde{\xi}^{kk'}_{p,t}=\int^t_0\hat{\xi}^{k,k'}_{p,s}\cdot d\mathcal{W}_s+\int^t_0 \hat{\eta}^{k,k'}_{p,s}ds,
\end{equation}
then we have
\begin{eqnarray}
|\partial^v_{\sigma}\hat{\xi}^{k,k'}_{p,s}|&\leq & 2\cdot 4\cdot 4^{v+1}(2p+1)^{v+1}\cdot 4(2p+2)\cdot n_2^2R^{2v+6}(1-s_n(T))^{(2p-v-2)\vee 0}(\tilde{M}_p)_{k,k'}1_{\theta_{2p+2,k}}(s), \nonumber \\
|\partial^v_{\sigma}\hat{\eta}^{k,k'}_{p,s}|&\leq &2\cdot 8\cdot 4^{v+1}(2p+1)^{v+1}\cdot \frac{1}{2}4^2(2p+2)^2\cdot n_2^3R^{2v+9}(1-s_n(T))^{(2p-v-3)\vee 0}(\tilde{M}_p)_{k,k'}1_{\theta_{2p+2,k}}(s). \nonumber
\end{eqnarray}
Therefore we can write
\begin{eqnarray}
-\frac{1}{2}Z^{\star}(M-\tilde{M})Z 
&=& -\frac{1}{2}\sum_{p=0}^{\infty}\sum_{k,k'}\left(\int\xi^{k,k'}_{p,t}\cdot d\mathcal{W}_t+\int \eta^{k,k'}_{p,t}dt\right)Z_kZ_{k'} \nonumber \\
&=&-\frac{1}{2}\sum_{p=0}^{\infty}\sum_{k,k'}\int\tilde{\xi}^{k,k'}_{p,t}\cdot d\mathcal{W}_t\tilde{Z}_{k,p}\tilde{Z}_{k',p}-\frac{1}{2}\sum_{p=0}^{\infty}\sum_{k,k'}\int(\xi^{k,k'}_{p,t}-\tilde{\xi}^{k,k'}_{p,t})\cdot d\mathcal{W}_tZ_kZ_{k'} \nonumber \\
&&-\frac{1}{2}\sum_{p=0}^{\infty}\sum_{k,k'}\int\tilde{\xi}^{k,k'}_{p,t}\cdot d\mathcal{W}_t(Z_kZ_{k'}-\tilde{Z}_{k,p}\tilde{Z}_{k',p})-\frac{1}{2}\sum_{p=0}^{\infty}\sum_{k,k'}\int\eta^{k,k'}_{p,t}dtZ_kZ_{k'} \nonumber \\
&=& \mathcal{X}+R_1+R_2+R_3, \nonumber
\end{eqnarray}
where 
\begin{equation*}
(\tilde{Z}_{k,p})_{k=1}^{l_n+m_n}=\bigg(\bigg(\frac{b^1_{L(\theta_{2p+2,i}),\ast}\cdot W(I^i)}{|b^1_{L(\theta_{2p+2,i})\wedge \tau_n}|\sqrt{|I^i|}}\bigg)_{i=1}^{l_n}, \bigg(\frac{b^2_{L(\theta_{2p+2,j+l_n}),\ast}\cdot W(J^j)}{|b^2_{L(\theta_{2p+2,j+l_n})\wedge \tau_n}|\sqrt{|J^j|}}\bigg)_{j=1}^{m_n}\bigg).
\end{equation*} 

Let $k^1(K)$ denotes $k\leq l_n$ which satisfies $K\subset I^{k}$, $k^2(K)$ denotes $k>l_n$ which satisfies $K\subset J^{k-l_n}$ and
\begin{equation*}
\tilde{B}_{k,p}=\left\{
\begin{array}{ll}
b^1_{L(\theta_{2p+2,k}),\ast}/(|b^1_{L(\theta_{2p+2,k})\wedge \tau_n}|\sqrt{|I^k|}) & (k\leq l_n) \\
b^2_{L(\theta_{2p+2,k}),\ast}/(|b^2_{L(\theta_{2p+2,k})\wedge \tau_n}|\sqrt{|J^{k-l_n}|}) & (l_n < k \leq l_n + m_n). \\
\end{array}
\right.
\end{equation*} 
Then $\mathcal{X}$ can be rewritten as  
\begin{eqnarray}
\mathcal{X}&=&-\frac{1}{2}\sum_{p=0}^{\infty}\sum_{k,k'}\sum_{K''}\tilde{\xi}^{k,k'}_{p,K''}\cdot \mathcal{W}(K'')\tilde{Z}_{k,p}\tilde{Z}_{k',p} \nonumber \\
&=&-\frac{1}{2}\sum_{K,K',K''}\sum_{i,j=1}^{2}\sum_{p=0}^{\infty}\tilde{\xi}^{k^i(K),k^j(K')}_{p,K''}\cdot \mathcal{W}(K'')\tilde{B}_{k^i(K),p}\cdot W(K)\tilde{B}_{k^j(K),p}\cdot W(K'). \nonumber 
\end{eqnarray}
Let 
\begin{equation*}
F^{i,j,v}_{K,K',K''}=\frac{1}{2}\sum_{v_1,v_2,v_3\geq 0, v_1+v_2+v_3=v}\frac{v!}{v_1!v_2!v_3!}\sum_{p=0}^{\infty}|\partial^{v_1}_{\sigma}\tilde{\xi}^{k^i(K),k^j(K')}_{p,K''}||\partial^{v_2}_{\sigma}\tilde{B}_{k^i(K),p}||\partial^{v_3}_{\sigma}\tilde{B}_{k^j(K'),p}|
\end{equation*}
for $1\leq i,j\leq 2$, then $F^{i,j,v}_{K,K',K''}=F^{j,i,v}_{K',K,K''}$. Hence for general $\Pi$ and any $q\in 2\mathbb{N}$, we have  
\begin{eqnarray}
&&E[|\partial^{v}_{\sigma}\mathcal{X}|^q|\Pi] \nonumber \\
&\leq &C\sum_{i,j=1}^2\bigg\{\bigg(\sum_{K,K',K''}|K||K'||K''|E[(F_{K,K',K''}^{i,j,v})^q|\Pi]^{\frac{2}{q}}\bigg)^{\frac{q}{2}} \nonumber \\
&&+\bigg(\sum_K|K|\bigg(\sum_{K'}|K'| (2E[(F^{i,j,v}_{K,K',K'})^q|\Pi]^{\frac{1}{q}}+E[(F^{i,j,v}_{K',K',K})^q|\Pi]^{\frac{1}{q}}\bigg)^2\bigg)^{\frac{q}{2}}\bigg\} \nonumber \\
&\leq &C\bigg(\sum_{k,k',K''}|K''| E\bigg[\bigg(\sum_{p=0}^{\infty}R^C(p+1)^{v+1}(1-s_n(T))^{(2p-5)\vee 0}(\tilde{M}_p)_{k,k'}1_{\theta_{2p+2,k}}(K'')\bigg)^q\bigg|\Pi\bigg]^{\frac{2}{q}}\bigg)^{\frac{q}{2}} \nonumber \\
&&+C\bigg\{\sum_k|\theta_{0,k}|\bigg(\sum_{k'}|\theta_{0,k'}|\bigg\{E\bigg[\bigg(\sum_{p=0}^{\infty}R^C(p+1)^{v+1}\frac{(1-s_n(T))^{(2p-5)\vee 0}}{\sqrt{|\theta_{0,k}|}\sqrt{|\theta_{0,k'}|}}(\tilde{M}_p)_{k,k'}\bigg)^q\bigg|\Pi\bigg]^{\frac{1}{q}} \nonumber \\
&&\quad +E\bigg[\bigg(\sum_{p=0}^{\infty}R^C(p+1)^{v+1}\frac{(1-s_n(T))^{(2p-5)\vee 0}}{|\theta_{0,k'}|}1_{\theta_{2p+2,k}\cap \theta_{0,k'}\neq \emptyset}\bigg)^q\bigg|\Pi\bigg]^{\frac{1}{q}}\bigg\}\bigg)^2\bigg\}^{\frac{q}{2}} \nonumber
\end{eqnarray}
by Lemma \ref{Burk3}. 

Moreover, by Lemma \ref{Fp-est}, we have
\begin{eqnarray}
E[|\partial^{v}_{\sigma}\mathcal{X}|^q|\Pi]&\leq & CE[R^Cs_n(T)^{-(v+\frac{2r+3}{2})q}|\Pi]\bigg\{\bigg(\sum_{k,k',K''}|K''|\sum_{p=0}^{\infty}\frac{(\tilde{M}_p)_{k,k'}1_{\theta_{2p+2,k}(K'')}}{(p+1)^{2r}}\bigg)^{\frac{q}{2}} \nonumber \\
&&+\bigg(\sum_k\bigg(\sum_{k'}\sqrt{|\theta_{0,k'}|}\bigg(\sum_{p=0}^{\infty}\bigg(\frac{(\tilde{M}_p)_{k,k'}}{(p+1)^r}\bigg)^2\bigg)^{\frac{1}{2}}\bigg)^2\bigg)^{\frac{q}{2}} \nonumber \\
&&+\bigg(\sum_k|\theta_{0,k}|\bigg(\sum_{k'}\bigg(\sum_{p=0}^{\infty}\frac{1_{\theta_{2p+2,k}\cap \theta_{0,k'}\neq \emptyset}}{(p+1)^{2r}}\bigg)^{\frac{1}{2}}\bigg)^2\bigg)^{\frac{q}{2}}\bigg\}.\nonumber
\end{eqnarray}
Since $(\sum_{p\in\mathbb{N}}\alpha_p)^{1/q'}\leq \sum_{p\in\mathbb{N}}\alpha_p^{1/q'}$ for $\alpha_p\geq 0 \ (p\in\mathbb{N})$ and $q'>1$, we have
\begin{eqnarray}
E[|\partial^{v}_{\sigma}\mathcal{X}|^q|\Pi] &\leq & CE[R^Cs_n(T)^{-(v+\frac{2r+3}{2})q}|\Pi]\bigg\{\bigg(\sum_{k,k'}\sum_{p=0}^{\infty}\frac{(\tilde{M}_p)_{k,k'}|\theta_{2p+2,k}|}{(p+1)^{2r}}\bigg)^{\frac{q}{2}} \nonumber \\
&&+\bigg(\sum_k\sum_{k'_1,k'_2}\sqrt{|\theta_{0,k'_1}|}\sqrt{|\theta_{0,k'_2}|}\sum_{p_1,p_2=0}^{\infty}\frac{(\tilde{M}_{p_1})_{k,k'_1}(\tilde{M}_{p_2})_{k,k'_2}}{(p_1+1)^r(p_2+1)^r}\bigg)^{\frac{q}{2}} \nonumber \\
&&+\bigg(\sum_k|\theta_{0,k}|\sum_{k'_1,k'_2}\sum_{p_1,p_2=0}^{\infty}\frac{1_{\theta_{2p_1+2,k'_1}\cap \theta_{0,k}\neq \emptyset}1_{\theta_{2p_2+2,k'_2}\cap \theta_{0,k}\neq \emptyset}}{(p_1+1)^r(p_2+1)^r}\bigg)^{\frac{q}{2}}\bigg\} \nonumber
\end{eqnarray}
Therefore, $E[|\partial^{v}_{\sigma}\mathcal{X}|^q|\Pi]$ is less than the right-hand side of (\ref{H1-diff-ineq}) since $\parallel \tilde{M}_p\parallel \leq 2$,
\begin{equation*}
\sum_k|\theta_{0,k}|1_{\theta_{2p_1+2,k'_1}\cap \theta_{0,k}\neq \emptyset}1_{\theta_{2p_2+2,k'_2}\cap \theta_{0,k}\neq \emptyset}\leq |\theta_{2p_1+3,k_1'}\cap \theta_{2p_2+3,k_2'}|,
\end{equation*}
\begin{eqnarray}
\sum_{k,k'}(\tilde{M}_p)_{k,k'}|\theta_{2p+2,k}|\leq \parallel \tilde{M}_p\parallel (l_n+m_n)^{\frac{1}{2}}\Phi_{2p+2,2}^{\frac{1}{2}}\leq 2(l_n+m_n)^{\frac{1}{2}}\Phi_{2p+2,2}^{\frac{1}{2}}, \nonumber
\end{eqnarray} 
and
\begin{equation*}
\sum_k\sum_{k_1',k_2'}\sqrt{|\theta_{0,k'_1}|}\sqrt{|\theta_{0,k'_2}|}(\tilde{M}_{p_1})_{k,k'_1}(\tilde{M}_{p_2})_{k,k'_2} \leq \parallel \tilde{M}_{p_1}^{\star} \tilde{M}_{p_2}\parallel \sum_k|\theta_{0,k}|\leq 8T. 
\end{equation*}

We will estimate $E[|R_2|^q|\Pi]$ in the next step. Since 
\begin{equation*}
E[(\partial^{v}_{\sigma}(Z_kZ_{k'}-\tilde{Z}_{k,p}\tilde{Z}_{k',p}))^{2q}|\Pi]^{\frac{1}{2q}}\leq C_q(T\vee 1)(E[R^C]+\Theta^1_C)(|\theta_{2p+2,k}|^{\frac{1}{2}}+|\theta_{2p+2,k'}|^{\frac{1}{2}}),
\end{equation*}
Lemma \ref{q-sum-est} and the Cauchy-Schwarz inequality yield
\begin{eqnarray}
E[|\partial^{v}_{\sigma}R_2|^q|\Pi] 
&\leq & \frac{1}{2^q}\bigg(\sum_{k,k'}\sum_{p=0}^{\infty}\sum_{v_1,v_2\geq 0,v_1+v_2=v}\frac{v!}{v_1!v_2!}E\bigg[\bigg(\int_{\theta_{2p+2,k}}\partial^{v_1}_{\sigma}\tilde{\xi}^{k,k'}_{p,t}\cdot d\mathcal{W}_t\bigg)^{2q}\bigg|\Pi\bigg]^{\frac{1}{2q}} \nonumber \\
&&\times E\bigg[(\partial^{v_2}_{\sigma}(Z_kZ_{k'}-\tilde{Z}_{k,p}\tilde{Z}_{k',p}))^{2q}\bigg|\Pi\bigg]^{\frac{1}{2q}}\bigg)^q \nonumber \\
&\leq & C\bigg(\sum_{k,k'}\sum_{p=0}^{\infty}(p+1)^{v+1}|\theta_{2p+2,k}|^{\frac{1}{2}}E[R^C(1-s_n(T))^{2q(2p-5)\vee 0}|\Pi]^{\frac{1}{2q}}(\tilde{M}_p)_{k,k'} \nonumber \\
&&\times (E[R^C]+\Theta^1_C)(T\vee 1)(|\theta_{2p+2,k}|^{\frac{1}{2}}+|\theta_{2p+2,k'}|^{\frac{1}{2}})\bigg)^q. \nonumber 
\end{eqnarray}
Hence by the H${\rm \ddot{o}}$lder inequality, we have
\begin{eqnarray}
E[|\partial^{v}_{\sigma}R_2|^q|\Pi]
&\leq & C(T^q\vee 1)(E[R^C]+\Theta^1_C)E[R^Cs_n(T)^{-(2v+2r+3)q}|\Pi]^{\frac{1}{2}} \nonumber \\
&&\times \bigg(\sum_{p=0}^{\infty}\bigg(\sum_{k,k'}\frac{(\tilde{M}_p)_{k,k'}(|\theta_{2p+2,k}|+|\theta_{2p+2,k}|^{1/2}|\theta_{2p+2,k'}|^{1/2})}{(p+1)^r}\bigg)^{\frac{2q}{2q-1}}\bigg)^{q-\frac{1}{2}} \nonumber \\
&\leq & C(T^q\vee 1)(E[R^C]+\Theta^1_C)E[R^Cs_n(T)^{-(2v+2r+3)q}|\Pi]^{\frac{1}{2}} \nonumber \\
&&\times\bigg\{\sum_{p=0}^{\infty}\bigg(\frac{((l_n+m_n)\Phi_{2p+2,2})^{1/2}\vee \Phi_{2p+2,1}}{(p+1)^r}\bigg)^{\frac{2q}{2q-1}}\bigg\}^{q-\frac{1}{2}}. \nonumber 
\end{eqnarray}
Then $E[|R_2|^q|\Pi]$ is less than the right-hand side of (\ref{H1-diff-ineq}). 

Furthermore, by Lemma \ref{q-sum-est} and the Cauchy-Schwarz inequality, we have
\begin{eqnarray}\label{R1-est}
E[|\partial^{v}_{\sigma}R_1|^q|\Pi] &\leq &C\bigg(\sum_{k,k'}E\bigg[\bigg|\partial^{v}_{\sigma}\bigg(\int\sum_{p=0}^{\infty}(\xi^{k,k'}_{p,t}-\tilde{\xi}^{k,k'}_{p,t})\cdot d\mathcal{W}_tZ_kZ_{k'}\bigg)\bigg|^q\bigg|\Pi\bigg]^{\frac{1}{q}}\bigg)^q \nonumber \\
&\leq & C(E[R^C]+T^q(\Theta^1_{8q})^{1/4}) \bigg(\sum_{k,k'}\sum_{0\leq v_1\leq v}E\bigg[\bigg|\int\bigg(\sum_{p=0}^{\infty}\partial^{v_1}_{\sigma}(\xi^{k,k'}_{p,t}-\tilde{\xi}^{k,k'}_{p,t})\bigg)^2dt\bigg|^q\bigg|\Pi\bigg]^{\frac{1}{2q}}\bigg)^q. 
\end{eqnarray}
Since $\Pi$ is independent of $\{(X_t,(\tilde{b}^i_t)_i)\}_t$, we can choose conditional expectation for which
$t\mapsto E[(\sum_{p=0}^{\infty}\partial^{v_1}_{\sigma}(\xi^{k,k'}_{p,t}-\tilde{\xi}^{k,k'}_{p,t}))^{2q}|\Pi]$ is Lebesgue integrable almost surely for $0\leq v_1\leq v$.
Therefore, by (\ref{R1-est}) and similar argument to the proof of Lemma \ref{q-sum-est}, we have
\begin{eqnarray}\label{R1-est2}
E[|\partial^{v}_{\sigma}R_1|^q|\Pi]\leq  C(E[R^C]+T^q(\Theta^1_{8q})^{1/4}) \bigg(\sum_{k,k'}\sum_{0\leq v_1\leq v}\bigg(\int E\bigg[\bigg(\sum_{p=0}^{\infty}\partial^{v_1}_{\sigma}(\xi^{k,k'}_{p,t}-\tilde{\xi}^{k,k'}_{p,t})\bigg)^{2q}|\Pi\bigg]^{\frac{1}{q}}dt\bigg)^{\frac{1}{2}}\bigg)^q. 
\end{eqnarray}
By Lemma \ref{Fp-est}, (\ref{R1-est2}), (\ref{xi-diff-decom}) and the estimates after that, we have 
\begin{eqnarray}
&&E[|\partial^{v}_{\sigma}R_1|^q|\Pi] \nonumber \\
&\leq & C(E[R^C]+T^q(\Theta^1_{8q})^{\frac{1}{4}})\bigg(\sum_{k,k'}\bigg(\int E\bigg[\bigg(\sum_{p=0}^{\infty}(p+1)^{v+3}R^{2v+9} \nonumber \\
&&\times (1-s_n(T))^{(2p-7)\vee 0} (\tilde{M}_p)_{k,k'}|\theta_{2p+2,k}|^{\frac{1}{2}}(T^{\frac{1}{2}}\vee 1)1_{\theta_{2p+2,k'}}(t)\bigg)^{2q}\bigg|\Pi\bigg]^{\frac{1}{q}}dt\bigg)^{\frac{1}{2}}\bigg)^q \nonumber \\
&\leq &C(E[R^C]+T^q(\Theta^1_{8q})^{\frac{1}{4}})(T^{\frac{q}{2}}\vee 1)E[R^Cs_n(T)^{-(2v+2r+7)q}|\Pi]^{\frac{1}{2}} 
 \bigg(\sum_{k,k'}\bigg(\sum_{p=0}^{\infty}\frac{((\tilde{M_p})_{k,k'})^2|\theta_{2p+2,k}||\theta_{2p+2,k'}|}{(p+1)^{2r}}\bigg)^{\frac{1}{2}}\bigg)^q \nonumber \\
&\leq & C(E[R^C]+T^q(\Theta^1_{8q})^{\frac{1}{4}})(T^{\frac{q}{2}}\vee 1)E[R^Cs_n(T)^{-(2v+2r+7)q}|\Pi]^{\frac{1}{2}}\bigg(\sum_{p=0}^{\infty}\frac{\Phi_{2p+2,1}}{(p+1)^r}\bigg)^q. \nonumber
\end{eqnarray}
Then $E[|\partial^{v}_{\sigma}R_1|^q|\Pi]$ is less than the right-hand side of (\ref{H1-diff-ineq}). 
Similarly, we can see
\begin{equation*}
E[|\partial^{v}_{\sigma}R_3|^q] \leq CE[R^Cs_n(T)^{-(2v+2r+5)q}|\Pi]^{\frac{1}{2}}(E[R^C]+T^q(\Theta^1_{8q})^{\frac{1}{4}})\bigg(\sum_{p=0}^{\infty}\frac{\sqrt{(l_n+m_n)\Phi_{2p+2,2}}}{(p+1)^r}\bigg)^q.
\end{equation*}
Hence $E[|\partial^{v}_{\sigma}(Z^{\star}(M-\tilde{M})Z/2)|^q|\Pi]$ is less than the right-hand side of (\ref{H1-diff-ineq}). 

Similaly, we can see 
\begin{equation*}
E\bigg[\bigg|\partial^{v}_{\sigma}\bigg(\frac{1}{2}\sum_{p=1}^{\infty}\frac{(-1)^p}{p}({\rm tr}(\tilde{L}^p)-{\rm tr}(\tilde{L}_p^p))\bigg)\bigg|^q\bigg|\Pi\bigg]
\end{equation*}
is less than the right-hand side of (\ref{H1-diff-ineq}). 

Moreover, 
\begin{eqnarray}
E[|\partial_{\sigma}^v(\tilde{\mu}^{\star}\tilde{M}\tilde{\mu}/2)|^q|\Pi] \nonumber \leq  CE\bigg[\bigg(\sum_{v_1=0}^v|\partial_{\sigma}^{v_1}\tilde{\mu}|\bigg)^{2q}\bigg(\sum_{v_2=0}^v\parallel \partial_{\sigma}^{v_2}\tilde{M} \parallel\bigg)^{q} \bigg|\Pi\bigg] 
\leq CE[(R^Cs_n(T)^{-(v+1)q})^2|\Pi]^{\frac{1}{2}}(2T)^q(\Theta^1_{4q})^{\frac{1}{2}} \nonumber
\end{eqnarray}
since $\sum_{v_2=0}^v\parallel \partial_{\sigma}^{v_2}\tilde{M} \parallel \leq CR^{2v}\sum_{p=0}^{\infty}((2p)^v\vee 1)\bar{\rho}_{\tau_n}^{(2p-v)\vee 0}\leq CR^{2v}s_n(T)^{-(v+1)}$.

We will estimate $E[|\partial_{\sigma}^v(\tilde{\mu}^{\star}\tilde{M}\hat{Z})|^q|\Pi]$ at last. 
Let $\mathcal{L}_{p,k,k'}=((\tilde{L}_{2p})^{2p}-(\tilde{L}_{2p+1})^{2p+1})_{k,k'}$.
$\tilde{\mu}^{\star}\tilde{M}\hat{Z}$ can be decomposed as 
\begin{eqnarray}
\tilde{\mu}^{\star}\tilde{M}\hat{Z}&=&\sum_{k,k'}\sum_{p=0}^{\infty}\mathcal{L}_{p,k,k'}\tilde{\mu}_{p,k}\hat{Z}_{k'} +\sum_{k,k'}\sum_{p=0}^{\infty}\mathcal{L}_{p,k,k'}(\tilde{\mu}_k-\tilde{\mu}_{p,k})\hat{Z}_{k'} =\Xi_1+\Xi_2, \nonumber
\end{eqnarray}
where 
\begin{eqnarray}
\tilde{\mu}_{p,k}=\left\{
\begin{array}{ll}
\mu^1_{L(\theta_{p+1,k})}\sqrt{|I^k|}/|b^1_{L(\theta_{p+1,k})\wedge\tau_n}| & (k\leq l_n) \\
\mu^2_{L(\theta_{p+1,k})}\sqrt{|J^{k-l_n}|}/|b^2_{L(\theta_{p+1,k})\wedge\tau_n}| & (k>l_n) \\
\end{array}
\right. \nonumber
\end{eqnarray} 
Then by the Burkholder-Devis-Gundy inequality, we obtain
\begin{eqnarray}
E[|\partial_{\sigma}^v\Xi_1|^q|\Pi]&\leq &C\sum_{v_1+v_2=v}E\bigg[\bigg(\sum_{k'}\bigg(\sum_k\sum_{p=0}^{\infty}\partial_{\sigma}^{v_1}(\mathcal{L}_{p,k,k'}\tilde{\mu}_{p,k})\bigg)^2(\partial_{\sigma}^{v_2}\hat{Z}_{k'})^2\bigg)^{\frac{q}{2}}\bigg|\Pi\bigg] \nonumber \\
&\leq & C\sum_{v_1+v_2=v}\bigg(\sum_{k'}E\bigg[\bigg(\sum_k\sum_{p=0}^{\infty}\partial_{\sigma}^{v_1}(\mathcal{L}_{p,k,k'}\tilde{\mu}_{p,k})\bigg)^q(\partial_{\sigma}^{v_2}\hat{Z}_{k'})^q\bigg|\Pi\bigg]^{\frac{2}{q}}\bigg)^{\frac{q}{2}} \nonumber \\
&\leq & CE[R^{4q}]^{\frac{1}{2}}\sum_{v_1=0}^v\bigg(\sum_{k'}\bigg(\sum_k\sum_{p=0}^{\infty}E[(\partial_{\sigma}^{v_1}(\mathcal{L}_{p,k,k'}\tilde{\mu}_{p,k}))^{2q}|\Pi]^{\frac{1}{2q}}\bigg)^2\bigg)^{\frac{q}{2}}. \nonumber
\end{eqnarray}
Since
\begin{equation*}
E[(\partial_{\sigma}^{v_1}(\mathcal{L}_{p,k,k'}\tilde{\mu}_{p,k}))^{2q}|\Pi]^{\frac{1}{2q}}\leq C(2p+1)^vE[R^C\bar{\rho}_{\tau_n}^{4q(2p-5)}|\Pi]^{\frac{1}{4q}}(\tilde{M}_p)_{k,k'}(\Theta^1_{4q})^{\frac{1}{4q}}\sqrt{|\theta_{0,k}|},
\end{equation*}
we obtain
\begin{eqnarray}
E[|\partial_{\sigma}^v\Xi_1|^q|\Pi]&\leq &CE[R^{4q}]^{\frac{1}{2}}(\Theta^1_{4q})^{\frac{1}{4}}\bigg(\sum_{k'}\sum_{k_1,k_2}\sum_{p_1,p_2=0}^{\infty}E[R^C\bar{\rho}_{\tau_n}^{4q(2p_1-5)}|\Pi]^{\frac{1}{4q}}(p_1+1)^v(p_2+1)^v \nonumber \\
&&\times E[R^C\bar{\rho}_{\tau_n}^{4q(2p_2-5)}|\Pi]^{\frac{1}{4q}}(\tilde{M}_{p_1})_{k_1,k'}(\tilde{M}_{p_2})_{k_2,k'}\sqrt{|\theta_{0,k_1}|}\sqrt{|\theta_{0,k_2}|}\bigg)^{\frac{q}{2}} \nonumber \\
&\leq &CT^{q/2}E[R^{4q}]^{\frac{1}{2}}(\Theta^1_{4q})^{\frac{1}{4}}\bigg(\sum_{p=0}^{\infty}	E[R^C\bar{\rho}_{\tau_n}^{4q(2p-5)}|\Pi]^{\frac{1}{4q}}(p+1)^v\bigg)^q. \nonumber
\end{eqnarray}
Then $E[|\partial_{\sigma}^v\Xi_1|^q|\Pi]$ is less than the right-hand side of (\ref{H1-diff-ineq}) since
\begin{eqnarray}
\sum_{p=0}^{\infty}E[R^C\bar{\rho}_{\tau_n}^{4q(2p-5)}|\Pi]^{\frac{1}{4q}}(p+1)^v&\leq & \bigg(\sum_{p=0}^{\infty}E[R^C\bar{\rho}_{\tau_n}^{4q(2p-5)}|\Pi](p+1)^{q(4v+\frac{9}{2})}\bigg)^{\frac{1}{4q}}\bigg(\sum_{p=0}^{\infty}(p+1)^{-\frac{9}{8}}\bigg)^{1-\frac{1}{4q}} \nonumber \\
&\leq & CE[R^Cs_n(T)^{-q(4v+5)}]^{1/4q}. \nonumber 
\end{eqnarray}

On the other hand, Lemma \ref{q-sum-est} yields
\begin{eqnarray}
&&E[|\partial_{\sigma}^v\Xi_2|^q|\Pi] \nonumber \\
&\leq &C\bigg(\sum_{p=0}^{\infty}\sum_{k,k'}E\bigg[\bigg\{R^C\bar{\rho}_{\tau_n}^{(2p-v)\vee 0}(2p+1)^v(\tilde{M}_p)_{k,k'}\bigg(\sum_{v_1=0}^v|\partial_{\sigma}^{v_1}(\tilde{\mu}_k-\tilde{\mu}_{p,k})|\bigg)\bigg(\sum_{v_2=0}^v|\partial_{\sigma}^{v_2}\hat{Z}_k|\bigg)\bigg\}^q\bigg|\Pi\bigg]^{\frac{1}{q}}\bigg)^q \nonumber \\
&\leq & CE[R^C]\bigg\{\sum_{p=0}\sum_{k,k'}(\tilde{M}_p)_{k,k'}E[R^C\bar{\rho}_{\tau_n}^{2q(2p-v)\vee 0}(2p+1)^{2qv}|\Pi]^{\frac{1}{2q}} \nonumber \\
&&\times \left(E[R^C]^{\frac{1}{8q}}\sqrt{|\theta_{0,k}|} (\Theta^2_{8q})^{\frac{1}{8q}}|\theta_{p+1,k}|^{\gamma}+E[R^C]^{\frac{1}{8q}}\sqrt{|\theta_{0,k}|}(\Theta^1_{8q})^{\frac{1}{8q}}(T^{\frac{1}{2}}\vee 1)\sqrt{|\theta_{p+1,k}|}\right)\bigg\}^q \nonumber \\
&\leq & CE[R^C](T^{\frac{q}{2}}\vee 1)((\Theta^1_{8q})^{\frac{1}{8}}+(\Theta^2_{8q})^{\frac{1}{8}})\bigg(\sum_{p=0}^{\infty}E[R^C\bar{\rho}_{\tau_n}^{2q(2p-4)\vee 0}|\Pi](p+1)^{(2v+6)q}\bigg)^{\frac{1}{2}} \nonumber \\
&&\times \bigg\{\sum_{p=0}^{\infty}\frac{1}{(p+1)^2}\bigg(\sum_{k,k'}(\tilde{M}_p)_{k,k'}\sqrt{|\theta_{0,k}|}\frac{((4p+1)r_n)^{\gamma}\vee ((4p+1)r_n)^{1/2}}{p+1}\bigg)^{\frac{2q}{2q-1}}\bigg\}^{\frac{2q-1}{2}} \nonumber \\
&\leq & CE[R^C]r_n^{q(\gamma\wedge \frac{1}{2})}(T^{\frac{3}{2}q}\vee 1)(l_n+m_n)^{\frac{q}{2}}((\Theta^1_{8q})^{\frac{1}{8}}+(\Theta^2_{8q})^{\frac{1}{8}})E[R^Cs_n(T)^{-(2v+7)q}|\Pi]^{\frac{1}{2}}, \nonumber
\end{eqnarray}
where we use the fact $r_n^{\gamma}\vee r_n^{\frac{1}{2}}= T^{\gamma}(r_nT^{-1})^{\gamma}\vee T^{\frac{1}{2}}(r_nT^{-1})^{\frac{1}{2}}\leq (\sqrt{T}\vee 1)r_n^{\gamma\wedge \frac{1}{2}}$.
This complete the proof of (\ref{H1-diff-ineq}).

We next estimate $E[|\partial_{\sigma}^v\tilde{H}^2_{n,s_n}(T;\sigma)-\partial_{\sigma}^v\tilde{H}^3_{n,s_n}(T;\sigma)|^q|\Pi]$.
Let $\mathcal{J}(k)=1 \ (1\leq k\leq l_n)$, $\mathcal{J}(k)=2 \ (l_n<k\leq l_n+m_n)$ 
and $\check{B}^i_k=|b^i_{L(\theta_{0,k}),\ast}|/|b^i_{L(\theta_{0,k})\wedge \tau_n}|$ for $1\leq k\leq l_n+m_n, i=1,2$.
For $p\in\mathbb{Z}_+$ and $1\leq k,k'\leq l_n+m_n$, we define $\{\check{\xi}^{k,k'}_{p,t}\},\{\check{\eta}^{k,k'}_{p,t}\}$ as if follows.
\begin{enumerate}
\item the case $k=k'$:
\begin{eqnarray}
&&-\frac{1}{2}((\tilde{L}_{2p})^{2p})_{k,k}(\langle\hat{Z}\rangle_T)_{k,k}+\frac{1}{2}(\check{B}^{\mathcal{J}(k)}_k)^2\rho_{L(\theta_{0,k})}^{2p}(\tilde{M}_p)_{k,k} \nonumber \\
&&-(\log|b^{\mathcal{J}(k)}_{\theta_{0,k},\tau_n}| -\log |b^{\mathcal{J}(k)}_{L(\theta_{0,k})\wedge \tau_n}|)1_{\{p=0\} }+\frac{1}{4p}(((\tilde{L}_{2p})^{2p})_{k,k}-\rho^{2p}_{L(\theta_{0,k})\wedge \tau_n}(\tilde{M}_p)_{k,k})1_{\{p\geq 1\} } \nonumber \\
&&=\int\check{\xi}^{k,k}_{p,t}\cdot d\mathcal{W}_t+\int\check{\eta}^{k,k}_{p,t}dt. \nonumber
\end{eqnarray}
\item the case ($k\leq l_n$ and $k'>l_n$) or ($k>l_n$ and $k'\leq l_n$):
\begin{eqnarray}
&&-\frac{1}{2}\big\{(\check{B}^1_k\check{B}^2_k+\check{B}^{\mathcal{J}(k)}_k)\rho^{2p+1}_{L(\theta_{0,k})\wedge \tau_n}\rho_{L(\theta_{0,k}),\ast}-\check{B}^{\mathcal{J}(k)}_{k'}\rho^{2p+1}_{L(\theta_{0,k'})\wedge \tau_n}\rho_{L(\theta_{0,k'}),\ast}\big\} \nonumber \\
&&\times (\tilde{M}_p)_{k,k'}(\tilde{M}_0)_{k,k'}+ \frac{1}{2}((\tilde{L}_{2p+1})^{2p+1})_{k,k'}(\langle Z\rangle_T)_{k,k'}=\int\check{\xi}^{k,k'}_{p,t}\cdot d\mathcal{W}_t+\int\check{\eta}^{k,k'}_{p,t}dt. \nonumber 
\end{eqnarray}
\item other case : We set $\check{\xi}^{k,k'}_{p,t}\equiv 0$ and $\check{\eta}^{k,k'}_{p,t}\equiv 0$.
\end{enumerate}
Then by Ito's formula, we obtain
\begin{eqnarray}
|\sum_{k,k'}\partial_{\sigma}^v\check{\xi}^{k,k'}_{p,t}|&\leq &CR^C(p+1)^{v+1}\bar{\rho}_T^{(2p-v-1)\vee 0}\sum_k\{(\tilde{M}_p)_{k,k}+(\tilde{M}_{p+1})_{k,k}\}1_{\theta_{2p+2,k}}(t), \nonumber \\
|\sum_{k,k'}\partial_{\sigma}^v\check{\eta}^{k,k'}_{p,t}|&\leq &CR^C(p+1)^{v+2}\bar{\rho}_T^{(2p-v-2)\vee 0}\sum_k\{(\tilde{M}_p)_{k,k}+(\tilde{M}_{p+1})_{k,k}\}1_{\theta_{2p+2,k}}(t). \nonumber
\end{eqnarray}
Moreover, we have 
\begin{equation*}
\tilde{H}^2_{n,s_n}(T;\sigma)-\tilde{H}^3_{n,s_n}(T;\sigma)=\int\sum_{k,k'}\sum_{p=0}^{\infty}\check{\xi}^{k,k'}_{p,t}\cdot d\mathcal{W}_t+\int\sum_{k,k'}\sum_{p=0}^{\infty}\check{\eta}^{k,k'}_{p,t}dt. \nonumber
\end{equation*}
Therefore we obtain the conclusion by Lemma \ref{Fp-est}.
\end{proof}

\begin{lemma}\label{Hn-H2n-diff}
\begin{enumerate}
\item Assume that $[A1],[A2]$ hold and $\{b_n^{-1}(l_n+m_n)\}_n$ be tight. Then 
\begin{equation*}
\sup_{\sigma}b_n^{-1}|\partial_{\sigma}^vH_n(\sigma)-\partial_{\sigma}^v\tilde{H}^3_{n,s_n}(T;\sigma)|\to^p 0
\end{equation*}
as $n\to\infty$ for $0\leq v\leq 3$, where $s_n=r_n^{1/42}\wedge ((1-|\rho_0|)/2)$.
\item Let $0\leq v\leq 3$, $q\in 2\mathbb{N},q>n_1$, $\delta\geq 1$, and $\{s_n\}_{n\in\mathbb{N}}$ be stochastic processes which satisfy $[S]$.
Assume that $[A1],[A2],[A4\mathchar`-(2q),\delta]$ hold,
\begin{equation}\label{sn-est}
\limsup_{n\to \infty}E[s_n(T)^{-(2v+2[\delta]+12)q}]<\infty \quad {\rm and} \quad \limsup_{n\to \infty}E[b_n^{-2q}(l_n+m_n)^{2q}]<\infty. 
\end{equation}
Then there exists $n_0\in\mathbb{N}$ such that
\begin{equation*}
\sup_{n\geq n_0} E\left[\left(\sup_{\sigma}b_n^{-1/2}|\partial_{\sigma}^v\hat{H}_n(\sigma;s_n)-\partial_{\sigma}^v\tilde{H}^3_{n,s_n}(T;\sigma)|\right)^q\right] <\infty.
\end{equation*}
\end{enumerate}
\end{lemma}

\begin{proof}
We first prove $2$. Since $\{r_n\}_n$ is bounded and $r_n\to^p0$ as $n\to \infty$, we have $\lim_{n\to\infty}E[r_n^{q'}]=0$ for any $q'>0$.
Then by $[A1],[A4\mathchar`-(2q),\delta],(\ref{sn-est})$, Lemma \ref{H1-diff} with $r=[\delta]+2$, Cauchy-Schwarz inequalities, Jensen's inequality
and the estimate $\Phi_{2p+2,2}\leq r_n(8p+9)\Phi_{2p+2,1}$,
we have $\lim_{n\to \infty}\sup_{\sigma}E[b_n^{-q/2}|\partial_{\sigma}^v(\hat{H}_n-\tilde{H}_n^1)|^q]=0$ for $0\leq v\leq 4$.
Hence $\lim_{n\to \infty}E[b_n^{-q/2}\sup_{\sigma}|\partial_{\sigma}^v(\hat{H}_n-\tilde{H}_n^1)|^q]=0$ for $0\leq v\leq 3$ by Sobolev's inequality.
Similarly, we have $\lim_{n\to \infty}E[b_n^{-q/2}\sup_{\sigma}|\partial_{\sigma}^v(\tilde{H}_n^2-\tilde{H}_n^3)|^q]=0$ for $0\leq v\leq 3$.

We estimate $\tilde{H}^1_n-\tilde{H}^2_n$ in the next step. Let $0\leq v\leq 4$ and $\Pi$ is deterministic. By Ito's formula and symmetry of $\tilde{M}$, we have
\begin{eqnarray}
\tilde{H}^1_n(t)-\tilde{H}^2_n(t)=-\frac{1}{2}\sum_{k,k'}\tilde{M}_{k,k'}\big\{\hat{Z}_{k,t}\hat{Z}_{k',t}-(\langle\hat{Z}\rangle_t)_{k,k'}\big\} =-\sum_{k,k'}\tilde{M}_{k,k'}\int^t_0\hat{Z}_{k,s}d\hat{Z}_{k',s}. \nonumber
\end{eqnarray}
Therefore, $\{\partial_{\sigma}^v(\tilde{H}^1_n(t)-\tilde{H}_n^2(t))\}_{0\leq t\leq T}$ is the martingale. By the Burkholder-Devis-Gundy inequality, we obtain
\begin{equation*}
E[|\partial_{\sigma}^v(\tilde{H}^1_n-\tilde{H}_n^2)|^q]\leq CE[\langle\partial_{\sigma}^v(\tilde{H}^1_n-\tilde{H}_n^2)\rangle_T^{q/2}] \quad (0\leq v\leq 4).
\end{equation*}
Moreover, 
\begin{eqnarray}
\langle\partial_{\sigma}^v(\tilde{H}^1_n-\tilde{H}^2_n)\rangle_T \leq  CR^4\parallel \tilde{M}_0\parallel\bigg(\sum_{0\leq j_1+j_2\leq v}|(\partial_{\sigma}^{j_1}\hat{Z})^{\ast}|\times \parallel \{|\partial_{\sigma}^{j_2}\tilde{M}_{k,k'}|\}_{k,k'}\parallel \bigg)^2, \nonumber
\end{eqnarray}
where $|(\partial_{\sigma}^{j_1}\hat{Z})^{\ast}|^2=\sum_k\sup_t|\partial^{j_1}_{\sigma}\hat{Z}_{k,t}|^2$.
Since $E[|(\partial_{\sigma}^j\hat{Z})^{\ast}|^{2q}]\leq CE[R^{4q}](l_n+m_n)^q, \parallel \tilde{M}_0\parallel\leq 2$ and $\parallel \{|\partial_{\sigma}^j\tilde{M}_{k,k'}|\}_{k,k'}\parallel \leq CR^{2j}(1-\bar{\rho}_T)^{-j-1}$ for $0\leq j\leq 4$, we have
\begin{eqnarray}
b_n^{-\frac{q}{2}}E[|\partial_{\sigma}^v(\tilde{H}^1_n-\tilde{H}_n^2)|^q]\leq CE[R^C]E[(b_n^{-1}(l_n+m_n))^{q/2}]E[R^C(1-\bar{\rho}_T)^{-2(v+1)q}]^{1/2} \nonumber
\end{eqnarray}
for general $\Pi$.
Then by Sobolev's inequality, there exists $n_0\in\mathbb{N}$ such that
\begin{equation*}
\sup_{n\geq n_0}E\left[\left(\sup_{\sigma}b_n^{-\frac{1}{2}}|\partial_{\sigma}^v(\tilde{H}^1_n-\tilde{H}^2_n)|\right)^q\right]<\infty
\end{equation*}
for $0\leq v\leq 3$. This completes the proof of $2$.

Finally, we prove $1$. Since $|\Phi_{2p+2,i}|\leq C(p+1)^ir_n^i(l_n+m_n)$ and $|\Phi_{2p_1+3,2p_2+3}|\leq C(p_1+1)(p_2+1)(l_n+m_n)^2r_n$ 
for $p_1,p_2\in\mathbb{Z}_+$ and $i=1,2$, by Lemma \ref{H1-diff} with $r=3$, we have
\begin{equation*}
\sup_{\sigma\in\mathbb{Q}^{n_1}\cap \Lambda}E[|\partial_{\sigma}^v(\hat{H}_n(\sigma;s_n)-\tilde{H}_{n,s_n}^1(T;\sigma))|^q|\Pi]\leq C(1+r_n^{-\frac{q}{4}})(1+r_n^{q(\gamma \wedge \frac{1}{2})}(l_n+m_n)^{\frac{q}{2}}+r_n^{\frac{q}{2}}(l_n+m_n)^q)
\end{equation*}
for $q\in2\mathbb{N}$, $q>n_1$ and $0\leq v \leq 4$. Therefore, by Lemma \ref{prob-conv} $2$., the assumptions and the inequality $T=\sum_I|I|\leq r_nl_n$, 
we obtain $\{b_n^{-1}r_n^{-1}\}_n$ is tight and
\begin{equation*}
\sup_{\sigma}b_n^{-1}|\partial_{\sigma}^v(\hat{H}_n(\sigma;s_n)-\tilde{H}_{n,s_n}^1(T;\sigma))|\to^p 0
\end{equation*}
as $n\to \infty$ for $0\leq v\leq 3$. 
Similarly, we obtain $\sup_{\sigma}b_n^{-1}|\partial_{\sigma}^v(\tilde{H}_{n,s_n}^2(T;\sigma)-\tilde{H}_{n,s_n}^3(T;\sigma)|\to^p 0$
as $n\to \infty$ for $0\leq v\leq 3$.

Moreover, similarly to the proof of $2$., we have
\begin{equation*}
E[|b_n^{-1}\partial_{\sigma}^v(\tilde{H}^1_n-\tilde{H}^2_n)|^q|\Pi]\leq CE[R^C]b_n^{-q/2}(b_n^{-1}(l_n+m_n))^{q/2}E[R^C(1-\bar{\rho}_T)^{-C}]
\end{equation*}
for $q>n_1$ and $0\leq v\leq 4$.
Hence by Lemma \ref{prob-conv} $2$., we have
$b_n^{-1}\sup_{\sigma}|\partial_{\sigma}^v(\tilde{H}^1_n-\tilde{H}_n^2)|\to^p 0$ as $n\to \infty$ for $0\leq v\leq 3$. 

Moreover, since $P[\tau(s_n)<T]\to 0$ as $n\to \infty$, 
$b_n^{-1}\sup_{\sigma}|\partial_{\sigma}^v(H_n(\sigma)-\hat{H}_n(\sigma;s_n))|\to^p 0$ as $n\to \infty$,
which completes the proof.
\end{proof}

\begin{lemma}\label{A3''}
Assume $[A3']$. Then 
\begin{equation*}
\sup_{\sigma\in\Lambda}|\Psi^{p,1}(f(\cdot,\sigma),a_p)|\to^p 0 \quad {\rm and} \quad
\sup_{\sigma\in\Lambda}|\Psi^{p,2}(f(\cdot,\sigma),c_p)|\to^p 0
\end{equation*}
as $n\to \infty$ for $p\in\mathbb{Z}_+$ and $f(t,\sigma)$ : random variable defined on $[0,T]\times \bar{\Lambda}$ such that
$f$ is continuous with respect to $(t,\sigma)$.
\end{lemma}
\begin{proof}
Let $\{f^k\}_k$ be step functions such that $\sup_{t,\sigma}|f(t,\sigma)-f^k(t,\sigma)|\to^p 0$ as $k\to \infty$.
By $[A3']$, we obtain $\sup_{\sigma}|\int^T_0f^k(t,\sigma)\nu^{p,1}_n(dt)-\int^T_0f^k(t,\sigma)a_p(t)dt|\to^p 0$ as $n\to \infty$ for any $k\in\mathbb{N}$.

Since $\{\nu^{0,1}_n([0,T))\}_n$ is tight, for any $\epsilon, \delta>0$, there exists $K\in \mathbb{N}$ such that
\begin{equation*}
P\bigg[\sup_{\sigma}\bigg|\int^T_0(f-f^k)\nu^{p,1}_n(dt)\bigg|\vee \sup_{\sigma}\bigg|\int^T_0(f-f^k)a_p(t)dt\bigg|>\delta\bigg]<\epsilon \ (k\geq K, n\in \mathbb{N}).
\end{equation*}
Then there exists $N\in\mathbb{N}$ such that
$P[\sup_{\sigma}|\Psi^{p,1}(f(\cdot,\sigma),a_p)|> 3\delta ]<2\epsilon$
for $n\geq N$. Similarly, we have $\sup_{\sigma}|\Psi^{p,2}(f(\cdot,\sigma),c_p)|\to^p 0$ as $n\to \infty$.
\end{proof}
\noindent
{\bf Proof of Proposition \ref{Hn-lim}.}

We first prove $1$. By Lemma \ref{Hn-H2n-diff}, it is sufficient to show 
$\sup_{\sigma}|b_n^{-1}\partial_{\sigma}^v\tilde{H}^3_n(T;\sigma)-\int^T_0\partial_{\sigma}^vh^{\infty}_t(\sigma)dt|\to^p 0$ 
as $n\to \infty$ for $0\leq v\leq 3$, where $s_n=r_n^{1/42}\wedge ((1-|\rho_0|)/2)$.

Since $P[\tau(s_n)<T]\to 0$ as $n\to \infty$, 
\begin{equation*}
\sum_{p=0}^{\infty}\sum_{i=1}^2\sup_{\sigma}\bigg|\int^T_0(\partial_{\sigma}^v\mathcal{D}^i_p(t\wedge \tau(s_n),t;\sigma)-\partial_{\sigma}^v\mathcal{D}^i_p(t,t;\sigma))\nu^{p,i}_n(dt)\bigg|\to^p 0
\end{equation*}
as $n\to \infty$ for $0\leq v\leq 3$. 

Moreover, by Lemma \ref{A3''}, 
$\sup_{\sigma}|\Psi^{p,1}(\partial_{\sigma}^v\mathcal{D}^1_p(\cdot,\cdot;\sigma),a_p)| \to^p 0$
as $n\to \infty$ for $p\in\mathbb{Z}_+$ and $0\leq v\leq 3$.

Then by Lemma \ref{sum-conv}, the tightness of $\{\nu^{0,1}_n([0,T))\}_n$, and the estimates $\nu^{p,1}_n([0,T))\leq \nu^{0,1}_n([0,T))$ 
and $|\partial_{\sigma}^v\mathcal{D}^1_p(t,t,\sigma)|\leq CR^C\bar{\rho}_T^{(2p-v)\vee 0}(p+1)^v$, we have
\begin{equation*}
\sum_{p=0}^{\infty}\sup_{\sigma}|\Psi^{p,1}(\partial_{\sigma}^v\mathcal{D}^1_p(\cdot,\cdot;\sigma),a_p)| \to^p 0
\end{equation*}
as $n\to \infty$ for $0\leq v\leq 3$. Similarly, we obtain
\begin{equation*}
\sum_{p=0}^{\infty}\sup_{\sigma}|\Psi^{p,2}(\partial_{\sigma}^v\mathcal{D}^2_p(\cdot,\cdot;\sigma),c_p)| \to^p 0
\end{equation*}
as $n\to \infty$ for $0\leq v\leq 3$. 
Since 
$h^{\infty}_t(\sigma)=\sum_{p=0}^{\infty}(\mathcal{D}^1_p(t,t;\sigma) a_p(t)+\mathcal{D}^2_p(t,t;\sigma) c_p(t))$,
we obtain $1$.

We next prove $2$.
First, $[S\mathchar`-((2v+2[\delta]+12)q),\xi]$ and the estimate $\nu^{p,i}_n([0,T))\leq \nu^{0,i}([0,T))\leq b_n^{-1}(l_n+m_n) \ (p\in\mathbb{Z}_+)$ 
yield 
\begin{equation*}
\sup_nE\bigg[\sup_{\sigma}\bigg|b_n^{\frac{\xi}{2q}}\sum_{p=0}^{\infty}\sum_{i=1}^2\int^T_0\big\{\partial_{\sigma}^v\mathcal{D}^i_p(t\wedge \tau(s_n),t;\sigma)-\partial_{\sigma}^v\mathcal{D}^i_p(t,t;\sigma)\big\}\nu^{p,i}_n(dt)\bigg|^q\bigg]<\infty
\end{equation*}
for $0\leq v\leq 3$.

Then by Lemma \ref{Hn-H2n-diff}, it is sufficient to show that there exists $n_0\in\mathbb{N}$ such that
\begin{equation*}
\sup_{n\geq n_0}E\bigg[\sup_{\sigma}\bigg|b_n^{\eta}\sum_{p=0}^{\infty}\big\{\Psi^{p,1}(\partial_{\sigma}^v\mathcal{D}^1_p(\cdot,\cdot;\sigma),a_p)+\Psi^{p,2}(\partial_{\sigma}^v\mathcal{D}^2_p(\cdot,\cdot;\sigma),c_p)\big\}\bigg|^q\bigg]<\infty.
\end{equation*}

By $[A3'\mathchar`-q,\eta]$ and independence of $\{\Pi_n\}_n$ and $X$, we have
\begin{eqnarray}\label{psi-sum-est}
&&\sup_{n\geq n_0}E\bigg[\bigg|b_n^{\eta}\sum_{p=0}^{\infty}\Psi^{p,1}(\partial_{\sigma}^v\mathcal{D}^1_p(\cdot,\cdot;\sigma),a_p)\bigg|^q\bigg] \nonumber \\
&\leq & C\sum_{p=0}^{\infty}\frac{1}{(p+1)^2}\sup_{n\geq n_0}E[(b_n^{\eta}(p+1)^2|\Psi^{p,1}(\partial_{\sigma}^v\mathcal{D}^1_p(\cdot,\cdot;\sigma),a_p)|)^q] \nonumber \\
&\leq &C\sum_{p=0}^{\infty}(p+1)^{C}E\left[\sup_t|\partial_{\sigma}^v\mathcal{D}^1_p(t,t;\sigma)|^q+\omega_{\alpha}(\partial_{\sigma}^v\mathcal{D}^1_p(\cdot,\cdot;\sigma))^q\right] \nonumber \\
\end{eqnarray}
for $0\leq v\leq 4$, $\alpha$ in $[A3'\mathchar`-q,\eta]$ and $n_0$ which is renewed if necessary.

By Ito's formula, we obtain
\begin{equation*}
E[|\partial_{\sigma}^v\mathcal{D}^1_p(t,t;\sigma)-\partial_{\sigma}^v\mathcal{D}^1_p(s,s;\sigma)|^q]\leq CE[((p+1)^{v+2}\bar{\rho}_T^{(2p-v-2)\vee 0}R^C)^q]|t-s|^{q/2}
\end{equation*}
for $s<t$.

Hence by {\colorg Kolmogorov's} criterion(\cite{Rev-Yor} Theorem (2.1)) and its proof, we have
\begin{equation}\label{omega-est}
E[\omega_{\alpha}(\partial_{\sigma}^v\mathcal{D}^1_p(\cdot,\cdot;\sigma))^q]\leq CE[((p+1)^{v+2}\bar{\rho}_T^{(2p-v-2)\vee 0}R^C)^q].
\end{equation}
(\ref{psi-sum-est}),(\ref{omega-est}) yield $\sup_{\sigma}\sup_{n\geq n_0}E[|b_n^{\eta}\sum_{p=0}^{\infty}\Psi^{p,1}(\partial_{\sigma}^v\mathcal{D}^1_p(\cdot,\cdot;\sigma),a_p)|^q]<\infty$.
Then by Sobolev's inequality, we have 
$\sup_{n\geq n_0}E[\sup_{\sigma}|b_n^{\eta}\sum_{p=0}^{\infty}\Psi^{p,1}(\partial_{\sigma}^v\mathcal{D}^1_p(\cdot,\cdot;\sigma),a_p)|^q]<\infty$
for $0\leq v\leq 3$. Similarly, there exists $n_1\in\mathbb{N}$ such that\\
$\sup_{n\geq n_1}E[\sup_{\sigma}|b_n^{\eta}\sum_{p=0}^{\infty}\Psi^{p,2}(\partial_{\sigma}^v\mathcal{D}^2_p(\cdot,\cdot;\sigma),c_p)|^q]<\infty$
for $0\leq v\leq 3$.
\qed

\subsection{Proof of Lemmas \ref{sep-lem} and \ref{y0-est}}

{\bf Proof of Lemma \ref{sep-lem}.}

Let $G_{[s,t)}=\{G_{I,J}\}_{L(I),L(J)\in [s,t)}$ for $0\leq s<t\leq T$,
$\{\lambda'_i\}_{i=1}^{l'}$ be the eigenvalues of $G_{[s,t)}G^{\star}_{[s,t)}$
and $f^{(s)}_1(t)=f_1(t,B^1_sB^2_s,\rho_s,\rho_{s,\ast})$.
Since $|b_n^{-1}{\rm tr}((G_{[s,t)}G_{[s,t)}^{\star})^p)-\nu^{p,1}_n([s,t))|\to^p 0$ as $n\to \infty$
by a similar argument to the proof of Lemma \ref{A3equiv}, we have
\begin{equation*}
\int^t_sa_p(u)du=\underset{n\to \infty}{\rm P\mathchar`-\lim} \ b_n^{-1}\sum_{i=1}^{l'}(\lambda'_i)^p,
\end{equation*}
where $\rm P\mathchar`-\lim$ denotes the limit in probability.
Moreover, similarly to the proof of Lemma \ref{GGT-lambda}, we have $\sup_i|\lambda'_i|\leq 1$.

Let $g_i=g_i(\rho_s)=\sqrt{1-\lambda'_i\rho_s^2}$, $g_{i,\ast}=g_i(\rho_{s,\ast})$.
Then since 
\begin{eqnarray}
\mathcal{A}(\rho)\frac{\rho_{\ast}}{\rho}-\mathcal{A}(\rho)-a_0&=&\sum_{p=1}^{\infty}a_p\rho^{2p}\left(\frac{\rho_{\ast}}{\rho}-1\right)-a_0=\sum_{p=0}^{\infty}a_{p+1}\rho^{2p+1}\rho_{\ast}-\sum_{p=0}^{\infty}a_p\rho^{2p}, \nonumber \\
\int^{\rho}_{\rho_{\ast}}\frac{\mathcal{A}(\rho')}{\rho'}d\rho'&=&\frac{1}{2}\sum_{p=1}^{\infty}\frac{a_p}{p}(\rho^{2p}-\rho_{\ast}^{2p}) \nonumber 
\end{eqnarray}
for $\rho,\rho_{\ast}\in (-1,1)$,
we have
\begin{eqnarray}\label{y-est}
&&\int^t_s f^{(s)}_1(u)du \nonumber \\
&=&\underset{n\to \infty}{\rm P\mathchar`-\lim} \ b_n^{-1}\sum_{i=1}^{l'}\bigg\{1+\log (B^1_sB^2_s)+B^1_sB^2_s\sum_{p=0}^{\infty}\left((\lambda'_i)^{p+1}\rho_s^{2p+1}\rho_{s,\ast}-(\lambda'_i)^p\rho^{2p}_s\right) 
+\frac{1}{2}\sum_{p=1}^{\infty}\frac{(\lambda'_i)^p}{p}(\rho^{2p}_s-\rho^{2p}_{s,\ast}) \bigg\} \nonumber \\
&=&\underset{n\to \infty}{\rm P\mathchar`-\lim} \ b_n^{-1}\sum_{i=1}^{l'}\left\{1+B^1_sB^2_sg_i^{-2}(\lambda'_i\rho_s\rho_{s,\ast}-1)+\log(B^1_sB^2_sg_{i,\ast}g_i^{-1})\right\} \nonumber \\
&=&\underset{n\to \infty}{\rm P\mathchar`-\lim} \ b_n^{-1}\sum_{i=1}^{l'}\left\{B^1_sB^2_sg_i^{-2}(\lambda'_i\rho_s\rho_{s,\ast}-1)+B^1_sB^2_sg_{i,\ast}g_i^{-1}+F(B^1_sB^2_sg_{i,\ast}g_i^{-1}) \right\} 
\end{eqnarray}
by Lemma \ref{sum-conv}.

Since
\begin{eqnarray}\label{y-second-est}
g_i^{-2}(\lambda'_i\rho_s\rho_{s,\ast}-1)+g_{i,\ast}g_i^{-1} &=&-\frac{(\lambda'_i\rho_s\rho_{s,\ast}-1)^2-g_{i,\ast}^2g_i^2}{g_i^2(1-\lambda'_i\rho_s\rho_{s,\ast}+g_{i,\ast}g_i)} 
=-\frac{\lambda'_i(\rho_s-\rho_{s,\ast})^2}{g_i^2(1-\lambda'_i\rho_s\rho_{s,\ast}+g_{i,\ast}g_i)} \leq -\lambda'_i(\rho_s-\rho_{s,\ast})^2/3 \nonumber \\
\end{eqnarray}
and $B^1_sB^2_sg_{i,\ast}g_i^{-1}-1\leq R^4/\sqrt{1-\bar{\rho}_T^2}$, it follows that 
\begin{eqnarray}
\int^t_s f^{(s)}_1(u)du&\leq &\underset{n\to \infty}{\rm P\mathchar`-\lim} \ b_n^{-1}\sum_{i=1}^{l'}\left\{-\frac{B^1_sB^2_s\lambda'_i}{3}(\rho_s-\rho_{s,\ast})^2-\frac{1-\bar{\rho}_T^2}{4R^8}(B^1_sB^2_sg_{i,\ast}g_i^{-1}-1)^2\right\} \nonumber
\end{eqnarray}
{\colorg from} (\ref{y-est}),(\ref{y-second-est}) and Lemma \ref{Fx-est}.

Moreover, since
\begin{eqnarray}
(B^1_sB^2_sg_{i,\ast}g_i^{-1}-1)^2 &\geq & (B^1_sB^2_sg_{i,\ast}-g_i)^2 \geq g_{i,\ast}^2(B^1_sB^2_s-1)^2/2-(g_i-g_{i,\ast})^2 \nonumber \\
&=& g_{i,\ast}^2(B^1_sB^2_s-1)^2/2-(\lambda'_i)^2(\rho_s-\rho_{s,\ast})^2(\rho_s+\rho_{s,\ast})^2/(g_i+g_{i,\ast})^2 \nonumber \\
&\geq & (1-\bar{\rho}_T^2)(B^1_sB^2_s-1)^2/2-\lambda'_i(\rho_s-\rho_{s,\ast})^2/(1-\bar{\rho}_T^2), \nonumber
\end{eqnarray}
we obtain 
\begin{eqnarray}
&&\int^t_sf_1^{(s)}(u)du \nonumber \\
&\leq &-\frac{B^1_sB^2_s(\rho_s-\rho_{s,\ast})^2}{3}\int^t_sa_1(u)du -\frac{1-\bar{\rho}_T^2}{4R^8} 
\left\{\frac{(1-\bar{\rho}_T^2)(B^1_sB^2_s-1)^2}{2}\int^t_sa_0(u)du-\frac{(\rho_s-\rho_{s,\ast})^2}{1-\bar{\rho}_T^2}\int^t_sa_1(u)du\right\} \nonumber \\
&= & -\left(\frac{B^1_sB^2_s}{3}-\frac{1}{4R^8}\right)(\rho_s-\rho_{s,\ast})^2\int^t_sa_1(u)du
-\frac{(1-\bar{\rho}_T^2)^2}{8R^8}(B^1_sB^2_s-1)^2\int^t_sa_0(u)du \nonumber \\
&\leq & -C_1\int^t_s\left\{a_1(u)(\rho_s-\rho_{s,\ast})^2+a_0(u)(B^1_sB^2_s-1)^2\right\}du. \nonumber
\end{eqnarray}
Since $s<t$ is arbitrary, we obtain
\begin{equation*}
\int^t_sf_1(u,B^1_uB^2_u,\rho_u,\rho_{u,\ast})du\leq -C_1\int^t_s\big\{a_1(u)(\rho_u-\rho_{u,\ast})^2+a_0(u)(B^1_uB^2_u-1)^2\big\}du.
\end{equation*}
Then we have
\begin{equation*}
f_1(t,B^1_tB^2_t,\rho_t,\rho_{t,\ast})\leq  -C_1\big\{a_1(t)(\rho_t-\rho_{t,\ast})^2+a_0(t)(B^1_tB^2_t-1)^2\big\} \quad dt\times P\mathchar`- \ {\rm a.e.} \ (t,\omega).
\end{equation*}
Similar argument using the eigenvalues of $G^{\star}_{[s,t)}G_{[s,t)}$ instead of that of $G_{[s,t)}G^{\star}_{[s,t)}$ yields
\begin{equation*}
f_2(t,B^1_tB^2_t,\rho_t,\rho_{t,\ast})\leq  -C_1\big\{a_1(t)(\rho_t-\rho_{t,\ast})^2+c_0(t)(B^1_tB^2_t-1)^2\big\} \quad dt\times P\mathchar`- \ {\rm a.e.} \ (t,\omega).
\end{equation*}
\qed
\\
\\
{\bf Proof of Lemma \ref{y0-est}.}

For the case that observation intervals $\{I\},\{J\}$ are synchronous and equi-spaced : 
$|I|=|J|=T/[b_n]$, we obtain $a_0\equiv c_0\equiv a_1\equiv 1$, $\mathcal{A}(\rho)=\rho^2/(1-\rho^2)$. 
Let us denote $y_t$ by $y_{t,0}$ for the synchronous and equi-spaced sampling case, then by (\ref{y-rep}) we have
\begin{eqnarray}
y_{t,0}&=&-\frac{(B^1_t-B^2_t)^2}{2(1-\rho^2_t)}+1+\log B^1_tB^2_t+\frac{1}{2}\log\frac{1-\rho_{t,\ast}^2}{1-\rho_t^2}+B^1_tB^2_t\frac{\rho_t\rho_{t,\ast}-1}{1-\rho^2_t} \nonumber \\
&=&-\frac{(B^1_t-B^2_t)^2}{2(1-\rho^2_t)}+F\left(B^1_tB^2_t\sqrt{\frac{1-\rho_{t,\ast}^2}{1-\rho_t^2}}\right)+B^1_tB^2_t\left(\frac{\rho_t\rho_{t,\ast}-1}{1-\rho^2_t}+\sqrt{\frac{1-\rho_{t,\ast}^2}{1-\rho_t^2}}\right). \nonumber
\end{eqnarray}
Since $B^1_tB^2_t\sqrt{1-\rho_{t,\ast}^2}/\sqrt{1-\rho_t^2}\geq R^{-4}\sqrt{1-\bar{\rho}^2_T}$,
by Lemma \ref{Fx-est} and similar argument to (\ref{y-second-est}), it follows that
\begin{eqnarray}
y_{t,0}&\geq &-\frac{(B^1_t-B^2_t)^2}{2(1-\bar{\rho}^2_T)}-\left(\log\frac{R^4}{\sqrt{1-\bar{\rho}^2_T}}\vee 1\right)\left(B^1_tB^2_t\sqrt{\frac{1-\rho_{t,\ast}^2}{1-\rho_t^2}}-1\right)^2-R^4\frac{(\rho_t-\rho_{t,\ast})^2}{(1-\bar{\rho}^2_T)^2}. \nonumber
\end{eqnarray}
Since
\begin{eqnarray}
(B^1_tB^2_t\sqrt{1-\rho_{t,\ast}^2}/\sqrt{1-\rho_t^2}-1)^2&\leq &\frac{2(B^1_tB^2_t-1)^2+2(\sqrt{1-\rho_{t,\ast}^2}-\sqrt{1-\rho_t^2})^2}{1-\bar{\rho}^2_T} \nonumber \\
&\leq & \frac{2(B^1_tB^2_t-1)^2}{1-\bar{\rho}^2_T}+\frac{(\rho_{t,\ast}-\rho_t)^2(\rho_{t,\ast}+\rho_t)^2}{2(1-\bar{\rho}^2_T)^2}, \nonumber
\end{eqnarray}
there exists a positive random variable $\mathcal{R}'$ which does not depend on $\sigma,\sigma_{\ast},t$
such that $E[(\mathcal{R}')^q]<\infty$ for any $q>0$ and
\begin{equation*}
y_{t,0}\geq -\mathcal{R}'\left\{(B^1_t-B^2_t)^2+(B^1_tB^2_t-1)^2+(\rho_t-\rho_{t,\ast})^2\right\}.
\end{equation*}
By integrating with respect to $t$, we have the desired conclusion.
\qed

\subsection{Proof of Proposition \ref{pld} and Theorem \ref{asym-dist}}
{ \bf Proof of Proposition \ref{pld}.}

We use Theorem $2$ in Yoshida \cite{yos05}. 

Let $\beta_1=\delta$, $\beta_2=1/2-\delta$, $0<\rho'_2<\delta, 0<\alpha<1\wedge (\rho_2'/2)$, $\beta=\alpha/(1-\alpha)$ 
and $0<\rho'_1<1\wedge \beta \wedge (2\beta_1/(1-\alpha))$. 
Let 
\begin{equation*}
\hat{\mathcal{Y}}_n(\sigma;\sigma_{\ast})=b_n^{-1}(\hat{H}_n(\sigma)-\hat{H}_n(\sigma_{\ast})), \quad \hat{\Gamma}_n(\sigma)=-b_n^{-1}\partial_{\sigma}^2\hat{H}_n(\sigma),
\end{equation*}
then it is sufficient to prove the following five conditions for any $L>0$. 
\begin{enumerate}
\item There exists $c_L>0$ such that for any $r>0$, we have $P[\chi\leq r^{-(\rho'_2-2\alpha)}]\leq c_L/r^L$ and
$P[\{r^{-\rho'_1}|u|^2\leq u^{\star}\Gamma u/4 \ {\rm for \ any } \ u\in\mathbb{R}^{n_1}\}^c]\leq c_L/r^L$.
\item For $M_1=L(1-\rho'_1)^{-1}$, $\sup_nE[(b_n^{-1/2}|\partial_{\sigma}\hat{H}_n(\sigma_{\ast})|)^{M_1}]<\infty$.
\item For $M_2=L(1-2\beta_2-\rho'_2)^{-1}$, 
\begin{equation*}
\sup_nE\bigg[\bigg(\sup_{\sigma}b_n^{\frac{1}{2}-\beta_2}|\hat{\mathcal{Y}}_n(\sigma;\sigma_{\ast})-\mathcal{Y}(\sigma;\sigma_{\ast})|\bigg)^{M_2}\bigg]<\infty.
\end{equation*}
\item For $M_3=L(\beta-\rho'_1)^{-1}$, $\sup_nE[(b_n^{-1}\sup_{\sigma}|\partial_{\sigma}^3\hat{H}_n(\sigma)|)^{M_3}]<\infty$.
\item For $M_4=L(2\beta_1/(1-\alpha)-\rho'_1)^{-1}$, $\sup_nE[(b_n^{\beta_1}|\hat{\Gamma}_n(\sigma_{\ast})-\Gamma|)^{M_4}]<\infty$.
\end{enumerate}

By using Taylor's formula for $h^{\infty}_t(\sigma)-h^{\infty}_t(\sigma_{\ast})$, we obtain $\chi\leq \inf_{u\in\mathbb{R}^{n_1}\setminus \{0\}}u^{\star}\Gamma u/(2|u|^2)$.
Then $[H]$ yields $1$. 
Moreover, $3.$ and $5.$ obviously hold by Proposition \ref{Hn-lim} $2$.
By Proposition \ref{Hn-lim} and the estimate $E[(\sup_{\sigma}|\int^T_0\partial_{\sigma}^3h^{\infty}_t(\sigma)dt|)^{M_3}]<\infty$, $4.$ also holds.
Finally, Lemma \ref{Hn-H2n-diff}, $[S\mathchar`-q',2q'\delta]$ for some sufficiently large $q'$ and the estimate 
$\partial_{\sigma}\tilde{H}^3_n(T;\sigma_{\ast})\equiv 0$ on $\{\tau(s_n)=T\}$ show $2$.
\qed

\begin{proposition}\label{st-conv}
Assume $[A1]-[A4]$.
Then $(\mathcal{V}_n(u_1),\cdots, \mathcal{V}_n(u_k))\to^{s\mathchar`-\mathcal{L}} (\mathcal{V}(u_1),\cdots, \mathcal{V}(u_k))$
as $n\to \infty$ for $k\in\mathbb{N}$, $u_1,\cdots,u_k\in\mathbb{R}^{n_1}$, 
where $\mathcal{V}_n(u)=b_n^{-1/2}\partial_{\sigma}H_n(\sigma_{\ast})u+b_n^{-1}u^{\star}\partial_{\sigma}^2H_n(\sigma_{\ast})u/2$, 
$\mathcal{V}(u)=u^{\star}\Gamma^{1/2} \mathcal{N}-u^{\star}\Gamma u/2$ and
$\mathcal{N}$ is defined before the statement of Theorem \ref{asym-dist}.
Moreover, 
\begin{eqnarray}
\partial_{\sigma}^2h^{\infty}_t(\sigma_{\ast})&=&\mathcal{A}(\rho_{t,\ast})\left(\frac{\partial_{\sigma}\rho_{t,\ast}}{\rho_{t,\ast}}-\partial_{\sigma} B^1_{t,\ast}-\partial_{\sigma} B^2_{t,\ast}\right)^2-\partial_{\rho}\mathcal{A}(\rho_{t,\ast})\frac{(\partial_{\sigma} \rho_{t,\ast})^2}{\rho_{t,\ast}} \nonumber \\
&&-2(a_0(t)+\mathcal{A}(\rho_{t,\ast}))(\partial_{\sigma} B^1_{t,\ast})^2-2(c_0(t)+\mathcal{A}(\rho_{t,\ast}))(\partial_{\sigma} B^2_{t,\ast})^2. \nonumber
\end{eqnarray}
\end{proposition}

\begin{proof}
By (\ref{h-inf-A}) we have
\begin{eqnarray}
\partial_{\sigma} h^{\infty}_t&=& -\partial_{\sigma} B^1_t B^1_t(a_0+\mathcal{A}(\rho_t))-\frac{1}{2}(B^1_t)^2\partial_{\sigma} (\mathcal{A}(\rho_t))-\partial_{\sigma} B^2_t B^2_t(c_0+\mathcal{A}(\rho_t)) -\frac{1}{2}(B^2_t)^2\partial_{\sigma} (\mathcal{A}(\rho_t)) \nonumber \\
&&+(\partial_{\sigma} B^1_t B^2_t+B^1_t\partial_{\sigma} B^2_t)\mathcal{A}\frac{\rho_{t,\ast}}{\rho_t}+B^1_tB^2_t\rho_{t,\ast}\partial_{\sigma}\left(\frac{\mathcal{A}(\rho_t)}{\rho_t}\right) +a_0\frac{\partial_{\sigma} B^1_t}{B^1_t}+c_0\frac{\partial_{\sigma} B^2_t}{B^2_t}+\frac{\mathcal{A}(\rho_t)}{\rho_t}\partial_{\sigma} \rho_t \nonumber \\
&=& \partial_{\sigma} B^1_t \mathcal{A}(\rho_t)\left(B^2_t\frac{\rho_{t,\ast}}{\rho_t}-B^1_t\right)+\partial_{\sigma} B^2_t \mathcal{A}(\rho_t)\left(B^1_t\frac{\rho_{t,\ast}}{\rho_t}-B^2_t\right) +a_0\partial_{\sigma} B^1_t \left(\frac{1}{B^1_t}-B^1_t\right) \nonumber \\
&&+c_0\partial_{\sigma} B^2_t\left(\frac{1}{B^2_t}-B^2_t\right) +\partial_{\sigma} (\mathcal{A}(\rho_t))\left(B^1_tB^2_t\frac{\rho_{t,\ast}}{\rho_t}-\frac{(B^1_t)^2}{2}-\frac{(B^2_t)^2}{2}\right)+\mathcal{A}\frac{\partial_{\sigma} \rho_t}{\rho_t}\left(1-B^1_tB^2_t\frac{\rho_{t,\ast}}{\rho_t}\right).\nonumber 
\end{eqnarray}
Since $B^1_{t,\ast}=B^2_{t,\ast}=1$ and 
each term of the right-hand side of the previous equation has a factor which equals $0$ if we substitute $\sigma=\sigma_{\ast}$,
it follows that
\begin{eqnarray}
\partial_{\sigma}^2 h^{\infty}_t(\sigma_{\ast})&=&(\partial_{\sigma} B^1_{t,\ast}\partial_{\sigma} B^2_{t,\ast}+\partial_{\sigma} B^2_{t,\ast}\partial_{\sigma} B^1_{t,\ast})\mathcal{A}_{\ast}-((\partial_{\sigma} B^1_{t,\ast})^2+(\partial_{\sigma} B^2_{t,\ast})^2)\mathcal{A}_{\ast} -(\partial_{\sigma} B^1_{t,\ast}+\partial_{\sigma} B^2_{t,\ast})\mathcal{A}_{\ast}\frac{\partial_{\sigma} \rho_{t,\ast}}{\rho_{t,\ast}} \nonumber \\
&&-2a_0(\partial_{\sigma} B^1_{t,\ast})^2 -2c_0(\partial_{\sigma} B^2_{t,\ast})^2 -\partial_{\rho}\mathcal{A}(\rho_{t,\ast}) \frac{(\partial_{\sigma}\rho_{t,\ast})^2}{\rho_{t,\ast}}+\mathcal{A}_{\ast}\frac{\partial_{\sigma} \rho_{t,\ast}}{\rho_{t,\ast}}\left(\frac{\partial_{\sigma} \rho_{t,\ast}}{\rho_{t,\ast}}-\partial_{\sigma} B^1_{t,\ast}-\partial_{\sigma} B^2_{t,\ast}\right) \nonumber \\
&=&\mathcal{A}_{\ast}\left(\frac{\partial_{\sigma}\rho_{t,\ast}}{\rho_{t,\ast}}-\partial_{\sigma} B^1_{t,\ast}-\partial_{\sigma} B^2_{t,\ast}\right)^2-\partial_{\rho}\mathcal{A}(\rho_{t,\ast})\frac{(\partial_{\sigma} \rho_{t,\ast})^2}{\rho_{t,\ast}} -2(a_0+\mathcal{A}_{\ast})(\partial_{\sigma} B^1_{t,\ast})^2-2(c_0+\mathcal{A}_{\ast})(\partial_{\sigma} B^2_{t,\ast})^2, \nonumber
\end{eqnarray}
where $\mathcal{A}_{\ast}=\mathcal{A}(\rho_{t,\ast})$.

On the other hand, for $u\in\mathbb{R}^{n_1}$, let $s_n(t)=(1-\bar{\rho}_t)/2$, $\Upsilon_1=b_n^{-1/2}(\partial_{\sigma}H(\sigma_{\ast})-\partial_{\sigma}\hat{H}_n(\sigma_{\ast};s_n))u$,
$\Upsilon_2=b_n^{-1/2}(\partial_{\sigma}\hat{H}_n(\sigma_{\ast};s_n)+\sum_{i=1}^3(-1)^i\partial_{\sigma}\tilde{H}^i_{n,s_n}(T;\sigma_{\ast}))u$, 
$\Upsilon_3=b_n^{-1}u^{\star}\partial^2_{\sigma}H_n(\sigma_{\ast})u/2+u^{\star}\Gamma u/2$, 
$\Upsilon_4=b_n^{-1/2}\partial_{\sigma}\tilde{H}^3_{n,s_n}(T;\sigma_{\ast})$
and $\tilde{\mathcal{X}}_t=\tilde{\mathcal{X}}_{t,n}(u)=b_n^{-1/2}(\partial_{\sigma}\tilde{H}^1_{n,s_n}(t;\sigma_{\ast})-\partial_{\sigma}\tilde{H}^2_{n,s_n}(t;\sigma_{\ast}))u$.
Then 
\begin{eqnarray}
\mathcal{V}_n(u) =\tilde{\mathcal{X}}_{T,n}(u)-\frac{1}{2}u^{\star}\Gamma u+\sum_{j=1}^4\Upsilon_j. \nonumber
\end{eqnarray}

As $n\to \infty$, since $P[\tau(s_n)<T]\to 0$, we have $\Upsilon_1\to^p 0$.
By $[A1]-[A4]$ and Lemmas \ref{prob-conv} and \ref{H1-diff} with $q=2$, we have $\Upsilon_2\to^p 0$.
Furthermore, we obtain $\Upsilon_3\to^p 0$ by Proposition \ref{Hn-lim}.
Moreover, $\Upsilon_4\to^p 0$ since $P[\tau(s_n)<T]\to 0$ 
and $\partial_{\sigma}\tilde{H}^3_{n,s_n}(T;\sigma_{\ast})\equiv 0$ on $\{\tau(s_n)=T\}$.

Then it is sufficient to show 
\begin{equation*}
\sum_{i=1}^kv_i(\tilde{\mathcal{X}}_{T,n}(u_i)-\frac{1}{2}u_i^{\star}\Gamma u_i)\to^{s\mathchar`-\mathcal{L}}\sum_{i=1}^kv_i\mathcal{V}(u_i)
\end{equation*}
as $n\to \infty$ for any $v_1,\cdots, v_k\in\mathbb{R}$ and $u_1,\cdots, u_k\in\mathbb{R}^{n_1}$.

Let $\mathcal{F}^{\dagger}_t=\cap_{t'>t}\{\mathcal{F}_{t'}\bigvee \sigma(\{\Pi_n\}_n)\}$ for $t\in [0,T)$ and $\mathcal{F}^{\dagger}_T=\mathcal{F}_T\bigvee \sigma(\{\Pi_n\}_n)$.
Then $\{W_t,\mathcal{F}^{\dagger}_t\}_{0\leq t\leq T}$ is also a Wiener process and $\{\tilde{\mathcal{X}}_t(u),\mathcal{F}^{\dagger}_t\}_t$ is a martingale for $u\in\mathbb{R}^{n_1}$.
By Theorem 2-1 of Jacod \cite{jacod}, it is sufficient to show that
\begin{equation*}
\langle\tilde{\mathcal{X}}_{\cdot,n}(u)\rangle_t\to^p u^{\star}\Gamma_t u, \ \ \langle\tilde{\mathcal{X}}_{\cdot,n}(u),W\rangle_t\to^p 0, 
\ \ \langle\tilde{\mathcal{X}}_{\cdot,n}(u),N'\rangle_t\to^p 0
\end{equation*}
as $n\to \infty$ for any $t\in[0,T]$, $u\in\mathbb{R}^{n_1}$ and $N'\in\mathcal{M}_b(W^{\perp})$,
where $\Gamma_t=-\int^t_0\partial_{\sigma}^2h^{\infty}_s(\sigma_{\ast})ds$ and $\mathcal{M}_b(W^{\perp})$ is the class of all bounded $\mathcal{F}^{\dagger}_t$-martingales 
which are orthogonal to $W$.

By Ito's formula and symmetry of $\tilde{M}$, we obtain
\begin{eqnarray}
\tilde{\mathcal{X}}_t=-b_n^{-\frac{1}{2}}\sum_{k_1,k_2}\partial_{\sigma}\bigg\{\tilde{M}_{k_1,k_2}\int^t_0\hat{Z}_{k_1,s} d\hat{Z}_{k_2,s}\bigg\}\bigg|_{\sigma=\sigma_{\ast}}u. \nonumber
\end{eqnarray}
Hence it is obvious that
$\langle\tilde{\mathcal{X}},N'\rangle_t= 0$ for all $N'\in\mathcal{M}_b(W^{\perp})$.

Moreover, 
\begin{eqnarray}
\langle\tilde{\mathcal{X}},W^i\rangle_t &=&-b_n^{-\frac{1}{2}}\sum_{k_1,k_2}\sum_{v_1+v_2+v_3=1}\partial_{\sigma}^{v_1}\tilde{M}_{k_1,k_2}\int^t_0\partial_{\sigma}^{v_2}\hat{Z}_{k_1,s}d\langle\partial_{\sigma}^{v_3}\hat{Z}_{k_2},W^i\rangle_su \nonumber \\
&=&-b_n^{-\frac{1}{2}}\sum_k\int^t_0\frac{\int_{(\theta_{0,k})_s}b^{\mathcal{J}(k)}_{v,\ast}dW_v}{\sqrt{|\theta_{0,k}|}} \mathcal{B}_{k,s}^ids+o_p(1) \nonumber
\end{eqnarray}
for $i=1,2$, where $\mathcal{J}(k)=1 \ (1\leq k\leq l_n)$, $\mathcal{J}(k)=2 \ (l_n< k\leq l_n+m_n)$ and
\begin{equation*}
\mathcal{B}^i_{k,s}=\sum_{k_2}\sum_{v_1+v_2+v_3=1}\partial_{\sigma}^{v_1}\tilde{M}_{k,k_2}\partial_{\sigma}^{v_2}\left(|b^{\mathcal{J}(k)}_{\theta_{0,k},\tau(s_n)}|^{-1}\right)b^{\mathcal{J}(k_2),i}_{L(\theta_{0,k}),\ast}\partial_{\sigma}^{v_3}\left(|b^{\mathcal{J}(k_2)}_{L(\theta_{0,k})\wedge \tau(s_n)}|^{-1}\right)\frac{1_{\theta_{0,k_2}}(s)}{\sqrt{|\theta_{0,k_2}|}}u.
\end{equation*}
On the other hand, we have
\begin{eqnarray}
&&E\bigg[\bigg|-b_n^{-\frac{1}{2}}\sum_k\int^t_0\frac{\int_{(\theta_{0,k})_s}b^{\mathcal{J}(k)}_{v,\ast}dW_v}{\sqrt{|\theta_{0,k}|}} \mathcal{B}_{k,s}^ids\bigg|^2\bigg] \nonumber \\
&=&b_n^{-1}E\bigg[\sum_{k,k'}\int^t_0\int^t_0\frac{\int_{(\theta_{0,k})_{s_1}\cap (\theta_{0,k'})_{s_2}}b^{\mathcal{J}(k)}_{v,\ast}\cdot b^{\mathcal{J}(k')}_{v,\ast}dv}{\sqrt{|\theta_{0,k}|}\sqrt{|\theta_{0,k'}|}}\mathcal{B}^i_{k,s_1}\mathcal{B}^i_{k',s_2}ds_1ds_2\bigg] \nonumber \\
&\leq & b_n^{-1}E\bigg[R^C\sum_{k,k'}(\tilde{M}_0)_{k,k'}\int^t_0\int^t_0|\mathcal{B}^i_{k,s_1}||\mathcal{B}^i_{k',s_2}|ds_1ds_2\bigg] 
\leq  b_n^{-1}E\bigg[R^C\sum_{k}\bigg(\int^t_0|\mathcal{B}^i_{k,s}|ds\bigg)^2\bigg]\to 0 \nonumber
\end{eqnarray}
as $n\to \infty$ since $|\partial_{\sigma}^v(\tilde{M})_{k,k'}|\leq CR^{2v}(1-\bar{\rho}_T)^{-(v+5/2)}M'_{k,k'}$, where $M'_{k,k'}=\sum_{p=0}^{\infty}(\tilde{M}_p)_{k,k'}/(p+1)^2$. 
Hence we have $\langle\tilde{\mathcal{X}},W\rangle_t\to^p 0$ as $n\to \infty$ for any $t\in [0,T]$.

Then it is sufficient to show $\langle\tilde{\mathcal{X}}(u)\rangle_t\to^p u^{\star}\Gamma_tu$ as $n\to \infty$ for any $t\in[0,T]$ and $u\in\mathbb{R}^{n_1}$.
\begin{eqnarray}\label{X-est1}
&&\langle\tilde{\mathcal{X}}\rangle_t \nonumber \\
&=&b_n^{-1}u^{\star}\sum_{k_1,k_2,k_3,k_4}\int^t_0\partial_{\sigma}\bigg(\tilde{M}_{k_1,k_2}\hat{Z}_{k_1,s}\frac{b^{\mathcal{J}(k_2)}_{s,\ast}}{|b^{\mathcal{J}(k_2)}_{\theta_{0,k_2},\tau(s_n)}|}\bigg) \partial_{\sigma}\bigg(\tilde{M}_{k_3,k_4}\hat{Z}_{k_3,s}\frac{b^{\mathcal{J}(k_4)}_{s,\ast}}{|b^{\mathcal{J}(k_4)}_{\theta_{0,k_4},\tau(s_n)}|}\bigg)\frac{1_{\theta_{0,k_2}\cap\theta_{0,k_4}}(s)}{\sqrt{|\theta_{0,k_2}|}\sqrt{|\theta_{0,k_4}|}}ds\bigg|_{\sigma=\sigma_{\ast}}u \nonumber \\
&=&b_n^{-1}\sum_{k_1,k_2,k_3,k_4}\int^t_0\tilde{\mathcal{B}}_{k_1,k_2}\cdot \tilde{\mathcal{B}}_{k_3,k_4}\frac{Z'_{k_1,s}Z'_{k_3,s}}{\sqrt{|\theta_{0,k_1}|}\sqrt{|\theta_{0,k_3}|}} \frac{1_{\theta_{0,k_2}\cap\theta_{0,k_4}}(s)}{\sqrt{|\theta_{0,k_2}|}\sqrt{|\theta_{0,k_4}|}}ds +o_p(1), 
\end{eqnarray}
where $\tilde{\mathcal{B}}_{k_1,k_2}=\partial_{\sigma}\big(\tilde{M}_{k_1,k_2}|b^{\mathcal{J}(k_2)}_{L(\theta_{0,k_1})\wedge\tau(s_n)}|^{-1}|b^{\mathcal{J}(k_1)}_{\theta_{0,k_1},\tau(s_n)}|^{-1}\big)|_{\sigma=\sigma_{\ast}}b^{\mathcal{J}(k_2)}_{L(\theta_{0,k_1}),\ast}u$ and $Z'_{k,s}=\int_{(\theta_{0,k})_s}b^{\mathcal{J}(k)}_{v,\ast}dW_v$.

Ito's formula yields
\begin{equation}\label{X-est2}
Z'_{k_1,s}Z'_{k_3,s}=\int^s_0Z'_{k_1,v}dZ'_{k_3,v}+\int^s_0Z'_{k_3,v}dZ'_{k_1,v}+\langle Z'_{k_1},Z'_{k_3}\rangle_s.
\end{equation}
Moreover, let 
\begin{equation*}
F_k(v,s)=\sum_{k_1,k_2,k_4}\tilde{\mathcal{B}}_{k_1,k_2}\cdot \tilde{\mathcal{B}}_{k,k_4}\frac{Z'_{k_1,v}}{\sqrt{|\theta_{0,k_1}|}\sqrt{|\theta_{0,k}|}} \frac{1_{\theta_{0,k_2}\cap\theta_{0,k_4}}(s)}{\sqrt{|\theta_{0,k_2}|}\sqrt{|\theta_{0,k_4}|}},
\end{equation*}
then we have
\begin{equation*}
\sup_v|F_k(v,s)|\leq CR^{10}(1-\bar{\rho}_T)^{-7}\sum_{k_1,k_2}\sum_{k_4}\frac{M'_{k_1,k_2}M'_{k,k_4}1_{\theta_{0,k_2}\cap\theta_{0,k_4}}(s)}{\sqrt{|\theta_{0,k_1}|}\sqrt{|\theta_{0,k_2}|}\sqrt{|\theta_{0,k}|}\sqrt{|\theta_{0,k_4}|}}\sup_v|Z'_{k_1,v}||u|^2
\end{equation*}
and therefore
\begin{equation*}
\int^t_0E[\sup_v|F_k(v,s)|^4|\Pi_n]^{\frac{1}{4}}ds\leq C|u|^2|\theta_{0,k}|^{-1/2}\sum_{k_1}(M'\tilde{M}_0M')_{k_1,k}.
\end{equation*}
Then we obtain 
\begin{eqnarray}
&&E\bigg[\bigg(\sum_k\int^t_0\int^s_0F_k(v,s)dZ'_{k,v}ds\bigg)^2\bigg|\Pi_n\bigg] \nonumber \\
&=&\sum_{k,k'}\int^t_0\int^t_0E\bigg[\int^{s_1\wedge s_2}_0F_k(v,s_1)F_{k'}(v,s_2)d\langle Z'_{k},Z'_{k'}\rangle_v\bigg|\Pi_n\bigg]ds_1ds_2 \nonumber \\
&\leq & E[R^4]^{\frac{1}{2}}\sum_{k,k'}|\theta_{0,k}\cap \theta_{0,k'}|\int^t_0E[\sup_v|F_k(v,s_1)|^4|\Pi_n]^{\frac{1}{4}}ds_1\int^t_0E[\sup_v|F_{k'}(v,s_2)|^4|\Pi_n]^{\frac{1}{4}}ds_2 \nonumber \\
&\leq & C|u|^4\sum_{k,k'}(\tilde{M}_0)_{k,k'}\sum_{k_1,k'_1}(M'\tilde{M}_0M')_{k_1,k}(M'\tilde{M}_0M')_{k'_1,k'}\leq C|u|^4(l_n+m_n)=o_p(b_n^2). \nonumber
\end{eqnarray}
Hence we obtain 
\begin{equation}\label{X-est3}
b_n^{-1}\sum_k\int^t_0\int^s_0F_k(v,s)dZ'_{k,v}ds\to^p 0
\end{equation}
as $n\to \infty$ by Lemma \ref{prob-conv}.

By (\ref{X-est1})-(\ref{X-est3}), we have
\begin{equation*}
\langle\tilde{\mathcal{X}}\rangle_t=b_n^{-1}\sum_{k_1,k_2,k_3,k_4}\int^t_0\tilde{\mathcal{B}}_{k_1,k_2}\cdot \tilde{\mathcal{B}}_{k_3,k_4}\frac{\langle Z'_{k_1},Z'_{k_3}\rangle_s}{\sqrt{|\theta_{0,k_1}|}\sqrt{|\theta_{0,k_3}|}} \frac{1_{\theta_{0,k_2}\cap\theta_{0,k_4}}(s)}{\sqrt{|\theta_{0,k_2}|}\sqrt{|\theta_{0,k_4}|}}ds +o_p(1).
\end{equation*}
Let
\begin{equation*}
\hat{L}_p(\rho_1,\rho_2)=\left(
\begin{array}{ll}
\rho_1 (GG^{\star})^p & -\rho_2 (GG^{\star})^pG \\
-\rho_2 (G^{\star}G)^pG^{\star} & \rho_1 (G^{\star}G)^p \\
\end{array}
\right), \hat{B}(x,y)=\left(
\begin{array}{ll}
x\mathcal{E}_{l_n} & 0 \\
0 & y\mathcal{E}_{m_n} \\
\end{array}
\right),
\end{equation*}
\begin{equation*}
\mathcal{D}'(t)=\sum_{p=0}^{\infty}\partial_{\sigma}\Big\{\hat{B}(B^1_t,B^2_t)\hat{L}_p(\rho^{2p}_t,\rho^{2p+1}_t)\hat{B}(B^1_t,B^2_t)\Big\}\Big|_{\sigma=\sigma_{\ast}}
\end{equation*}
and
\begin{eqnarray}
\hat{\mathcal{D}}(t)&=&\left(
\begin{array}{ll}
\hat{\mathcal{D}}_{11}(t) & \hat{\mathcal{D}}_{12}(t) \\
\hat{\mathcal{D}}_{21}(t) & \hat{\mathcal{D}}_{22}(t) \\
\end{array}
\right) =\mathcal{D}'(t)\hat{L}_0(1,-\rho_{t,\ast}), \nonumber
\end{eqnarray}
where $\mathcal{E}_l$ denotes the unit matrix of size $l$.
Then by $[A2]$ and the estimate $P[\tau(s_n)<T]\to 0$, we obtain 
\begin{eqnarray}
\langle\tilde{\mathcal{X}}\rangle_t&=&b_n^{-1}u^{\star}\sum_{k_1,k_2,k_3,k_4}(\mathcal{D}'(L(\theta_{0,k_1})))_{k_1,k_2}(\mathcal{D}'(L(\theta_{0,k_1})))_{k_3,k_4}\frac{b^{\mathcal{J}(k_1)}_{L(\theta_{0,k_1}),\ast}\cdot b^{\mathcal{J}(k_3)}_{L(\theta_{0,k_1}),\ast}}{|b^{\mathcal{J}(k_1)}_{L(\theta_{0,k_1}),\ast}||b^{\mathcal{J}(k_3)}_{L(\theta_{0,k_1}),\ast}|} \nonumber \\
&&\times \frac{b^{\mathcal{J}(k_2)}_{L(\theta_{0,k_1}),\ast}\cdot b^{\mathcal{J}(k_4)}_{L(\theta_{0,k_1}),\ast}}{|b^{\mathcal{J}(k_2)}_{L(\theta_{0,k_1}),\ast}||b^{\mathcal{J}(k_4)}_{L(\theta_{0,k_1}),\ast}|}\frac{\int^t_0\int^s_01_{\theta_{0,k_1}\cap\theta_{0,k_3}}(v)1_{\theta_{0,k_2}\cap\theta_{0,k_4}}(s)dvds}{\sqrt{|\theta_{0,k_1}|}\sqrt{|\theta_{0,k_2}|}\sqrt{|\theta_{0,k_3}|}\sqrt{|\theta_{0,k_4}|}}u+o_p(1). \nonumber
\end{eqnarray}
Since for intervals $K_1,K_2$, we have 
\begin{equation*}
\int^t_0\int^s_01_{K_1}(v)1_{K_2}(s)dvds+\int^t_0\int^s_01_{K_2}(v)1_{K_1}(s)dvds=|(K_1)_t||(K_2)_t|,
\end{equation*}
then by symmetry of $\mathcal{D}'$, we have 
\begin{eqnarray}
\langle\tilde{\mathcal{X}}\rangle_t&=&\frac{1}{2}b_n^{-1}u^{\star}\sum_{k_1,k_2,k_3,k_4}(\mathcal{D}'(L(\theta_{0,k_1})))_{k_1,k_2}(\mathcal{D}'(L(\theta_{0,k_1})))_{k_3,k_4}\frac{b^{\mathcal{J}(k_1)}_{L(\theta_{0,k_1}),\ast}\cdot b^{\mathcal{J}(k_3)}_{L(\theta_{0,k_1}),\ast}}{|b^{\mathcal{J}(k_1)}_{L(\theta_{0,k_1}),\ast}||b^{\mathcal{J}(k_3)}_{L(\theta_{0,k_1}),\ast}|} \nonumber \\
&&\times \frac{b^{\mathcal{J}(k_2)}_{L(\theta_{0,k_1}),\ast}\cdot b^{\mathcal{J}(k_4)}_{L(\theta_{0,k_1}),\ast}}{|b^{\mathcal{J}(k_2)}_{L(\theta_{0,k_1}),\ast}||b^{\mathcal{J}(k_4)}_{L(\theta_{0,k_1}),\ast}|}\frac{|(\theta_{0,k_1}\cap\theta_{0,k_3})_t||(\theta_{0,k_2}\cap\theta_{0,k_4})_t|}{\sqrt{|\theta_{0,k_1}|}\sqrt{|\theta_{0,k_2}|}\sqrt{|\theta_{0,k_3}|}\sqrt{|\theta_{0,k_4}|}}u+o_p(1) \nonumber \\
&=&\frac{1}{2}b_n^{-1}u^{\star}\sum_{k;L(\theta_{0,k})\in [0,t)}((\hat{\mathcal{D}}(L(\theta_{0,k})))^2)_{k,k}u+o_p(1). \nonumber
\end{eqnarray}
On the other hand, for $p\in\mathbb{Z}_+$, $x,y\in\mathbb{R}$, $\rho_1,\rho_2,\rho_{\ast}\in [-1,1]$, we can write
\begin{eqnarray}
\hat{B}(x,y)\hat{L}_p(\rho_1,\rho_2)\hat{B}(x,y)\hat{L}_0(1,-\rho_{\ast})
=\left(
\begin{array}{ll}
x^2\rho_1(GG^{\star})^p-xy\rho_2\rho_{\ast}(GG^{\star})^{p+1} & (x^2\rho_1\rho_{\ast}-xy\rho_2)(GG^{\star})^pG \\
(y^2\rho_1\rho_{\ast}-xy\rho_2)(G^{\star}G)^pG^{\star} & y^2\rho_1(GG^{\star})^p-xy\rho_2\rho_{\ast}(G^{\star}G)^{p+1} \\
\end{array}
\right). \nonumber
\end{eqnarray}
Then for $\mathcal{Q}^1_t=(\partial_{\sigma} B^1_{t,\ast}-\partial_{\sigma} B^2_{t,\ast})+\partial_{\sigma}\rho_{t,\ast}/\rho_{t,\ast}$,
$\mathcal{Q}^2_t=(\partial_{\sigma} B^2_{t,\ast}-\partial_{\sigma} B^1_{t,\ast})+\partial_{\sigma}\rho_{t,\ast}/\rho_{t,\ast}$,
we have
\begin{eqnarray}
&&\hat{\mathcal{D}}_{11}(t) \nonumber \\
&=&\sum_{p=0}^{\infty}\bigg\{(2\partial_{\sigma} B^1_{t,\ast}\rho_{t,\ast}^{2p}+2p\partial_{\sigma} \rho_{t,\ast}\rho_{t,\ast}^{2p-1})(GG^{\star})^p
-((\partial_{\sigma} B^1_{t,\ast}+\partial_{\sigma} B^2_{t,\ast})\rho_{t,\ast}^{2p+2}+(2p+1)\partial_{\sigma}\rho_{t,\ast}\rho_{t,\ast}^{2p+1})(GG^{\star})^{p+1}\bigg\} \nonumber \\
&=&2\partial_{\sigma} B^1_{t,\ast}\mathcal{E}_{l_n}+\mathcal{Q}^1_t\sum_{p=1}^{\infty}\rho^{2p}_{t,\ast}(GG^{\star})^p. \nonumber
\end{eqnarray}
Similarly, we have
\begin{eqnarray}
\hat{\mathcal{D}}_{22}(t)=2\partial_{\sigma} B^2_{t,\ast}\mathcal{E}_{m_n}+\mathcal{Q}^2_t\sum_{p=1}^{\infty}\rho^{2p}_{t,\ast}(G^{\star}G)^p, \nonumber
\end{eqnarray}
and
\begin{equation*}
\hat{\mathcal{D}}_{12}(t)=-\mathcal{Q}^2_t\sum_{p=0}^{\infty}\rho_{t,\ast}^{2p+1}(GG^{\star})^pG, \quad \hat{\mathcal{D}}_{21}(t)=-\mathcal{Q}^1_t\sum_{p=0}^{\infty}\rho_{t,\ast}^{2p+1}(G^{\star}G)^pG^{\star}.
\end{equation*}
Then by the estimate $a_p\equiv c_p \ (p\geq 1)$, $[A3]$ and Lemma \ref{A3''}, it follows that
\begin{eqnarray}
\langle\tilde{\mathcal{X}}\rangle_t 
&=&\frac{1}{2}b_n^{-1}u^{\star}\bigg\{\sum_{i;L(I^i)\in [0,t)}(\hat{\mathcal{D}}^2_{11} +\hat{\mathcal{D}}_{12}\hat{\mathcal{D}}_{21})_{ii}(L(I^i)) +\sum_{j;L(J^j)\in[0,t)}(\hat{\mathcal{D}}_{22}^2+\hat{\mathcal{D}}_{21}\hat{\mathcal{D}}_{12})_{jj}(L(J^j))\bigg\}u+o_p(1) \nonumber \\
&=&u^{\star}\int^t_0\bigg\{2(\partial_{\sigma} B^1_{s,\ast})^2a_0(s)+2(\partial_{\sigma} B^2_{s,\ast})^2c_0(s)+(\partial_{\sigma} B^1_{s,\ast}\mathcal{Q}^1_s+\mathcal{Q}^1_s\partial_{\sigma} B^1_{s,\ast})\mathcal{A}(\rho_{s,\ast}) \nonumber \\
&&+(\partial_{\sigma} B^2_{s,\ast}\mathcal{Q}^2_s+\mathcal{Q}^2_s\partial_{\sigma} B^2_{s,\ast})\mathcal{A}(\rho_{s,\ast})+\frac{\mathcal{Q}^1_s\mathcal{Q}^2_s+\mathcal{Q}^2_s\mathcal{Q}^1_s}{2}\sum_{p_1=0}^{\infty}\sum_{p_2=0}^{\infty}\rho_{s,\ast}^{2p_1+2p_2+2}a_{p_1+p_2+1}(s) \nonumber \\
&&+\frac{(\mathcal{Q}^1_s)^2+(\mathcal{Q}^2_s)^2}{2}\sum_{p_1=1}^{\infty}\sum_{p_2=1}^{\infty}\rho_{s,\ast}^{2p_1+2p_2}a_{p_1+p_2}(s)\bigg\}dsu+o_p(1). \nonumber 
\end{eqnarray}
Since 
\begin{eqnarray}
\sum_{p_1=0}^{\infty}\sum_{p_2=0}^{\infty}\rho_{s,\ast}^{2p_1+2p_2+2}a_{p_1+p_2+1}(s)=\frac{\partial_{\rho}\mathcal{A}(\rho_{s,\ast})\rho_{s,\ast}}{2}, \quad
\sum_{p_1=1}^{\infty}\sum_{p_2=1}^{\infty}\rho_{s,\ast}^{2p_1+2p_2}a_{p_1+p_2}(s)=\frac{\partial_{\rho}\mathcal{A}(\rho_{s,\ast})\rho_{s,\ast}}{2}-\mathcal{A}(\rho_{s,\ast}), \nonumber
\end{eqnarray}
we have
\begin{eqnarray}
\langle\tilde{\mathcal{X}}\rangle_t &=&u^{\star}\int^t_0\bigg\{2(a_0(s)+\mathcal{A}(\rho_{s,\ast}))(\partial_{\sigma} B^1_{s,\ast})^2+2(c_0(s)+\mathcal{A}(\rho_{s,\ast}))(\partial_{\sigma} B^2_{s,\ast})^2 +\frac{(\mathcal{Q}^1_s+\mathcal{Q}^2_s)^2}{4}\partial_{\rho}\mathcal{A}(\rho_{s,\ast})\rho_{s,\ast}\nonumber \\
&&-\frac{\mathcal{A}(\rho_{s,\ast})}{2}\big\{(\mathcal{Q}^1_s-2\partial_{\sigma}B^1_{s,\ast})^2+(\mathcal{Q}^2_s-2\partial_{\sigma}B^2_{s,\ast})^2\big\} 
\bigg\}dsu+o_p(1) \nonumber \\
&=&u^{\star}\Gamma_t u+o_p(1). \nonumber
\end{eqnarray}
\end{proof}

\noindent
{\bf Proof of Theorem \ref{asym-dist}.}

$1$. Since $\Lambda$ is open, there exists $\epsilon>0$ such that $O(\epsilon,\sigma_{\ast})=\{\sigma;|\sigma-\sigma_{\ast}|<\epsilon\}\subset \Lambda$.
For $\hat{\sigma}_n\in O(\epsilon,\sigma_{\ast})$,we have 
\begin{equation*}
-\partial_{\sigma}H_n(\sigma_{\ast})=\int^1_0\partial_{\sigma}^2H_n(\sigma_{\ast}+u(\hat{\sigma}_n-\sigma_{\ast}))(\hat{\sigma}_n-\sigma_{\ast})du
\end{equation*}
since $\partial_{\sigma}H_n(\hat{\sigma}_n)=0$.
Therefore, for $\tilde{\Gamma}_n=-b_n^{-1}\int^1_0\partial_{\sigma}^2H_n(\sigma_{\ast}+u(\hat{\sigma}_n-\sigma_{\ast}))du$, we obtain
$b_n^{1/2}(\hat{\sigma}_n-\sigma_{\ast})=\tilde{\Gamma}_n^{-1}b_n^{-1/2}\partial_{\sigma}H_n(\sigma_{\ast})$ on $\{\det \tilde{\Gamma}_n\neq 0 \ {\rm and } \ \hat{\sigma}_n\in O(\epsilon,\sigma_{\ast})\}$.
Then since Proposition \ref{Hn-lim} and Theorem \ref{thm-cons} yield $P[\det \tilde{\Gamma}_n=0]\to 0$, 
$P[\hat{\sigma}_n\in O(\epsilon,\sigma_{\ast})^c]\to 0$ and $\tilde{\Gamma}_n^{-1}1_{\{\det \tilde{\Gamma}_n\neq 0\} }\to^p \Gamma^{-1}$,
we have $b_n^{1/2}(\hat{\sigma}_n-\sigma_{\ast})\to^{s\mathchar`-\mathcal{L}}\Gamma^{-1/2}\mathcal{N}$ by Proposition \ref{st-conv}.

$2.$
Let $s_n(t)=(1-\bar{\rho}_t)/2$ for $n\in\mathbb{N}$ and $t\in [0,T]$ and $\{\sigma'_n\}_{n\in\mathbb{N}}$ be random variables where $\sigma'_n$ maximize $\hat{H}_n(\cdot;s_n)$ 
and $\sigma'_n\equiv \hat{\sigma}_n$ on $\{\tau(s_n)= T\}$. 
We first show the statement of Theorem \ref{asym-dist} replacing $\hat{\sigma}_n$ with $\sigma'_n$.

To this end, we extend $\mathcal{Z}_n(\cdot; \sigma_{\ast})$ to a continuous function 
which is defined on $\mathbb{R}^{n_1}$, tend to zero as $|u|\to \infty$, and has the same supremum as $\mathcal{Z}_n(\cdot; \sigma_{\ast})$.
We denote the extension of $\mathcal{Z}_n(\cdot; \sigma_{\ast})$ by the same symbol.

Let $\mathcal{Z}(u,\sigma_{\ast})=\exp(u^{\star}\Gamma^{1/2}\mathcal{N}-u^{\star}\Gamma u/2)$ and $B(R')=\{u;|u|\leq R'\}$ for $R'>0$.
Then it is sufficient to show that
$\limsup_{n\to \infty}E[|b_n^{1/2}(\sigma'_n-\sigma_{\ast})|^p]<\infty$ for any $p>2$ 
and $\mathcal{Z}_n(\cdot, \sigma_{\ast})\to^{s\mathchar`-\mathcal{L}} \mathcal{Z}(\cdot, \sigma_{\ast})$ in $C(B(R'))$ 
as $n\to \infty$ for any $R'>0$, by virtue of Theorem $5$ and Remark $5$ in Yoshida \cite{yos05}. 

By Lemmas \ref{Hn-H2n-diff} and \ref{Sq-est} and Proposition \ref{Hn-lim}, for any $R'>0$, there exists $n_0\in\mathbb{N}$ such that 
\begin{equation*}
\sup_{n\geq n_0}E\left[\sup_{u\in C(B(R'))}|\partial_u\log \mathcal{Z}_n(u;\sigma_{\ast})|\right]<\infty.
\end{equation*}
Then by Propositions \ref{Hn-lim} and \ref{st-conv} and tightness criterion in $C$ space in Billingsley \cite{bil} which can be extended to the one in $C(B(R'))$, it follows that
$\log \mathcal{Z}_n(\cdot, \sigma_{\ast})\to^{s\mathchar`-\mathcal{L}} \log \mathcal{Z}(\cdot, \sigma_{\ast})$ in $C(B(R'))$ 
as $n\to \infty$.

On the other hand, for any $p>2$, let $L>p$, then by Proposition \ref{pld} and Lemma \ref{Sq-est}, we have
\begin{equation*}
P[|b_n^{\frac{1}{2}}(\sigma'_n-\sigma_{\ast})|\geq r] \leq P\left[\sup_{u\in V_n(r,\sigma_{\ast})}\mathcal{Z}_n(u,\sigma_{\ast})\geq 1\right]\leq \frac{C_L}{r^L} \quad (r>0).
\end{equation*}
Therefore we obtain $\sup_nE[|b_n^{1/2}(\sigma'_n-\sigma_{\ast})|^p]<\infty$.
This complete the proof of the statement of Theorem \ref{asym-dist} for $\sigma'_n$.

We will prove the statement for $\hat{\sigma}_n$.
By $[A1],[A2\mathchar`-q,\delta]$ for any $q>2\vee n_1$, and Lemma \ref{Sq-est}, we have $P[\tau(s_n)<T]=O(b_n^{-\xi})$ for any $\xi>0$.
Then it follows that $b_n^{1/2}(\hat{\sigma}_n-\sigma_{\ast})\to^{s\mathchar`-\mathcal{L}} \Gamma^{-1/2}\mathcal{N}$ as $n\to \infty$ by the result for $\sigma'_n$ and the inequality
\begin{equation*}
P[\sigma'_n\neq \hat{\sigma}_n]\leq P[\tau(s_n)<T]=O(b_n^{-\xi})
\end{equation*}
for any $\xi>0$.

Moreover, for any continuous function $f$ of at most polynomial growth, we have 
\begin{equation*}
|E[f(b_n^{1/2}(\hat{\sigma}_n-\sigma_{\ast}))]-E[f(b_n^{1/2}(\sigma'_n-\sigma_{\ast}))]|\leq C(1+b_n^{1/2}R'')^CP[\sigma'_n\neq \hat{\sigma}_n]\to 0
\end{equation*}
as $n\to \infty$, where $R''$ denotes the diameter of the parameter space $\Lambda$.
\qed

\noindent
{\bf Proof of Theorem \ref{bayes-asym-dist}.}
Similarly to the argument in the proof of Theorem \ref{asym-dist}, 
we have $P[H_n\equiv \hat{H}_n(\cdot;s_n)]=1-O(b_n^{-\xi})$ for any $\xi>0$, where $s_n(t)=(1-\bar{\rho}_t)/2$.
Then by virtue of Theorem $10$ in Yoshida \cite{yos05}, it is sufficient to show that there exists $n'_0\in\mathbb{N}$ such that
\begin{equation}\label{bayes-cond}
\sup_{n\geq n'_0}E\bigg[\bigg(\int_{U_n(\sigma_{\ast})}\mathcal{Z}_n(u)\pi(\sigma_{\ast}+b_n^{-1/2}u)du\bigg)^{-1}\bigg]<\infty.
\end{equation}
By Proposition \ref{Hn-lim} and Lemmas \ref{Hn-H2n-diff} and \ref{Sq-est}, for any $\delta >0$, there exists $p\in 2\mathbb{N}, p>n_1\vee 2$, $n_0'\in \mathbb{N}$ and $C_0>0$ such that
\begin{equation*}
\sup_{n\geq n_0'}E[|\hat{H}_n(\sigma_{\ast}+b_n^{-1/2}u)-\hat{H}_n(\sigma_{\ast})|^p]\leq C_0|u|^p
\end{equation*}
for any $u\in U'(\delta)$ where $U'(\delta)=\{u\in \mathbb{R}^{n_1};|u_i|\leq \delta \ (i=1,\ldots, n_1)\}$.
Then we have (\ref{bayes-cond}) by Lemma $2$ in Yoshida \cite{yos05}.
\qed

\subsection{Proof of Propositions \ref{p-lim} - \ref{inf-a1-suff}}

First, we look back Rosenthal-type inequalities in Doukhan and Louhichi \cite{Dou-Lou01} (Theorem $3$ and Lemma $7$).
\begin{theorem}\label{rosenthal}(Rosenthal-type inequalities)
Let $q\geq 2$ and $q\in\mathbb{N}$. Let $\{X'_n\}_{n\in\mathbb{N}}$ be a centered process, $\alpha_0=1/4$ and 
\begin{equation*}
\alpha_k=\sup_{i,j\in\mathbb{N},j-i\geq k}\sup_{A\in\sigma(X'_l;l\leq i)}\sup_{B\in\sigma(X'_m;m\geq j)} |P(A\cap B)-P(A)P(B)|
\end{equation*} 
for $k\in\mathbb{N}$. Suppose $\alpha_k\to 0 \ (k\to \infty)$. Then
\begin{eqnarray}
\bigg|E\bigg[\bigg(\sum_{j=1}^nX'_j\bigg)^q\bigg]\bigg|\leq \frac{2^q(2q-2)!}{(q-1)!}\Bigg\{\left(\sum_{i=1}^n\int^1_0(\alpha^{-1}(u)\wedge n)^{q-1}Q_{X'_i}^q(u)du\right) \vee \left(\sum_{i=1}^n\int^1_0(\alpha^{-1}(u)\wedge n)Q_{X'_i}^2(u)du\right)^{\frac{q}{2}}\Bigg\}, \nonumber
\end{eqnarray}
where $\alpha^{-1}(u)=\sum_{k=0}^{\infty}1_{\{\alpha_k> u\} }$ and $Q_{X'}(s)=\inf\{t>0, P[|X'|>t]\leq s\}$.
\end{theorem}

\noindent
{\bf Proof of Proposition \ref{p-lim}.}

In this proof, we set general constants denoted by $C$ do not depend on $n,p,f$.

By Lemma \ref{GGT-lambda}, we obtain
\begin{equation*}
((GG^{\star})^p)_{II}\vee ((G^{\star}G)^p)_{JJ}\leq \parallel (GG^{\star})^p\parallel \vee \parallel (G^{\star}G)^p\parallel \leq 1
\end{equation*}
for $p\in\mathbb{Z}_+$.
Hence
$b_n\nu^{p,i}_n([t_{k-1},t_k)) \leq N^i_{t_k}-N^i_{t_{k-1}}+1$
for $1\leq k\leq [b_n]$, $p\in\mathbb{Z}_+$ and $i=1,2$. 
Therefore we obtain
\begin{equation}\label{C1-est}
\sup_{1\leq k\leq [b_n]}E[\max_{p\in\mathbb{Z}_+, i=1,2}|b_n\nu^{p,i}_n([t_{k-1},t_k))|^{q(1+\delta)}]\leq \sup_kE[\max_{i=1,2}(N^i_{t_k}-N^i_{t_{k-1}}+1)^{q(1+\delta)}]\leq C
\end{equation}
for sufficiently large $n$ by $[B1\mathchar`-(q(1+\delta))]$.

For $h>0$ and $k \in \mathbb{N}$, let 
\begin{eqnarray}
A^{p,+}_{h,t}&=&\cap_{i=1,2}\cap_{l\in [1, p\wedge h^{-1}(T-t)]\cap \mathbb{N}}\{\omega ; N^i_{t+lh}-N^i_{t+(l-1)h}>0\}, \nonumber \\
A^{p,-}_{h,t}&=&\cap_{i=1,2}\cap_{l\in [1, p\wedge h^{-1}t]\cap\mathbb{N}}\{\omega ; N^i_{t-(l-1)h}-N^i_{t-lh}>0\}, \nonumber \\
A^p_{k,h}&:=&A^{2p+1,+}_{[b_n]^{-1}hT,t_k}\cap A^{2p+1,-}_{[b_n]^{-1}hT,t_{k-1}}, \nonumber
\end{eqnarray}
where $\cap_{\emptyset}=\Omega$. 

Fix $p\in\mathbb{Z}_+$, $i=1,2$ and a $\beta-$H${\rm \ddot{o}}$lder continous function $f$ on $[0,T]$. 
Then we have
\begin{equation*}
\nu^{p,i}_n([t_{k-1},t_k))1_{A^p_{k,h}}\in \mathcal{G}^n_{(k-2-[(2p+1)h])\vee 0,(k+[(2p+1)h]+1)\wedge [b_n]}.
\end{equation*}
Let $\alpha^{-1}(u)=\sum_{k=0}^{\infty}1_{\{\alpha^n_k>u\} },f^n_k=f_{t^n_{k-1}}$, $\delta'=(1+\delta)/(2(1+\delta-\epsilon\delta))$ and
\begin{equation*}
X'_k=b_nf^n_k\big\{\nu^{p,i}_n([t_{k-1},t_k))1_{A^p_{k,b_n^{\delta'}}}-E[\nu^{p,i}_n([t_{k-1},t_k))1_{A^p_{k,b_n^{\delta'}}}]\big\},
\end{equation*}
then by Rosenthal-type inequalities, we obtain
\begin{eqnarray}
E\bigg[\bigg|b_n^{-1}\sum_{k=1}^{[b_n]}X'_k\bigg|^q\bigg] 
&\leq & b_n^{-q}\frac{2^{q/2}(2q-2)!}{(q-1)!}\bigg\{\bigg(\sum_{k=1}^{[b_n]}\int^{\frac{1}{4}}_0(\alpha^{-1}(u)+2[(2p+1)b_n^{\delta'}]+3)^{q-1}Q^q_{X'_k}(u)du\bigg) \nonumber \\
&& \vee \bigg(\sum_{k=1}^{[b_n]}\int^{\frac{1}{4}}_0(\alpha^{-1}(u)+2[(2p+1)b_n^{\delta'}]+3)Q^2_{X'_k}(u)du\bigg)^{\frac{q}{2}}\bigg\} \nonumber \\
&\leq & Cb_n^{-q}[b_n]^{q/2}\sup_k\int^{\frac{1}{4}}_0(\alpha^{-1}(u)+2[(2p+1)b_n^{\delta'}]+3)^{q-1}Q^q_{X'_k}(u)du \nonumber \\
&\leq & C(p+1)^{q-1}b_n^{q\delta'-\frac{q}{2}}\bigg(\int^{1}_0(\alpha^{-1}(u))^{\frac{(1+\delta)(q-1)}{\delta}}du\bigg)^{\frac{\delta}{1+\delta}} \bigg(\sup_k\int^1_0Q^{q(1+\delta)}_{X'_k}(u)du\bigg)^{\frac{1}{1+\delta}}. \nonumber 
\end{eqnarray}
For sufficiently large $n$, since (\ref{alpha-coeff-est}) and (\ref{C1-est}) hold, $\int^1_0Q^{q(1+\delta)}_{X'_k}(u)du=E[|X'_k|^{q(1+\delta)}]$, $(x+1)^{q'}-x^{q'}\leq q'(x+1)^{q'-1} \ (x\geq 0,q'\geq 1)$ and $\alpha^{-1}(u)=k'$ if $\alpha^n_{k'}\leq u < \alpha^n_{k'-1}$, we have
\begin{equation*}
\int^1_0(\alpha^{-1}(u))^{q'}du=\sum_{k=1}^{\infty}k^{q'}(\alpha^n_{k-1}-\alpha^n_k)\leq q'\sum_{k=0}^{\infty}(k+1)^{q'-1}\alpha^n_k
\end{equation*}
for $q'\geq 1$ and 
\begin{eqnarray}
E\bigg[\bigg|b_n^{-1}\sum_{k=1}^{[b_n]}X'_k\bigg|^q\bigg]\leq C(p+1)^{q-1}b_n^{q\delta'-\frac{q}{2}} \sup_t|f_t|^q.  \nonumber
\end{eqnarray}

On the other hand, 
\begin{eqnarray}
E\bigg[\bigg|\sum_{k=1}^{[b_n]}f^n_k\nu^{p,i}_n([t_{k-1},t_k))1_{(A^p_{k,b_n^{\delta'}})^c}\bigg|^q\bigg]
\leq [b_n]^{-1}\sum_{k=1}^{[b_n]}\sup_kE[|N^i_{t_k}-N^i_{t_{k-1}}+1|^{q(1+\delta)}]^{\frac{1}{1+\delta}}\sup_t|f_t|^qP[(A^p_{k,b_n^{\delta'}})^c]^{\frac{\delta}{1+\delta}}. \nonumber
\end{eqnarray}
Moreover, by $[B2\mathchar`-(q\epsilon)]$, we obtain
\begin{eqnarray}
P[(A^p_{k,b_n^{\delta'}})^c]\leq  4(2p+1)\sup_{i=1,2}\sup_tP[N^i_{t+[b_n]^{-1}b_n^{\delta'}T}-N^i_t=0]
\leq C(p+1)b_n^{-q\epsilon\delta'}. \nonumber
\end{eqnarray}
Hence we have
\begin{eqnarray}
&&E\bigg[\bigg|\sum_{k=1}^{[b_n]}f^n_k\nu^{p,i}_n([t_{k-1},t_k))(1_{A^p_{k,b_n^{\delta'}}}-1)\bigg|^q\bigg]\leq C(p+1)b_n^{-\frac{q\epsilon\delta\delta'}{1+\delta}}\sup_t|f_t|^q. \nonumber
\end{eqnarray}
Therefore we obtain
\begin{equation}\label{Zk-est}
E\bigg[\bigg|\sum_{k=1}^{[b_n]}f^n_k(\nu^{p,i}_n([t_{k-1},t_k))-\zeta^{p,i}_n([t_{k-1},t_k)))\bigg|^q\bigg]\leq C(p+1)^{q-1}b_n^{-q\eta}\sup_t|f_t|^q.
\end{equation}
Furthermore, H${\rm \ddot{o}}$lder continuity of $f$ and (\ref{C1-est}) yield 
\begin{eqnarray}\label{Holder-est1}
E\bigg[\bigg|\sum_{k=1}^{[b_n]}\int^{t^n_k}_{t^n_{k-1}}(f_t-f^n_k)d\nu^{p,i}_n\bigg|^q\bigg]\leq  [b_n]^{q-1}\omega_{\beta}(f)^q\sum_{k=1}^{[b_n]}(T[b_n]^{-1})^{q\beta}E[\nu^{p,i}_n([t_{k-1},t_k))^q] 
\leq Cb_n^{-q\beta}\omega_{\beta}(f)^q.
\end{eqnarray}
By (\ref{Zk-est}) and (\ref{Holder-est1}), we have
\begin{equation}\label{f-est}
E\left[\bigg|\int^T_0f_td\nu^{p,i}_n-\int^T_0f_td\zeta^{p,i}_n\bigg|^q\right]\leq C(p+1)^{q-1}b_n^{-q\eta}\big\{\sup_t|f_t|^q+\omega_{\beta}(f)^q\big\}.
\end{equation}
Since $p\in\mathbb{Z}_+$ and $i=1,2$ are arbitrary, we obtain $[A3'\mathchar`-q,\eta]$ by (\ref{a2-lim}) and (\ref{f-est}).
\qed
\\
\\
{\bf Proof of Proposition \ref{A4q-suff}.}

1. For $h>0$ and $1\leq i\leq [b_n]$, let
\begin{eqnarray} 
\hat{A}^{p,+}_{h,t}&=&\cap_{r=1}^{n_2+2}\cap_{l\in [1,p\wedge h^{-1}(T-t)]\cap \mathbb{N}}\{\omega ; N^r_{t+lh}-N^r_{t+(l-1)h}>0\}, \nonumber \\
\hat{A}^{p,-}_{h,t}&=&\cap_{r=1}^{n_2+2}\cap_{l\in [1,p\wedge h^{-1}t]\cap \mathbb{N}}\{\omega ; N^r_{t-(l-1)h}-N^r_{t-lh}>0\}, \nonumber \\
\hat{A}^p_{i,h}&:=&\hat{A}^{2p+1,+}_{[b_n]^{-1}hT,t_i}\cap \hat{A}^{2p+1,-}_{[b_n]^{-1}hT,t_{i-1}}. \nonumber
\end{eqnarray}
Then $\omega\in \hat{A}^p_{i,j}$ and $t_{i-1}<R(\theta_{0,k})\leq t_i$ 
imply $|\theta_{p,k}|\leq j(4p+2)[b_n]^{-1}T$
for $\omega\in\Omega$, $1\leq k\leq l_n+m_n$, 
$n\in\mathbb{N}$, $0\leq i\leq [b_n]$ and $j\in \mathbb{N}$.   
Moreover, $\hat{A}^p_{i,j}=\Omega$ if $j$ is sufficiently large for each $i$ and $p$.
Therfore, for $\Delta N_i=N^1_{t_i}-N^1_{t_{i-1}}+N^2_{t_i}-N^2_{t_{i-1}}$ 
and $\dot{A}^p_{i,j}=\hat{A}^p_{i,j}\setminus \cup_{j'=0}^{j-1}\hat{A}^p_{i,j'}$, 
we obtain
\begin{eqnarray}\label{A4qH-est}
E[(\Phi_{p,1})^q] 
&=& E\bigg[\bigg(\sum_{i=1}^{[b_n]}\sum_{k; R(\theta_{0,k})\in (t_{i-1},t_i]}|\theta_{p,k}|\sum_{j=1}^{\infty}1_{\dot{A}^p_{i,j}}\bigg)^q\bigg]
\leq E\bigg[\bigg(\sum_{i=1}^{[b_n]}\sum_{j=1}^{\infty}j\cdot (4p+2)[b_n]^{-1}T\Delta N_i1_{\dot{A}^p_{i,j}}\bigg)^q\bigg] \nonumber \\
&\leq & [b_n]^{q-1}\sum_{i=1}^{[b_n]}\sum_{j=1}^{\infty}j^q\cdot (4p+2)^qT^q[b_n]^{-q}E[(\Delta N_i)^q1_{\dot{A}^p_{i,j}}] \nonumber 
\end{eqnarray}
for $p\in\mathbb{Z}_+$, since $\{\dot{A}^p_{i,j}\}_{j\in\mathbb{N}}$ are disjoint.
Then by $[B1\mathchar`-(p'_1q)],[B2\mathchar`-(p'_2(q+2)]$, the H${\rm \ddot{o}}$lder inequality and a similar estimate for $P[(A^p_{k,h})^c]$ in the proof of Proposition \ref{p-lim}, we have
\begin{eqnarray}
E[(\Phi_{p,1})^q]\leq C[b_n]^{-1}\sum_{i=1}^{[b_n]}\sum_{j=1}^{\infty}j^q (4p+2)^qP[(\hat{A}^p_{i,j-1})^c]^{1/p'_2}
\leq C(p+1)^q\sum_{j=1}^{\infty}j^q\{C(p+1)j^{-p'_2(q+2)}\}^{\frac{1}{p'_2}}\leq C(p+1)^{q+1} \nonumber 
\end{eqnarray}
for sufficiently large $n$.

In particular, by the H${\rm \ddot{o}}$lder inequality and Jensen's inequality, we have
\begin{equation*}
E\bigg[r_n^{q'}\sum_{p=0}^{\infty}\frac{(\Phi_{2p+2})^{q'}}{(p+1)^{q'+3}}\bigg]\leq E[r_n^{\frac{qq'}{q-q'}}]^{\frac{q-q'}{q}}E\bigg[\sum_{p=0}^{\infty}\frac{1}{(p+1)^2}\bigg(\frac{(\Phi_{2p+2})^q}{(p+1)^{q+\frac{q}{q'}}}\bigg)\bigg].
\end{equation*}
Therefore $[A4\mathchar`-q',(1+3/q')]$ holds since $r_n\to^p 0$ by the next Proposition \ref{A2q-suff}.
\\
2. The proof is similar to that of 1. For sufficiently large $n$, we have
\begin{eqnarray}
E[(\bar{\Phi}_{p_1,p_2})^{q/2}]
&\leq & [b_n]^{\frac{q}{2}-1}E\bigg[\sum_{i=1}^{[b_n]}\sum_{j=1}^{\infty}\bigg(\sum_{R(\theta_{0,k_1})\in (t_{i-1},t_i]}\sum_{k_2}|\theta_{p_1,k_1}|\wedge |\theta_{p_2,k_2}|1_{\theta_{p_1+p_2,k_1}\cap \theta_{0,k_2}\neq \emptyset}\bigg)^{\frac{q}{2}} 1_{\dot{A}^{p_1+2p_2+1}_{i,j}}\bigg] \nonumber \\
&\leq & [b_n]^{\frac{q}{2}-1}E\bigg[\sum_{i=1}^{[b_n]}\sum_{j=1}^{\infty}\{(4p_1+2)j\wedge (4p_2+2)j\}^{\frac{q}{2}}[b_n]^{-\frac{q}{2}}T^\frac{q}{2}(\Delta N_i)^{\frac{q}{2}} \nonumber \\
&&\times \bigg(\sum_{v=1}^2(N^v_{(t_i+(2p_1+2p_2+1)j[b_n]^{-1}T)\wedge T}-N^v_{(t_{i-1}-(2p_1+2p_2+2)j[b_n]^{-1}T)\vee 0})\bigg)^{\frac{q}{2}}1_{\dot{A}^{p_1+2p_2+1}_{i,j}}\bigg] \nonumber \\
&\leq & C[b_n]^{-1}\sum_{i=1}^{[b_n]}\sum_{j=1}^{\infty}\{(4p_1+2)j\wedge (4p_2+2)j\}^{\frac{q}{2}}\{(4p_1+4p_2+3)j+1\}^{\frac{q}{2}} P[(\hat{A}^{p_1+2p_2+1}_{i,j-1})^c]^{\frac{1}{p'_2}}. \nonumber
\end{eqnarray}
Since $(a\wedge b)(a+b)\leq 2ab \ (a,b\geq 1)$, we obtain
\begin{eqnarray}
E[(\bar{\Phi}_{p_1,p_2})^{q/2}]\leq C\sum_{j=1}^{\infty}(p_1+1)^{\frac{q}{2}}(p_2+1)^{\frac{q}{2}}j^{q}\{C(p_1+p_2+1)j^{-p'_2(q+2)}\}^{\frac{1}{p'_2}}
\leq C(p_1+1)^{q/2+1}(p_2+1)^{q/2+1}. \nonumber
\end{eqnarray}
\qed
\\
\\
\noindent
{\bf Proof of Proposition \ref{A2q-suff}.}

Let $A'_j=\hat{A}^{[b_nj^{-1}T],+}_{jb_n^{-1},0}$ for $j\in\mathbb{N}$. 
Then since $r_n\leq 2jb_n^{-1}$ on $A'_j$, for sufficiently large $n$, we have
\begin{eqnarray}
E[r_n^q]=E\bigg[r_n^q\sum_{j=1}^{\infty}1_{A'_j\setminus \cup_{j'=0}^{j-1}A'_{j'}}\bigg] \leq \sum_{j=1}^{\infty}(2jb_n^{-1})^qP[(A'_{j-1})^c] 
\leq Cb_n^{-q}\sum_{j=1}^{\infty}j^q \cdot [b_nj^{-1}T] \cdot j^{-q-1}\leq Cb_n^{1-q}, \nonumber
\end{eqnarray}
where $A'_0=\emptyset$.

\qed
\\
\\
{\bf Proof of Proposition \ref{inf-a1-suff}.}

By $[B2\mathchar`-q]$, there exists $N\in\mathbb{N}$ such that
\begin{equation}\label{non-sparse2}
\sup_{n\geq n_0}\max_{i=1,2}\sup_{0\leq t\leq T-N[b_n]^{-1}T}P[N^i_{t+N[b_n]^{-1}T}-N^i_t=0]\leq \frac{1}{12}.
\end{equation}
For $M=[b_n/3N]$, $h=[b_n]^{-1}T$ and $s_k=3kNh$, we have
\begin{eqnarray}
a_1=\frac{1}{T}\int^T_0a_1dt=\frac{1}{T}\underset{n\to \infty}{\rm P\mathchar`-\lim} \ b_n^{-1}\sum_{I,J}\frac{|I\cap J|^2}{|I||J|} 
=\frac{1}{T}\underset{n\to \infty}{\rm P\mathchar`-\lim} \ b_n^{-1}\sum_{k=1}^M\sum_{I,J;L(I)\in [s_{k-1},s_k)}\frac{|I\cap J|^2}{|I||J|}. \nonumber
\end{eqnarray} 

Let
\begin{equation*}
\bar{A}^0_k=\emptyset, \quad \bar{A}^j_k=A^{2,+}_{jh,s_k}\cap A^{1,-}_{jh,s_{k-1}}, \quad \ddot{A}^j_k=\bar{A}^j_k\setminus \cup_{j'=0}^{j-1}\bar{A}^{j'}_k, 
\end{equation*}
and
\begin{equation*}
E_k=\cap_{l=1}^3\{N^1_{s_{k-1}+lNh}-N^1_{s_{k-1}+(l-1)Nh}>0\}
\end{equation*}
for $1\leq k\leq M$ and $j\in\mathbb{N}$.
Then for sufficiently large $j$, $\bar{A}^j_k=\Omega$.
Moreover, for sufficiently large $n$, we have $\inf_{1\leq k \leq M}P[E_k]\geq 3/4$ by (\ref{non-sparse2}) and
\begin{eqnarray}
\sum_{I,J;L(I)\in [s_{k-1},s_k)}\frac{|I\cap J|^2}{|I||J|}&\geq &\sum_{j=1}^{\infty}\sum_{I,J;L(I)\in [s_{k-1},s_k)}\frac{|I\cap J|^21_{\ddot{A}^j_k}}{((3N+j)h)((3N+3j)h)} \nonumber \\
&\geq & \sum_{j=1}^{\infty}\frac{(N+j)^{-2}}{9h^2}1_{\ddot{A}^j_k}\sum_{I,J;L(I)\in [s_{k-1},s_k)}|I\cap J|^21_{E_k}. \nonumber
\end{eqnarray} 
For $r\in \mathbb{N}$ and $u>0$,
we have $x_1^2+\ldots+x_r^2\geq u^2/r$ when $x_i\geq 0 \ ( 1\leq i \leq r)$, $x_1+\ldots + x_r\geq u$.
Hence 
\begin{eqnarray}
\sum_{I,J;L(I)\in [s_{k-1},s_k)}\frac{|I\cap J|^2}{|I||J|}&\geq &\sum_{j=1}^{\infty}\frac{(N+j)^{-2}}{9h^2}\frac{(Nh)^21_{\ddot{A}^j_k}1_{E_k}}{\Delta N^1_k+\Delta N^2_k+1}, \nonumber
\end{eqnarray}
where $\Delta N^i_k=N^i_{s_k}-N^i_{s_{k-1}} \ (1\leq k \leq M, i=1,2)$.
Then we obtain
\begin{equation}\label{a1-est1}
\frac{b_n^{-1}}{T}\sum_{I,J}\frac{|I\cap J|^2}{|I||J|}\geq b_n^{-1}\sum_{j=1}^{\infty}\sum_{k=1}^MX'_{j,k} \quad {\rm a.s., \ where} \quad
X'_{j,k}=\frac{N^2}{9Tj(N+j)^2}\frac{1_{\ddot{A}^j_k}1_{E_k}}{\Delta N^1_k+\Delta N^2_k+1}.
\end{equation}
On the other hand, Theorem \ref{rosenthal} and a similar argument to the proof of Proposition \ref{p-lim} yield
\begin{equation*}
E\bigg[\bigg|\sum_{k=1}^M(X'_{j,k}-E[X'_{j,k}])\bigg|^2\bigg]\leq \frac{Cb_n}{j(N+j)^4}
\end{equation*}
for $j\in\mathbb{N}$ and sufficiently large $n$. Therefore 
\begin{equation}\label{a1-est2}
b_n^{-1}\sum_{j=1}^{\infty}\sum_{k=1}^M(X'_{j,k}-E[X'_{j,k}])\to^p 0
\end{equation}
as $n\to \infty$.

(\ref{a1-est1}) and (\ref{a1-est2}) yield
\begin{equation}\label{a1-est3}
a_1\geq \limsup_{n\to\infty}b_n^{-1}\sum_{j=1}^{\infty}\sum_{k=1}^ME[X'_{j,k}].
\end{equation}

Furthermore, since $\{\Delta N^i_k\}_{1\leq k\leq M, i=1,2, n\geq n_1}$ are tight by the assumption, 
there exists $R'>0$ such that
$\sup_{n\geq n_1,k,i}P\left[\Delta N^i_k>R'\right]<1/8$.
Consequently,
\begin{equation}\label{deltaNest}
\sup_{n\geq n_1,k}P\left[(\Delta N^1_k+\Delta N^2_k+1)^{-1}< (2R'+1)^{-1}\right]<1/4.
\end{equation}

On the other hand, by $[B2\mathchar`-q]$, we obtain
\begin{eqnarray}
P\left[\cup_{j=J+1}^{\infty}\ddot{A}^j_k\right]&\leq &P[(\bar{A}^J_k)^c]\leq 6\sup_{n\geq n_0,t,i}P[N^i_{t+Jh}-N^i_t=0] \leq CJ^{-q} \nonumber 
\end{eqnarray}
for $J\in\mathbb{N}$ and $n\geq n_0$.
Thus, there exists $J$ which does not depend on $n,k$ such that
\begin{equation}\label{tildeA-est}
P\left[\cup_{j=1}^J\ddot{A}^j_k\right]= 1- P\left[\cup_{j=J+1}^{\infty}\ddot{A}^j_k\right]\geq \frac{3}{4}.
\end{equation}
Therefore by (\ref{a1-est3}),(\ref{deltaNest}),(\ref{tildeA-est}) and the estimate $\inf_{1\leq k \leq M}P[E_k]\geq 3/4$, we obtain
\begin{eqnarray}
a_1\geq \frac{N^2}{9TJ(N+J)^2}\limsup_{n\to \infty}b_n^{-1}M\frac{1}{2R'+1}\cdot \frac{1}{4} = \frac{N^2(2R'+1)^{-1}}{36TJ(N+J)^2}\frac{1}{3N}. \nonumber
\end{eqnarray}
\qed

\begin{discuss}
\newpage

\section{Appendix}
\subsection{定義のまとめ}

$[a]$:aを超えない最大整数. 
\begin{eqnarray}
S=\left( 
\begin{array}{cc}
{\rm diag}(| b^1_I |^2) & \left\{ b^1_I\cdot b^2_J\frac{|I\cap J|}{\sqrt{|I|}\sqrt{|J|}}\right\}_{IJ} \\
\left\{ b^1_I\cdot b^2_J\frac{|I\cap J|}{\sqrt{|I|}\sqrt{|J|}}\right\}_{JI} & {\rm diag}(| b^2_J |^2)\\
\end{array} 
\right) \nonumber
\end{eqnarray}
\begin{eqnarray}
D&=&{\rm diag}(\{| b^1_I|\}_I,\{| b^2_J|\}_J), \quad L=\left\{\frac{b^1_I\cdot b^2_J}{| b^1_I| | b^2_J| }\frac{|I\cap J|}{\sqrt{|I||J|}}\right\}_{I,J}, \nonumber \\
\tilde{L}&=&\left(
\begin{array}{ll}
0 & L \\
L^{\star} & 0 \\
\end{array}
\right), \quad Z=\left(\left(\frac{X(I)}{| b^1_I| \sqrt{|I|}}\right)_I^{\star}, \left(\frac{Y(J)}{| b^2_J| \sqrt{|J|}}\right)_J^{\star}\right)^{\star} \nonumber
\end{eqnarray}
$M=(I+\tilde{L})^{-1}$

\begin{equation*}
\rho_t=\frac{b^1_t\cdot b^2_t}{| b^1_t| | b^2_t|}, \quad \bar{\rho}=\sup_{\sigma,0\leq t\leq T}|\rho_t|.
\end{equation*}
\begin{equation*}
G=\left\{ \frac{|I\cap J|}{\sqrt{|I||J|}} \right\}_{IJ},\quad \rho_{I,J}=\frac{b^1_I\cdot b^2_J}{| b^1_I| | b^2_J | }, \quad \tilde{\rho}_n=\sup_{\sigma, I,J;I\cap J \neq \emptyset}|\rho_{I,J}|\vee \bar{\rho}
\end{equation*}

\begin{eqnarray}
H_n&=&-\frac{1}{2}\left(\left(\frac{X(I)}{\sqrt{|I|}}\right)_I^{\star},\left(\frac{Y(J)}{\sqrt{|J|}}\right)_J^{\star}\right)S^{-1}\left(\left(\frac{X(I)}{\sqrt{|I|}}\right)_I^{\star},\left(\frac{Y(J)}{\sqrt{|J|}}\right)_J^{\star}\right)^{\star} -\frac{1}{2}\log \det S \nonumber \\
&=&-\frac{1}{2}Z^{\star}MZ-\log\det D+\frac{1}{2}\log\det M \nonumber \\
&=&-\frac{1}{2}Z^{\star}\left(\sum_{p=0}^{\infty}(-1)^p\tilde{L}^p\right)Z-\log\det D+\frac{1}{2}\sum_{p=1}^{\infty}\frac{(-1)^p}{p}{\rm tr}(\tilde{L}^p) \nonumber
\end{eqnarray}

\begin{equation*}
\hat{H}_n=H_n1_{\{ 1-\tilde{\rho}_n> s_n\} }
\end{equation*}
\begin{eqnarray}
G_{(s,t]}&=&\left\{\frac{|I\cap J|}{\sqrt{|I|}\sqrt{|J|}}\right\}_{I,J;I\cap (s,t] \neq \emptyset}, \quad G_{(s,t]}'=\left\{\frac{|I\cap J|}{\sqrt{|I|}\sqrt{|J|}}\right\}_{I,J;J\cap (s,t] \neq \emptyset} \nonumber
\end{eqnarray}
$B^i_t=| b^i_{t,\ast}|/| b^i_t| \ (i=1,2)$

\begin{equation*}
\theta_{p,k}=\left\{
\begin{array}{ll}
\cup\{I';(GG^{\star})^p_{I^kI'}>0\} & (k\leq l_n) \\
\cup\{J';(G^{\star}G)^p_{J^{k-l_n}J'}>0\} & (k > l_n) \\
\end{array}
\right.
\end{equation*}
\begin{equation*}
R=\max_{0\leq i\leq 4,j,k\geq 0,j+k\leq 3,1\leq p\leq 2}\sup_{\sigma,0\leq s,t\leq T}\left(| \partial_x^j\partial_y^k\partial_{\sigma}^ib^p|\vee \frac{1}{| b^p|}(X_t,Y_s,\sigma)\right)\vee 1.
\end{equation*} 

\begin{eqnarray}
L_p&=&\{\rho_{L(\theta_{p,i}\cup\theta_{p,j+l_n})}G_{I^i,J^j}\}_{i,j}, \quad \tilde{L}_p=\left(
\begin{array}{ll}
0 & L_p \\
L^{\star}_p & 0 \\
\end{array}
\right),\quad \tilde{M}=\sum_{p=0}^{\infty}(-1)^p\tilde{L}_p^p \nonumber
\end{eqnarray}
\begin{equation*}
Z_{k,t}=\left\{
\begin{array}{ll}
X(I^k_t)/(|b^1_{I^k}|\sqrt{|I^k|}) & (k\leq l_n) \\
Y(J_t^{k-l_n})/\left(|b^2_{J^{k-l_n}}|\sqrt{|J_t^{k-l_n}|}\right) & (k>l_n) \\
\end{array}
\right.
\end{equation*}
$<Z>_t=(<Z_k,Z_{k'}>_t)_{k,k'}$とする. このとき
$\tilde{H}^1_n,\tilde{H}^2_n$を
\begin{eqnarray}
\tilde{H}^1_n&=& -\frac{1}{2}Z^{\star}\tilde{M}Z-\log\det D +\frac{1}{2}\sum_{p=1}^{\infty}\frac{(-1)^p}{p}{\rm tr}(\tilde{L}_p^p), \nonumber \\
\tilde{H}^2_n&=&-\frac{1}{2}{\rm tr}(\tilde{M}<Z>_T)-\log\det D +\frac{1}{2}\sum_{p=1}^{\infty}\frac{(-1)^p}{p}{\rm tr}(\tilde{L}_p^p) \nonumber
\end{eqnarray}

\begin{eqnarray}
\Psi^{p,1}(f,g)&=&\Psi^{p,1,n}(f,g)=b_n^{-1}\sum_If(L(I))(GG^{\star})^p_{II}-\int^T_0f(s)g(s)ds, \nonumber \\
\Psi^{p,2}(f,g)&=&\Psi^{p,2,n}(f,g)=b_n^{-1}\sum_Jf(L(J))(G^{\star}G)^p_{JJ}-\int^T_0f(s)g(s)ds \nonumber 
\end{eqnarray}
\begin{eqnarray}
&&b_n^{-1}{\rm tr}((I-\rho^2 G_{(0,t]}G^{\star}_{(0,t]})^{-1}) \to^p \int^t_0a(\rho, s)ds,\nonumber \\
&&b_n^{-1}{\rm tr}((I-\rho^2 (G'_{(0,t]})^{\star}G'_{(0,t]})^{-1}) \to^p \int^t_0c(\rho, s)ds,\nonumber \\
&&b_n^{-1}\sum_I(GG^{\star})^p_{II}1_{\{[0,t]\cap I \neq \emptyset\} } \to^p \int^t_0a_p(s)ds,\nonumber \\
&&b_n^{-1}\sum_J(G^{\star}G)^p_{JJ}1_{\{[0,t]\cap J \neq \emptyset\} } \to^p \int^t_0c_p(s)ds, \nonumber
\end{eqnarray}
\begin{equation*}
A(\rho, \rho_{\ast}, t)=\left(
\begin{array}{ll}
A_{11}(\rho,\rho_{\ast}, t) & A_{12}(\rho,\rho_{\ast}, t) \\
A_{21}(\rho,\rho_{\ast}, t) & A_{22}(\rho,\rho_{\ast}, t) \\
\end{array}
\right)
\end{equation*}
\begin{equation*}
\mathcal{A}(\rho,t)=a(\rho,t)-a_0(\rho,t)=c(\rho,t)-c_0(\rho,t)
\end{equation*}
\begin{eqnarray}
A_{11}(\rho,\rho_{\ast}, t)&=&\sum_{p=0}^{\infty}a_p(t)\rho^{2p}=a_0(t)+\mathcal{A}(\rho,t) \nonumber \\
A_{12}(\rho,\rho_{\ast}, t)&=&A_{21}(\rho,t)=\sum_{p=0}^{\infty}a_{p+1}(t)\rho^{2p+1}\rho_{\ast}=\mathcal{A}(\rho,t)\frac{\rho_{\ast}}{\rho} \nonumber \\
A_{22}(\rho,\rho_{\ast}, t)&=&\sum_{p=0}^{\infty}c_p(t)\rho^{2p}=c_0(t)+\mathcal{A}(\rho,t) \nonumber \\
C(\rho,t)&=&\sum_{p=1}^{\infty}\frac{a_p(t)}{p}\rho^{2p}=\sum_{p=1}^{\infty}\frac{c_p(t)}{p}=2\int^{\rho}_0\frac{\mathcal{A}(\rho,t)}{\rho}d\rho \nonumber
\end{eqnarray}
\begin{eqnarray}
h_t^{\infty}&=& -\frac{1}{2}(B^1_t,B^2_t)A(\rho_t,\rho_{t,\ast},t)\left(
\begin{array}{l}
B^1_t \\
B^2_t \\
\end{array}
\right)-a_0\log | b^1_t| -c_0\log | b^2_t| -\frac{1}{2}C(\rho_t,t)
\nonumber \\
&=& -\frac{1}{2}(B^1_t)^2\sum_{p=0}^{\infty}a_p(t)\rho^{2p}_t+B^1_tB^2_t\sum_{p=0}^{\infty}a_{p+1}(t)\rho^{2p+1}_t\rho_{t,\ast}-\frac{1}{2}(B^2_t)^2\sum_{p=0}^{\infty}c_p(t)\rho^{2p}_t \nonumber \\
&&-a_0\log| b^1_t| - c_0\log | b^2_t | + \frac{1}{2}\sum_{p=1}^{\infty}\frac{a_p}{p}\rho^{2p} \nonumber
\end{eqnarray}
\begin{equation*}
y_t=h^{\infty}_t-h^{\infty}_{t,\ast}
\end{equation*}
\begin{eqnarray}
y_t&=&-\frac{1}{2}(B^1_t)^2(a_0+\mathcal{A})-\frac{1}{2}(B^2_t)^2(c_0+\mathcal{A})+B^1_tB^2_t\mathcal{A}\frac{\rho_{t,\ast}}{\rho_t}+\frac{a_0}{2}+\frac{c_0}{2} \nonumber \\
&&+a_0\log B^1_t+c_0\log B^2_t+\int^{\rho_t}_{\rho_{t,\ast}}\frac{\mathcal{A}}{\rho}d\rho \nonumber \\
&=&-\frac{1}{2}(a_0+\mathcal{A})(B^1_t-B^2_t)^2+a_0+a_0\log B^1_tB^2_t + \int^{\rho_t}_{\rho_{t,\ast}}\frac{\mathcal{A}}{\rho}d\rho \nonumber \\
&&+ \frac{c_0-a_0}{2}(1-(B^2_t)^2+\log (B^2_t)^2)+B^1_tB^2_t\left(\mathcal{A}\left(\frac{\rho_{t,\ast}}{\rho_t}-1\right)-a_0\right) \nonumber \\
&=&-\frac{1}{2}(c_0+\mathcal{A})(B^1_t-B^2_t)^2+c_0+c_0\log B^1_tB^2_t + \int^{\rho_t}_{\rho_{t,\ast}}\frac{\mathcal{A}}{\rho}d\rho \nonumber \\
&&+ \frac{a_0-c_0}{2}(1-(B^1_t)^2+\log (B^1_t)^2)+B^1_tB^2_t\left(\mathcal{A}\left(\frac{\rho_{t,\ast}}{\rho_t}-1\right)-c_0\right). \nonumber 
\end{eqnarray}
\begin{equation*}
\hat{Y}_n=b_n^{-1}(\hat{H}_n(\sigma)-\hat{H}_n(\sigma_{\ast})), \quad \hat{\Gamma}_n(\sigma)=-b_n^{-1}\partial_{\sigma}^2\hat{H}_n(\sigma)[u,u], \quad \Gamma=\int^T_0\partial_{\sigma}^2h^{\infty}_t(\sigma_{\ast})dt
\end{equation*}
$U_n(\sigma)=\{u\in\mathbb{R}^d;\sigma_{\ast}+b_n^{-1/2}u\in\Lambda\}$, $V_n(r,\sigma_{\ast})=\{|u|\geq r\}\cap U_n(\sigma_{\ast})$と定め, $u\in U_n(\sigma)$に対し, $\hat{Z}_n(u,\sigma_{\ast})=\exp(\hat{H}_n(\sigma_{\ast}+b_n^{-1/2}u)-H_n(\sigma_{\ast}))$と定める.

$f:[0,T]\to \mathbb{R}$, $\alpha-$H${\rm \ddot{o}}$lder連続に対し, $\omega_{\alpha}(f)=\sup_{t\neq s}|f_t-f_s|/|t-s|^{\alpha}$と定める.
\subsection{HYのinefficiency}
\begin{eqnarray}
\left\{
\begin{array}{lll}
dX_t&=&\sigma_1dW^1_t \\
dY_t&=&\sigma_3dW^1_t+\sigma_2dW^2_t, \\
\end{array}
\right. \nonumber
\end{eqnarray}
$\{I\}=\{J\}$ : deterministic, $l_n=\#\{I\}$とし, 
\begin{equation*}
{\rm HY}=\sum_IX(I)Y(I), \quad M=l_n^{-1}\sum_I\frac{X(I)Y(I)}{|I|}T
\end{equation*}
とおくと, $E[HY]=E[M]=\sigma_1\sigma_3 T$であり, 
\begin{eqnarray}
{\rm Var(HY)}&=&\sum_IE[(X(I)Y(I)-\sigma_1\sigma_3|I|)^2]=\sum_I(E[X(I)^2Y(I)^2]-\sigma_1^2\sigma_3^2|I|^2) \nonumber \\
&=&\sum_I(E[\sigma_1^2W^1(I)^2(\sigma_3W^1(I)+\sigma_2W^2(I))^2]-\sigma_1^2\sigma_3^2|I|^2) \nonumber \\
&=&\sum_I(3\sigma_1^2\sigma_3^2|I|^2+\sigma_1^2\sigma_2^2|I|^2-\sigma_1^2\sigma_3^2|I|^2) \nonumber \\
&=&(2\sigma_1^2\sigma_3^2+\sigma_1^2\sigma_2^2)\sum_I|I|^2. \nonumber
\end{eqnarray}
一方, 
\begin{eqnarray}
{\rm Var(M)}&=&l_n^{-2}T^2E\left[\left(\frac{X(I)Y(I)}{|I|}-\sigma_1\sigma_3\right)^2\right]=l_n^{-2}T^2\sum_I\left(E\left[\frac{X(I)^2Y(I)^2}{|I|^2}\right]-\sigma_1^2\sigma_3^2\right) \nonumber \\
&=&T^2l_n^{-1}(2\sigma_1^2\sigma_3^2+\sigma_1^2\sigma_2^2). \nonumber
\end{eqnarray}
$l_n=3n$, 
\begin{eqnarray}
|I^i|=\left\{
\begin{array}{ll}
T/2n & (i\leq n) \\
T/4n & (i>n) \\
\end{array}
\right. \nonumber
\end{eqnarray}
とおくと
\begin{equation*}
\sum_I|I|^2=\left(\frac{T^2}{4n^2}n+\frac{T^2}{16n^2}2n\right)=\frac{3T^2}{8n}>\frac{T^2}{3n}.
\end{equation*}
よって${\rm Var(HY)>Var(M)}$.

\subsection{その他}
\begin{proposition}\label{sup-lambda}
対称行列$P$の固有値を$\lambda_1,\lambda_2, \ldots, \lambda_n$と置くと, $\parallel P\parallel=\sup_i|\lambda_i|$. 
\end{proposition}
\begin{proof}
任意の$i \in \{1,2,\ldots, n\}$に対し, 
ある$x\in \mathbb{R}^n$があって, 
$| x | =1$かつ$Px=\lambda_ix$となるので, $\parallel P \parallel\geq \sup_i|\lambda_i|$は明らか. 

$\parallel P \parallel> \sup_i|\lambda_i|$とすると, ある$x\in \mathbb{R}^n$があって, $| x | =1$かつ$| Px|>\sup_i|\lambda_i|$とできる. 一方, ある直行行列$U$に対し, $UPU^{\star}={\rm diag}(\lambda_i)$と書け, $y=Ux$とすると, $U^{\star}y=x,| y | = 1$かつ
\begin{eqnarray}
| UPU^{\star}y| &=& | UPx| =| Px| > \sup_i|\lambda_i|, \nonumber \\
| UPU^{\star}y| &=& \sqrt{\sum_i(\lambda_iy_i)^2} \leq \sup_i|\lambda_i|| y | = \sup_i|\lambda_i|, \nonumber
\end{eqnarray}
となり矛盾. 
\end{proof}

\begin{eqnarray}
\int^t_s\sum_{p=0}^{\infty}a_p(u)\rho^{2p}du&=&P \lim_{n\to \infty}b_n^{-1}\sum_i\sum_{p=1}^{\infty}(\lambda_i)^p\rho^{2p} \nonumber \\
&=&P \lim_{n\to \infty}b_n^{-1}\sum_i\frac{1}{1-\lambda_i\rho^2} \nonumber \\
&=&P \lim_{n\to \infty}b_n^{-1}\frac{1}{2\pi\sqrt{-1}}\int\frac{1}{1-z\rho^2}\left(\sum_i\frac{1}{z-\lambda_i}\right)dz \nonumber \\	
&=&P \lim_{n\to \infty}b_n^{-1}\frac{1}{2\pi\sqrt{-1}}\int\frac{1}{1-z\rho^2}\frac{{\rm tr}((I-z^{-1}G_{(s,t]}G^{\star}_{(s,t]})^{-1})}{z}dz \nonumber \\	
&=&\frac{1}{2\pi\sqrt{-1}}\int\frac{1}{1-z\rho^2}\int^t_s\frac{1}{z}a(1/\sqrt{z},u)dudz. \nonumber 
\end{eqnarray}
ゆえに
\begin{equation*}
\sum_{p=0}^{\infty}a_p(t)\rho^{2p}=\frac{1}{2\pi\sqrt{-1}}\int\frac{1}{1-z\rho^2}\frac{1}{z}a(1/\sqrt{z},t)dz.
\end{equation*}
同様に
\begin{eqnarray}
\sum_{p=0}^{\infty}a_{p+1}(t)\rho^{2p+1}\rho_{\ast}&=&\frac{1}{2\pi\sqrt{-1}}\int\frac{z\rho\rho_{\ast}}{1-z\rho^2}\frac{1}{z}a(1/\sqrt{z},t)dz, \nonumber \\
\sum_{p=0}^{\infty}\frac{a_p(t)}{p}\rho^{2p}&=&-\frac{1}{2\pi\sqrt{-1}}\int\log(1-z\rho^2)\frac{1}{z}a(1/\sqrt{z},t)dz. \nonumber
\end{eqnarray}
よって
\begin{eqnarray}
y_t&=&\frac{1}{2\pi\sqrt{-1}}\int\frac{a(1/\sqrt{z},t)}{z}\bigg\{-\frac{1}{2(1-z\rho^2)}\left\{(B^1_t)^2-2z\rho\rho_{\ast}B^1_tB^2_t+(B^2_t)^2\right\} \nonumber \\
&&+1+\frac{1}{2}\log\left((B^1_t)^2(B^2_t)^2\frac{(1-z\rho_{\ast}^2)}{1-z\rho^2}\right)\bigg\}dz+\frac{c_0-a_0}{2}\left(1-(B^2_t)^2+\log(B^2_t)^2\right). \nonumber
\end{eqnarray}
\end{discuss}

\bibliographystyle{plain}
\bibliography{bunken-hofyos08.bib}




%
\end{document}